\documentclass[a4paper,10pt]{article}
\usepackage[utf8]{inputenc}

\usepackage{amsfonts,latexsym,amstext}
\usepackage{amsmath}
\usepackage[active]{srcltx}
\usepackage[active]{srcltx}
\usepackage{color}
\usepackage{amsmath,amsfonts,calrsfs,amssymb,color,enumerate,mathrsfs }
\usepackage{amsthm}
\usepackage{amssymb} 
\usepackage{cancel}

\newtheorem{theorem}{Theorem}[section]
\newtheorem{lemma}[theorem]{Lemma}
\newtheorem{proposition}[theorem]{Proposition}
\newtheorem{definition}[theorem]{Definition}

\newtheorem{hypothesis}[theorem]{Hypothesis}

\newtheorem{remark}[theorem]{Remark}
\newtheorem{corollary}[theorem]{Corollary}

\numberwithin{equation}{section}

\def\sqr#1#2{{\vcenter{\vbox{\hrule height .#2pt \hbox{\vrule
 width .#2pt height#1pt \kern#1pt \vrule
width .#2pt} \hrule height .#2pt}}}}

\def\ds{\begin{displaystyle}}
\def\eds{\end{displaystyle}}

\def\<{\langle }
\def\>{\rangle }

\def\T{\mathcal T}

\def\R{\mathbb R}
\def\N{\mathbb N}

\def\E{\mathbb E}
\def\P{\mathbb P}

\def\cala{{\cal A}}

\def\calf{{\cal F}}
\def\calg{{\cal G}}

\def\calp{{\cal P}}

\def\calr{{\cal R}}

\def\Etau{\E^{\calf_\tau}}

\allowdisplaybreaks

\title{{ A BSDEs approach to pathwise uniqueness
{ for stochastic evolution equations } }}

\usepackage{authblk}

\title{{A BSDEs approach to pathwise uniqueness for stochastic evolution equations} }

\usepackage{authblk}

\author[1]{Davide ADDONA\thanks{davide.addona@unipr.it}}
\author[2]{Federica MASIERO\thanks{federica.masiero@unimib.it}} 
\author[3]{Enrico PRIOLA\thanks{enrico.priola@unipv.it}} 

\affil[1]{Dipartimento di Scienza Matematiche, Fisiche e Infomatiche, Universit\`a di Parma, Parma, Italy}

\affil[2]{Dipartimento di Matematica e Applicazioni, Universit\`a di Milano Bicocca, Milano, Italy}

\affil[3]{Dipartimento di Matematica, Universit\`a di Pavia, Pavia, Italy}

\date{}  

\begin{document}

\maketitle

\begin{abstract}
We prove strong well-posedness for a class of  stochastic  evolution equations in  Hilbert spaces $H$ when the drift term is   H\"older continuous.   This class includes examples of semilinear  
stochastic beam equations  which describe elastic systems with structural damping
and  semilinear stochastic 3D heat equations.   
In the deterministic case, there are
examples of non-uniqueness in our framework. Strong (or pathwise) uniqueness is restored by means of a suitable additive Wiener noise. 
 The  proof of uniqueness  relies  on   the  study of 
 related systems of infinite dimensional forward-backward  SDEs  (FBSDEs). This  is  a different  approach with respect to the well-known method based on the It\^o formula and   the associated Kolmogorov equation (the so-called Zvonkin transformation or It\^o-Tanaka trick). 
 We deal with   approximating FBSDEs in which the linear  part generates a  group of bounded linear operators in $H$; such approximations depend on the type of SPDEs we are  considering. We also prove Lipschitz dependence of solutions from their  initial conditions. 
\end{abstract}

\smallskip


\smallskip


\section{Introduction}
\subsection{Problem and main result}
In this paper we consider the problem of strong well-posedness for a class of  stochastic partial differential equations (SPDEs) when the drift term is only H\"older continuous. 

In a real and separable Hilbert space $H$ we consider a stochastic evolution equation of the form
\begin{align}
\label{nonlinear_stochastic_eq_intro}
\left\lbrace
\begin{array}{ll}dX_\tau=AX_\tau d\tau+G\widetilde C(\tau,X_\tau)d\tau+GdW_\tau, & \tau\in[0,T],\vspace{1mm} \\ 
X_0=x\in H, \end{array}\right.
\end{align}
where $A:D(A)\subset H\rightarrow H$ is the infinitesimal generator of a strongly continuous semigroup $(e^{tA})_{t\geq0}$, the operator 
 $G:U\rightarrow H$ is a bounded linear operator defined on another real and separable Hilbert space $U$ and  $W$ is a cylindrical Wiener process on $U$ (cf. Section \ref{sub:abstract_equation}). Moreover,  $\widetilde C:[0,T]\times H\rightarrow U$ is bounded and continuous and $\widetilde C(t,\cdot)$ is $\beta$-H\"older continuous, uniformly in  $t\in[0,T]$, $\beta \in (0,1)$.

In our framework   there are examples of non uniqueness of solutions in the deterministic case (i.e., when $W =0$ in \eqref{nonlinear_stochastic_eq_intro}) (see \cite{dapr-fl} and \cite{MasPri17}). Hence our main result on strong uniqueness, see Theorem \ref{teo:abs_uni1},  is due to  the presence of the  Wiener noise (regularization by noise).  
 Clearly, such theorem holds also in the  
 Lipschitz case $\beta =1$ (in this case the result does not depend on  $W$). Moreover, the  boundedness of $\tilde C$ can be relaxed  by a localization procedure (see Remark \ref{gen}). 
  Our pathwise result also implies the existence of a  strong mild solution (cf. Remark \ref{strong}).

Examples of singular  semilinear  SPDEs of the form \eqref{nonlinear_stochastic_eq_intro} we can treat are 
stochastic damped equations  which describe elastic systems with structural damping
and semilinear stochastic 3D heat equations (see \eqref{abst_sto_damped_WE_1a}, \eqref{ww42} and  Section \ref{sec:applications} which is about applications).

In the literature the problem of  regularization by noise  for stochastic evolution equations \eqref{nonlinear_stochastic_eq_intro}   of parabolic type has been widely studied, see \cite{dapr-fl}, \cite{DPFPR13}, \cite{DPFPR15}, \cite{DPFRV16}, \cite{GP93}, \cite{Zv74} and the references therein. 
Moreover, we mention 
\cite{CeDaFl13} for SDEs in Banach space,  
and \cite{MasPri17} {and \cite{MasPri23}}, where the semilinear stochastic wave equation is studied. 
 Usually in such papers pathwise uniqueness is obtained by using the  It\^o formula after solving  a Kolmogorov equation. In finite dimension this is the so-called Zvonkin transformation or the It\^o-Tanaka trick (see the seminal paper \cite{Ver},  the recent monograph \cite{Fl} and the references therein). 
 
 Note that establishing the It\^o formula in infinite dimensions  is a delicate issue when the  noise is cylindrical (cf. \cite{dapr-fl}, \cite{DPFPR13}, \cite{DPFPR15} and the references therein). 
We  replace the previous  approach with a method based on  forward-backward  SDEs  (FBSDEs).
This does not require proving the It\^o formula, it  works  for hyperbolic and parabolic SPDEs and 
 extends  the method introduced in \cite{MasPri17} for singular semilinear stochastic wave equations.

The techniques of \cite{MasPri17} {and \cite{MasPri23}} work only when the operator $A$ appearing in equation \eqref{nonlinear_stochastic_eq_intro} is the generator of a group of bounded linear operators. Here  we introduce 
approximating FBSDEs in which the linear  parts $A_n$ in the backward equation  generate a  group of bounded linear operators in $H$  (cf. equation \eqref{FBSDE_approssimato_n_intro} below); 
  such approximations depend on the type of SPDEs we are  considering. 
 Let us explain  two examples of 
singular SPDEs of the form \eqref{nonlinear_stochastic_eq_intro}
 we can consider.

{The first class of equations we can treat is stochastic semilinear damped Eulero-Bernoulli beam equations of the form
\begin{align} 
\label{abst_sto_damped_WE_1a}
\left\{  
\begin{array}{ll}
\displaystyle \frac{\partial ^2y}{\partial t^2}(t)=-\Lambda y(t)-\rho \Lambda^\alpha\frac{\partial y}{\partial t}(t)+  \Lambda^{-\gamma}\tilde C\left(t,y(t), \frac{\partial y}{\partial t}(t)\right)+\Lambda^{-\gamma}\dot{W}_t, & t\in(0,T], \vspace{1mm}\\
y(0)=y_0, \vspace{1mm} \\
\displaystyle \frac{\partial y}{\partial t}(0)=y_1,
\end{array}
\right.
\end{align}
with $\rho,\gamma>0$ and $\alpha\in[0,1)$.
Here, $\Lambda:D(\Lambda)\subset U\rightarrow U$ is a {positive self-adjoint operator} on a separable Hilbert space $U$, there exists { $\Lambda^{-\gamma}$ and it is a trace class operator from $U$ into $U$.}   
Such equations describe elastic systems with structural damping. For physical motivations we refer to  the seminal paper \cite{chen-russ} (see also  the references therein) for the deterministic equation, and to \cite{brez1,brez2,chow} for the stochastic counterpart.}

The second class of equations we can treat is 3D semilinear stochastic heat equations like 
 \begin{align} 
\label{ww42}
\left\lbrace
\begin{array}{ll}dX^{x}_t= \Delta X^{x}_t dt + (-\Delta)^{-\gamma/2} \widetilde C(X^{x}_t)dt + (-\Delta)^{-\gamma/2} dW_t, & t \in[0,T], \vspace{1mm} \\ 
X_0^{x}=x\in H, 
 \end{array}\right.
\end{align} 
 where we are dealing with   the Laplace operator in $H=U=L^2([0,\pi]^d)$ with periodic boundary conditions and  $G=(-\Delta)^{-\gamma/2}$ with $\gamma\geq 0$. Such equations are  also considered 
  in \cite[Example 6.1]{dapr-fl}   with $\beta$-H\"older continuous drifts $\widetilde C$ (even with $(-\Delta)^{-\gamma/2} \widetilde C(X^{x}_t)$ replaced by the more general term $ \widetilde C(X^{x}_t)$). However as the authors explain in \cite{dapr-fl} for $\beta$-H\"older continuous drift terms $\widetilde C$  they can only prove uniqueness  in dimensions $1$ and $2$.
 On the other hand, in Section \ref{sub:heat_equation} we  treat also the dimension $d=3$. 

 As in \cite{MasPri17} our method allows   to  establish     Lipschitz dependence of solutions from their initial conditions (cf. Theorem \ref{teo:abs_uni1}):
 \begin{equation}   
\label{lipp2}
 \sup_{t \in [0,T]} \E  [   \vert  X_t^{x_1}-X_t^{x_2}\vert_H^2] \leq c_T \vert x_1-x_2\vert_H^2, 
\end{equation} 
where $X^{x_1}$ and $X^{x_2}$ denote  the  weak mild solutions to \eqref{nonlinear_stochastic_eq_intro} starting at $x_1$ and $x_2 \in H$, respectively.
 Estimates like \eqref{lipp2} have not been proved before in  papers
 on regularization by noise for  stochastic evolution equations 
  of parabolic type  (cf. \cite{dapr-fl}, \cite{DPFPR13}, \cite{DPFPR15}, \cite{DPFRV16}, 
   \cite{CeDaFl13}). 
  
Note that when  $\widetilde C$ is only continuous and bounded even for parabolic SPDEs (with $A$ which verifies  Hypothesis \ref{ip_finite}) pathwise uniqueness for \eqref{nonlinear_stochastic_eq_intro} { for any initial $x \in H$} is still an open problem   (cf. \cite{DPFPR13}, \cite{DPFRV16} and the references therein).

\subsection{Strategy of the proof}
First notice the particular structure of equation \eqref{nonlinear_stochastic_eq_intro}, which in the BSDE literature is referred to as {\it structure condition}: the drift belongs to the image of the diffusion operator $G$.
 Under  the basic Hypothesis \ref{hyp_1} the existence of a (unique in law) weak mild solution given by
\begin{align} 
\label{intro_mild_solution}
X_t=e^{tA}x+\int_0^te^{(t-s)A}G\widetilde C(s,X_s)ds+\int_0^te^{(t-s)A}GdW_s, \quad \P{\textup{-a.s.}}, 
\end{align}
$t\in[0,T]$,  directly follows by an infinite dimensional version of the Girsanov theorem, see e.g. \cite[Proposition 7.1]{Ondre04} and \cite[Section 10.3]{DPsecond}. 

In order to prove pathwise uniqueness of solutions to equation \eqref{nonlinear_stochastic_eq_intro}, we complement equation \eqref{nonlinear_stochastic_eq_intro} with a family of BSDE, which gives the following family of systems of FBSDEs
\begin{align}
\label{FBSDE_approssimato_n_intro}  
\left\{
\begin{array}{ll}
d X_\tau^{t,x}=AX_\tau^{t,x}d\tau+G\widetilde C(\tau,X_\tau^{t,x})d\tau+GdW_\tau, & \tau\in[t,\mathcal T], \vspace{1mm}\\
X_t^{t,x}=x, & \tau\in[0,t], \vspace{1mm}\\
-dY_\tau^{t,x,n}=-A_nY_\tau^{t,x,n}d\tau+G\widetilde C(\tau,X_\tau^{t,x})d\tau-Z_\tau^{t,x,n}dW_\tau, & \tau\in[t,\mathcal T], \vspace{1mm} \\
Y_{\mathcal T}^{t,x,n}=0,
\end{array}
\right.  
\end{align}
with $\mathcal T\in(0,T]$ and $n\in\N$. Here, $(A_n)_{n\in\N}$ is a sequence of linear operators on $H$ and each $A_n$ generates a group of bounded operators $(e^{tA_n})_{t\in\R}$ which pointwise approximates $(e^{tA})_{t\geq0}$. We stress that the idea of associating to \eqref{nonlinear_stochastic_eq_intro} a BSDE has been exploited in \cite{MasPri17} {and \cite{MasPri23}}, but in that case the assumption that $A$ is the generator of a strongly continuous group $(e^{tA})_{t\in\R}$ is fundamental. In the present paper a crucial point is to consider a suitable sequence of approximating BSDEs, where if $A$ is not the generator of a group  of operators, each $A_n$ is. 

By an infinite dimensional version of Girsanov Theorem, the process
\begin{align*}
\widehat W_\tau:= W_\tau+\int_0^\tau \widetilde C(s,X_s^{t,x})ds,
\end{align*}
is a cylindrical Wiener process on $U$ up to time $T$. Under this transformation, system \eqref{FBSDE_approssimato_n_intro} reads as
\begin{align}
\label{intro_syst_FBSDE}
\left\{
\begin{array}{ll}
d X_\tau^{t,x}=AX_\tau^{t,x}d\tau+Gd\widehat W_\tau, & \tau\in[t,\mathcal T], \vspace{1mm}\\
X_t^{t,x}=x, & \tau\in[0,t], \vspace{1mm}\\
-dY_\tau^{t,x,n}=-A_nY_\tau^{t,x,n}d\tau+G\widetilde C(\tau,X_\tau^{t,x})d\tau
+Z_\tau^{t,x,n}\widetilde C(\tau,X_\tau^{t,x})d\tau-Z_\tau^{t,x,n}d\widehat W_\tau, & \tau\in[t,\mathcal T], \vspace{1mm} \\
Y_{\mathcal T}^{t,x,n}=0,
\end{array}
\right.  
\end{align}
and this system has a unique solution $(X^{t,x},Y^{t,x,n},Z^{t,x,n})$ where $(Y^{t,x,n},Z^{t,x,n})$ is a pair of predictable processes belonging to $L^2(\Omega;C([0,\mathcal T];H))\times  L^2(\Omega\times [0,\mathcal T];L_2(U;H))$. Further, $\mathbb P$-a.s.
\begin{align} \label{wq}
Y_\tau^{t,x,n}=e^{-(\T-\tau)A_n}u_n^{\T}(\tau, \Xi_{\tau}^{t,x}), \quad Z_\tau^{t,x,n}=e^{-(\T-\tau)}\nabla^Gu_n^{\T}(\tau, \Xi_\tau^{t,x}), \ \ \textup{a.e.} \; \tau\in[t,\T],
\end{align}
where $\nabla ^G$ denotes the G\^ateaux derivative along the direction of $G$ and $u_n^{\mathcal T}$ is the unique solution to
\begin{align}
\label{intro_int_det_eq}
v(t,x)
=   \int_t^{\mathcal T}R_{s-t}\left[e^{(\mathcal T-s)A_0}G\widetilde C(s,\cdot)\right](x)ds
+ \int_t^{\mathcal T}R_{s-t}\left[\nabla^Gv(s,\cdot)\widetilde C(s,\cdot)\right](x)ds.
\end{align}
Here, $(R_t)_{t\geq0}$ is the Ornstein-Uhlenbeck semigroup defined by $R_t[\Phi](x):=\mathbb E[\Phi(\Xi_t^{0,x})]$, for any $\Phi\in B_b(H;H)$, any $t\geq0$ and any $x\in H$ (see  
  Section \ref{subsec:abs_OUsemigroup_0}), and $\Xi_\tau^{0,x}$ is the Ornstein-Uhlenbeck process defined by means of  \eqref{nonlinear_stochastic_eq_intro} when $\widetilde C=0$. In mild formulation, the process $Y^{t,x,n}$ satisfies
 \begin{align}
 Y^{t,x,n}_\tau
 = & \int_\tau^{\mathcal T}e^{-(s-\tau)A_n}G\widetilde C(s,X_s^{t,x})ds
 + \int_\tau^{\mathcal T}e^{-(s-\tau)A_n}Z_s^{t,x,n}\widetilde C(s,X_s^{t,x})ds \notag \\
& -  \int_\tau^{\mathcal T}e^{-(s-\tau)A_n}Z_s^{t,x,n}d\widehat W_s \notag \\
= & \int_\tau^{\mathcal T}e^{-(s-\tau)A_n}G\widetilde C(s,X_s^{t,x})ds
-\int_\tau^{\mathcal T}e^{-(s-\tau)A_n}Z_s^{t,x,n}dW_s,  \quad \mathbb P{\textup {-a.s.}},
\label{intro_mild_form_BSDE}
 \end{align}
 for every $\tau\in[0,\mathcal T]$. By setting $t=\tau=0$, by applying the operator $e^{\mathcal TA_n}$ to the first and the last side of \eqref{intro_mild_form_BSDE} and by taking \eqref{wq} into account, we get
 \begin{align*}
  \int_\tau^{\mathcal T}e^{(\mathcal T-s)A_n}G\widetilde C(s,X_s^{t,x})ds=u_n^{\mathcal T}(0,x)+\int_0^{\mathcal T} e^{(\mathcal T-s)A_n}\nabla ^G u_n^{\mathcal T}(s,X_s^{0,x})dWs, \quad \P\textup{-a.s.},
 \end{align*}
 for every $\mathcal T\in[0,T]$. By replacing this formula in \eqref{intro_mild_solution} it follows that
 \begin{align}
 \notag
 X_t
 = & e^{tA}x+\int_0^t\left(e^{(t-s)A}-e^{(t-s)A_n}\right)G\widetilde C(s,X_s)ds+\int_0^te^{(t-s)A}GdW_s \\
 & u_n^{t}(0,x)+\int_0^{t} e^{(t-s)A_n}\nabla ^G u_n^{t}(s,X_s^{0,x})dWs, \quad \P{\textup{-a.s.}}, 
 \label{intro_mild_form_2}
 \end{align}
for every $t\in[0,T]$. From there, to obtain \eqref{lipp2} we need that the first integral in the right-hand side of \eqref{intro_mild_form_2} converges to $0$ as $n$ goes to $+\infty$, and that the function $u_n^t$ is smooth enough in order to get Lipschitz estimates of the addends which involve $u_n^t$ in \eqref{intro_mild_form_2}. Both these things are a consequence of Hypothesis \ref{hyp_2_1}, and so estimate \eqref{lipp2} follows.

\subsection{Plan of the paper}
The paper is organized as follows. In Section \ref{sub:abstract_equation} we state the main assumptions on the coefficients of equation \eqref{nonlinear_stochastic_eq_intro}, under which there exists a (unique in law) weak mild solution $(X_t)$ to \eqref{nonlinear_stochastic_eq_intro} (see Hypothesis \ref{hyp_1}). Further, we provide sufficient conditions which ensure existence and uniqueness of a smooth solution $u_n^{\mathcal T}$ to the integral equation \eqref{intro_int_det_eq} for any $n\in\N$ and any $\mathcal T\in(0,T]$ (see Hypothesis \ref{hyp_2_1}). We stress that estimates on the Hilbert-Schmidt norm of $\nabla\nabla^Gu^{\T}_n$, together with the family of systems of FBSDEs \eqref{FBSDE_approssimato_n_intro}, are one of the main tool which we need to prove our result. Finally, we prove a generalized Gronwall Lemma which will be applied in the proof of Theorem \ref{teo:abs_uni1}.

In Section \ref{sub:FBSDE_app} we consider the family of systems of FBSDEs \eqref{intro_mild_form_BSDE} and we show that, under our assumptions, for any $n\in\N$,  any $\mathcal T\in(0,T]$, any $x\in H$ and any $t\in[0,\mathcal T)$ there exists a unique mild solution $(X^{t,x}, Y^{t,x,n}, Z^{t,x,n})$ to \eqref{intro_syst_FBSDE} which satisfies \eqref{wq}.

In Section \ref{sec:path_uniq} we prove the main result of the paper. At first we show that representation \eqref{intro_mild_form_2} holds true for the weak mild solution $(X_t)$. Finally, by means of this representation and of Hypothesis \ref{hyp_2_1}, we prove Theorem \ref{teo:abs_uni1}, which states that there exists a positive constant $c_T$ such that for every $x_1,x_2\in H$  estimate \eqref{lipp2} is satisfied.

Section \ref{sec:OU_smgr_reg} is devoted to provide sufficient conditions on $A,G$ and on $\widetilde C$ which ensure that Hypothesis \ref{hyp_2_1} are verified. We split this section into three parts. In the former we show the existence of a unique smooth solution $u_n^{\mathcal T}$ to equation \eqref{intro_int_det_eq} for any $n\in\N$ and any $\mathcal T\in(0,T]$, while in the second and in the latter we prove the crucial estimate on the Hilbert-Schmidt norm of $\nabla\nabla ^Gu_n^\mathcal T$ when $(A_n)_{n\in\N}$ are the Yosida approximants of $A$ and when $(A_n)_{n\in\N}$ are finite dimensional approximations of $A$, respectively.

Finally, Section \ref{sec:applications} concerns with two concrete models to which our abstract results apply:
\begin{enumerate}[\rm (i)]
\item a stochastic damped {Euler-Bernoulli beam} equation
\begin{align*}
\left\{
\begin{array}{ll}
\displaystyle \frac{\partial ^2y}{\partial t^2}(t)=-\Lambda y(t)-\rho \Lambda^\alpha\frac{\partial y}{\partial t}(t)+ {\Lambda^{-\gamma}}c\left(t,y(t), \frac{\partial y}{\partial t}(t)\right)+{\Lambda^{-\gamma}}\dot{W}(t), & t\in(0,T], \vspace{1mm}\\
y(0)=y_0, \vspace{1mm} \\
\displaystyle \frac{\partial y}{\partial t}(0)=y_1,
\end{array}
\right.
\end{align*}
where $\Lambda:D(\Lambda)\subset U\rightarrow U$ is a positive self-adjoint operator such that $\Lambda^{-\gamma}$ is trace class;
\item a semilinear stochastic heat equation
\begin{align*}
\left\lbrace
\begin{array}{ll}dX_t= \Delta X_t dt + (-\Delta)^{-\gamma/2} \widetilde C(X_t)dt + (-\Delta)^{-\gamma/2} dW_t, & t \in[0,T], \vspace{1mm} \\ 
X_0=x\in H, 
 \end{array}\right.
\end{align*}
where $\Delta$ is the Laplacian operator on $L^2([0,\pi]^d)$ with periodic boundary conditions. 
\end{enumerate}
We note that
\begin{enumerate}
\item in the case of beam equation, the result is completely new;
\item in the case of the heat equation, the result extends the one from \cite{dapr-fl} to the dimension $d=3$. 
\end{enumerate}
Appendix contains minimal energy estimates  for the beam equation which are necessary for our approach. To the best of our knowledge, these estimates are new and of independent interest. They are inspired by the spectral methods used in  \cite{av-las,lasi-trig,trig}.

 \subsection{Notations} 
\label{sub:notazioni}
Throughout the paper we denote by $H,K, U$ real and  separable Hilbert spaces. 

$L(H; K)$ denotes the Banach space of all bounded and linear operators from $H$ into $K$ endowed with the operator norm; we set $L(H) = L(H; H)$. Moreover 
 $L_2(H;K) \subset L(H;K)$ is the Hilbert space of all  Hilbert-Schmidt operators, i.e., the space of operators $T\in L(H;K)$ such that
$$ 
\sum_{k \ge 1} |Te_k|^2_K<+\infty,
$$
where $(e_k)$ is an orthonormal basis of $H$. 
We introduce the norm $\|\cdot\|_{L_2(H;K)}$ 
as
\begin{align*}
\| T\|_{L_2(H,K)}^2 = \sum_{k \ge 1} |Te_k|^2_K, \quad T\in L_2(H;K).
\end{align*}
 We also set $L_2(H) = L_2(H;H)$.
 
If $F:H\rightarrow K$ is G\^ateaux differentiable at $x\in H$, we denote by $\nabla F(x)\in L(H;K)$ its G\^ateaux derivative at $x$ and by $\nabla_k F(x)$ its directional derivative in the direction of $k\in H$, i.e.,
\begin{align*}
\nabla_k F(x):=\lim_{t\rightarrow0}\frac{F(x+tk)-F(x)}{t}, \quad x,k\in H.
\end{align*} 
Let $G:U\rightarrow H$ be a linear bounded operator. We are interested into differentiating along $G$-directions, i.e., directions $k\in H$ such that $k=Ga$ for some $a\in U$. We introduce the notation $\nabla^GF(x):=\nabla F(x)G\in L(U;H)$ and 
\begin{align}
\nabla^G_a F(x)G:=\nabla_{Ga}F(x)=\lim_{t\rightarrow0}\frac{F(x+tGa)-F(x)}{t}, \quad x\in H, \quad a\in U.
\end{align}
\section{The abstract equation}
\label{sub:abstract_equation}
 
We fix $T>0$, and we will  consider the following semilinear stochastic differential equation in the real separable Hilbert space $H$:
\begin{align}
\label{abs_1nonlinear_stochastic_eq}
\left\lbrace
\begin{array}{ll}dX^{t,x}_\tau=AX^{t,x}_\tau d\tau+ G \widetilde C(\tau,X^{t,x}_\tau)d\tau+GdW_\tau, & \tau\in[t,T],\, 0\leq t\leq\tau, \vspace{1mm}\\
X_t^{t,x}=x\in H, \end{array}\right.
\end{align}
where  $W=(W_t)$ is a cylindrical Wiener process on another real separable Hilbert space $U$.  Recall that    
  $ W  $  is formally given by ``$W_t = \sum_{n \ge 1}
  W^{(n)}_t e_n$'' where  $(W^{(n)})_{n \ge 1}   $   are independent real Wiener processes and   $(e_n)$ is a basis of $U$; 
  $W$ defines a Wiener process on any Hilbert space $U_1 \supset U$  with Hilbert-Schmidt embedding (cf. Section 4.1.2 in \cite{DPsecond} and Section 2 in \cite{fute}).    Moreover $A$, $\widetilde C$ and $G$ satisfy the following assumptions.
\begin{hypothesis}
\label{hyp_1} 
\begin{enumerate}[(i)]
\item $A:D(A)\subset H\rightarrow H$ is the infinitesimal generator of a strongly continuous semigroup $(e^{tA})_{t\geq0}$.
\item $G:U\rightarrow H$ is a bounded linear operator.
\item   There exists $\alpha \in (0,1/2)$ such that  
\begin{equation}
\int_{0}^{T}t^{-2\alpha }\|e^{tA} \, G\|^{2}_{L_2(U;H)}\;dt<+ \infty.
\label{e5.17}
\end{equation}
\item 
The function $\widetilde C:[0,T]\times H\rightarrow U$ is bounded and continuous and, moreover,  there exists a positive constant $K$ and  $\beta\in ( 0,1)$ such that
\begin{align} \label{ey}
|\widetilde C(t,x)-\widetilde C (t,y)|_U\leq K|x-y|^\beta_H, \quad x,y\in H, \ t\in[0,T].
\end{align}
\end{enumerate}
\end{hypothesis}
{ Note that condition  \eqref{e5.17}  implies that 
 for any $t>0$ the linear bounded operator
\begin{align}
\label{abstr_conv_stocastica}
{Q_t:=\int_0^t e^{sA}G G^* e^{sA^*}ds:H\rightarrow H } \;\; \text{is a trace class operator.}  
 \end{align}
Condition  \eqref{e5.17} is implicitly assumed also   in \cite{dapr-fl}. It ensures that solutions have continuous paths with values in $H$ (cf. Section 5.3 in  \cite{DPsecond}).
}

Clearly,  $C := G \widetilde  C :[0,T]\times H\rightarrow H $
and $\| C\|_\infty:=\sup_{t\in[0,T], x \in H}| C(t,\cdot)|_{H}<+\infty.$

{We underline that in Hypothesis \ref{hyp_1} it is included the case $U=H$ and $G=I$, if the stochastic convolution $Q_t,\,t>0$, is a trace class operator and \eqref{e5.17} holds. 

\vskip 1mm
 Recall that a  (weak) mild solution  to
(\ref{abs_1nonlinear_stochastic_eq}) is given by $(
\Omega,$ $ {\mathcal F},
 ({ \mathcal F}_{\tau}), \P,$ $ W, X) $, where $(
\Omega, {\mathcal F},$ $ 
 ({\mathcal F}_{\tau}), $ $ \P )$ is a  stochastic basis  
  on {which it is defined}  a 
 {cylindrical} $U$-valued {${\mathcal F}_{\tau}$}-Wiener process  $W$ and
 a continuous ${\mathcal F}_{t}$-adapted $H$-valued
process $X = (X_\tau)$ $   = (X_\tau)_{\tau \in [ t, T] }$   such that
\begin{equation}
X_{\tau}=e^{(\tau - t) A}x+\int_{t}^{\tau}e^{\left(  \tau -s\right)  A}G \widetilde C (s, 
X_{s} )
ds+\int_{t}^{\tau}e^{\left(  \tau -s\right)  A}GdW_{s},\;\;\; \tau \in [t,T],\label{mild}    
\end{equation}
 $\P$-a.s.

 We say that
equation \eqref{abs_1nonlinear_stochastic_eq}  has a  { strong mild solution} if, on every stochastic basis 
$(\Omega,{\cal F}, ({\cal F}_t), P)$ on which there is defined a  cylindrical
${\cal F}_t$-Wiener process $W$ on $U$, 
there exists 
  a weak mild solution. 

Under Hypothesis \ref{hyp_1} the existence of a 
(weak)  mild  solution to problem \eqref{abs_1nonlinear_stochastic_eq} which is unique in law is a direct consequence of the Girsanov theorem.  Recall that  
  
\begin{proposition}\label{prop:esist-weak}
Under Hypothesis \ref{hyp_1}, for any $x\in H$, $t \in [0.T[$, there exists a unique in law  (weak) mild solution  to  equation  \eqref{abs_1nonlinear_stochastic_eq}. 
\end{proposition}
\begin{proof} The result directly follows from an infinite dimensional version of the Girsanov theorem, see e.g. \cite[Proposition 7.1]{Ondre04} and \cite[Section 10.3]{DPsecond}.
\end{proof}

For any $x\in H$ we consider the Ornstein-Uhlenbeck process $\Xi=(\Xi_t^{0,x})$, i.e. the unique solution to \eqref{abs_1nonlinear_stochastic_eq} with $\widetilde C=0$. The vector-valued Ornstein-Uhlenbeck transition semigroup $(R_t)_{t\geq0}$ associated to $\Xi$ is defined as $R_t[\Phi](x):=\mathbb E\big[\Phi(\Xi_t^{0,x})\big]$ for any $t\geq0$, any $\Phi\in B_b(H, H)$ and any $x\in H$ (cf. \cite{dapr-fl}). 

 {We take any  generator $A_0$ of a $C_0$-semigroup} of linear operators $(e^{tA_0})_{t\geq0}$ in $L(H)$, we fix $\mathcal T \in (0,T]$ and consider the integral equation
\begin{equation}\label{d5_A_0}
v(t,x)
=   \int_t^{\mathcal T}R_{s-t}\left[e^{(\mathcal T-s)A_0}G\widetilde C(s,\cdot)\right](x)ds
+ \int_t^{\mathcal T}R_{s-t}\left[\nabla^Gv(s,\cdot)\widetilde C(s,\cdot)\right](x)ds,
\end{equation} 
 for any $(t,x)\in[0,\mathcal T]\times H$.

We   define the space $\calg^{0,1}([0,\T]\times H;H)$, see \cite{fute} Section 2.2, as the subspace of
  $C_b ([0,\T]\times H;H)$ consisting of all functions $f$ 
which are G\^ateaux differentiable with respect to $x$ and such that the map $\nabla f:[0,\T]\times H\to L(H)$ is strongly continuous and globally bounded.  
 
\begin{definition}
\label{defn:soluzione_int}
A solution to \eqref{d5_A_0}  is a mapping  ${ u = u^{\mathcal T}} \in
\calg^{0,1}([0,\T]\times H;H)$ which solves \eqref{d5_A_0}.
  \end{definition} 
 To prove our main result we  require the existence of a family of operators $(A_n)_{n\in\N}$, generators of {\sl $C_0$-groups} of linear operators   $((e^{tA_n})_{t\in\R})_{n\in\N}$, and for each $n\in\N$ we consider the integral equation \eqref{d5_A_0} with $A_n$ in the place of $A_0$, that is 
\begin{equation}\label{d5}
v(t,x)
=   \int_t^{\mathcal T}R_{s-t}\left[e^{(\mathcal T-s)A_n}G\widetilde C(s,\cdot)\right](x)ds
+ \int_t^{\mathcal T}R_{s-t}\left[\nabla^Gv(s,\cdot)\widetilde C(s,\cdot)\right](x)ds.
\end{equation} 
{ Note that \eqref{d5} is a mild formulation of a    Kolmogorov PDE
   related  to \eqref{abs_1nonlinear_stochastic_eq}; on this aspect see  Remark \ref{rm:comp_PDE} which also compares such parabolic PDE with the  one considered in \cite{dapr-fl}. } 
 
We will denote by $u^{\mathcal T}_n$ the solution of this equation \eqref{d5}. 
On the family of operators $(A_n)_{n\in\N}$ and on the solution of the above integral equation we make the following assumptions.
\begin{hypothesis}\label{hyp_2_1}
\begin{itemize}
\item[(A)]
For any $n\in\N$ the operator $A_n$ generates a  strongly continuous  group of linear and bounded operators $(e^{tA_n})_{t\in\R}\subset L(H)$ such that for any $T>0$ we have
\begin{equation}\label{d3}
 \sup_{t \in [0,T]} \sup_{n \ge 1} \| e^{t A_n}\|_{L(H)} = K_T < \infty, \quad 
\lim_{n \to \infty} e^{t A_n} x = e^{t A} x, \;\; x \in H,\;\; t \ge 0.
\end{equation}
\item[(B)] 
\label{hyp_2_2}
For each $n \in\N $ there exists a  unique solution $u^{\mathcal T}_n$ to \eqref{d5} which verifies:

\noindent (i) there exists $C_T>0$, { independent of $n$ and $\T$,} such that  
\begin{gather} \label{sg} 
 \sup_{ x \in  H} | u^{\mathcal T}_n (0,x + y) - u^{\mathcal T}_n (0,x ) |_{H}\leq \, C_T  \, |y|_H, \quad y\in H.
\end{gather}
 (ii)
 For any $n\in \N$ and any $(t,x)\in[0,\T]\times H$, the map $k\mapsto 
 \nabla ^G_ku_n^{\T}(t,x)\in L_2(U;H)$  and $\nabla ^G_{\cdot}  u_n^{\T}   \in B_b([0,\mathcal  T]\times H;L_2(U;H))$.   Further, we assume that  
 there exists an integrable function $h : (0,\T) \to \R_+$, independent of $n\in\N$, such that for any $n\in\N$ we have 
  \begin{equation}\label{df}  
{ \sup_{x\in H}\|\nabla^G_{\cdot}u^{\mathcal T}_n (t,x + y) -\nabla^G_{\cdot}u^{\mathcal T}_n (t,x ) \|_{L_2(U;H)}^2 \, \leq \, h(\mathcal T-t)\, |y|_H^2, \quad t\in(0,\T), \ y\in H. }
\end{equation}
Further, there exists a positive constant $C=C(T)$ which only depends on $T$ such that $\|h\|_{L^1(0,\T)}\leq C$ for any $\T\in(0,T]$.
\end{itemize}
 \end{hypothesis}   
\begin{remark}  
In order to ensure the validity of (i) and (ii) in (B), a sufficient condition for solutions to \eqref{d5}   is the following one:  $u_n^{\mathcal T}(t,\cdot)$  is  Fr\'echet differentiable on $H$, for any $t\in[0,\mathcal T],$ with  uniformly  bounded  Fr\'echet derivative  $\nabla u_n^{\mathcal T}(t,\cdot)$, i.e.,   
  $$
  \sup_{n \ge 1} \sup_{t \in [0,\T]}  \sup_{x \in H} \,   \| \nabla u_n^{\mathcal T}(t,x) \|_{L(U,H)}      < \infty.
  $$   
 Moreover,  for any $k\in U$,  $t\in [0,\mathcal T]$, the map $x\mapsto \nabla^G_k u^{\mathcal T}_n(t,x)$ is Fr\'echet differentiable on $H$, 
   $\nabla ^G_{\cdot}\,  u_n^{\T}   \in B_b([0,\mathcal  T]\times H;L_2(U;H))$ and   
\begin{equation}\label{dfdf}   
 \sup_{x\in H}\|\nabla_y\nabla^G_{\cdot}u^{\mathcal T}_n (t,x)\|^2_{L_2(U;H)}\leq \, h(\mathcal T -t) \, |y|^2_H, \quad t\in(0,\T), \ y\in H,
\end{equation}
 for some   integrable  function $h: (0,\T) \to \R_+$ with $\|h\|_{L^1(0,\mathcal T)}\leq C$ for any $\mathcal T\in(0,T]$, for some positive constant $C=C(T)$. 
\end{remark}
 Note that for any $x,y\in H$ and $t\in(0,T)$ we have   (using an orthonormal basis $(f_k)$ in $U$)
\begin{align} 
 \|\nabla_y\nabla^G_{\cdot}u^{\mathcal T}_n (t,x)\|_{L_2(U;H)}^2
 =  & 
 \sum_{k \ge 1} |\nabla_y\nabla^G_{f_k}u^{\mathcal T}_n (t,x)|^2_H
 = \sum_{k \ge 1} |\nabla_y\nabla_{G f_k}\, 
 u^{\mathcal T}_n (t,x)|^2_H  \notag
\\
= &  \sum_{j \ge 1}\sum_{k \ge 1}   \langle \nabla_y\nabla_{G f_k}\, 
 u^{\mathcal T}_n (t,x) , e_j\rangle^2_H 
 =  
 \sum_{j \ge 1}\sum_{k \ge 1} \, [\nabla_y \nabla_{G f_k} 
 u^ {\mathcal T }_{n,j} (t,x)]^2  \notag \\
\label{drr}
= & \sum_{j \ge 1} \, |\nabla_y\nabla^{G}  
 u^{{\mathcal T }}_{n,j} (t,x)|^2_U,   
\end{align}   
where $u_{n,j}^{\T}:=\langle u_n^{\T},e_j\rangle$  (we are identifying $L(U; \R)$ with $U$). 
We will prove that the unique mild solution $u_n^{\T}$ of \eqref{d5} satisfies \eqref{dfdf}, and to prove this estimate we will make use of \eqref{drr}.

\begin{remark}
In the case of the wave equation as in \cite{MasPri17} {and \cite{MasPri23}} we can take $A_n=A$ and so $e^{tA_n} = e^{tA}$because $A$ is the generator of a group of operators, while in the case of the damped equation $A_n$
 will be  the Yosida approximations, i.e., 
\begin{align*}
A_n:=nAR(n,A), \quad n\in\N,
 \end{align*}
where $R(\lambda,A)$ is the resolvent operator of $A$, for any $\lambda$ in the resolvent set of $A$. 
Finally, for the stochastic parabolic PDEs considered in \cite{dapr-fl} we will consider $A_n$ as the finite  dimensional approximations of $A$.
\newline We notice that the stochastic damped equation and the stochastic parabolic PDEs considered in \cite{dapr-fl} are reformulated as a stochastic evolution equation in $H$ like \eqref{abs_1nonlinear_stochastic_eq}, with $A$ generator of a strongly continuous semigroup of linear operators, while the Yosida approximants of $A$ and the finite dimensional approximations of $A$ generate a group of linear operators.
\end{remark}

We stress that if the drift $C$ is Lipschitz continuous with respect to $x$, uniformly in $t$, existence of a strong  mild solution follows in a standard way, see e.g. \cite[Theorem 7.6]{DPsecond}.
\begin{proposition}\label{prop:esist-lip}
Assume that $C:[0,T]\times U\rightarrow H$ is continuous, bounded and $h\mapsto C(\tau,h)$ is Lipschitz continuous, uniformly with respect to $\tau$. Then for any $x\in H$ there exist a unique mild solution to equation \eqref{abs_1nonlinear_stochastic_eq}.
\end{proposition} 

To conclude this section, we provide the following generalization of the Gronwall lemma which will be used in the sequel  (see \cite[Lemma 3.1]{FeFlandoli2013}).
\begin{lemma}  
\label{lem:gron_mod}  
 Let $f, u, v : [0,T] \to \R$ be  bounded measurable functions. Moreover, $f: [0,T] \to \R_+$. Let $C \ge 0$ and $g,h :  (0,T) \to \R_+$ be  integrable functions. 
 \begin{enumerate}[(i)]
 \item If for any $t\in[0,T]$ 
$$
v(t) \le f(t) + \int_t^T g(s-t) v(s) ds,\;\;\; t \in [0,T],
$$
then, for any $t \in [0,T]$,  
\begin{align*}
v(t)\leq f(t) + \; e^{\|g\|_{L^1(0,T)}} \, \int_t^T f(s) g(s-t)  ds.
\end{align*}
\item  If for any $t\in[0,T]$   
$$ 
u(t) \le C + \int_0^t h(t-s) u(s) ds,\;\;\; t \in [0,T],
$$
then
\begin{align*}
u(t)\leq C\left(1+\|h\|_{L^1(0,T)}e^{\|h\|_{L^1(0,T)}}\right),\;\; { t \in [0,T]}.
\end{align*}
\end{enumerate}
\end{lemma}
\begin{proof} Assertion (i) directly follows from \cite[Lemma 3.1]{FeFlandoli2013}. Assertion (ii) can be deduced by (i) as follows. 
{  We have} $u(t) \le C + \int_0^t h(r) u(t-r) dr$; define   $w (t)= u(T-t)$. From the assumptions the function $s\mapsto h(t-s)w(s)\in L^1(t,T)$ for any $t\in[0,T]$. We get 
\begin{gather*}
w(t) \le C + \int_0^{T-t} h(r) u(T - t -r) dr 
=
C + \int_0^{T-t} h(r) w(t + r) dr
\\
= C + \int_t^{T} h(s-t) w(s) dr,
\end{gather*}
and (ii) follows by (i).
 \end{proof}

\begin{remark}
We recall that under the assumptions of Lemma \ref{lem:gron_mod}, if
\begin{align*}
u(t)\leq C+\int_0^t h(s)u(s)ds, \quad t\in[0,T],
\end{align*}
then
\begin{align*}
u(t)\leq C\left(1+\|h\|_{L^1(0,T)}e^{\|h\|_{L^1(0,1)}}\right), \quad t\in[0,T].
\end{align*}
\end{remark}

{ In the following remark we compare the Kolmogorov equations \eqref{d5}  used in the  present paper and the ones considered in \cite{dapr-fl}.} 


\begin{remark}\label{rm:comp_PDE}
Let $\mathcal T\in(0,T]$ and let $n\in\N$. Formal computations give that the $H$-valued solution $u_n^{\T}$ of \eqref{d5} formally solves the equation
\begin{align} 
\label{formal_PDE}
\left\{ 
\begin{array}{ll}
\displaystyle \frac{\partial u_n^{\mathcal T}(t,x)}{\partial t}+\mathcal L_t[  u_n^{\mathcal T}(t,\cdot)](x)
= -e^{(\mathcal T-t)A_n}G\widetilde C(t,x), & x\in H,\ t\in[0,\mathcal T], \vspace{1mm}
 \\
 u_n^{\mathcal T}(\mathcal T,x)=0,   & x\in H,
\end{array}
\right.   
\end{align}     
where $\mathcal L_tf(x):=\frac12{\rm Tr}[GG^*\nabla^2f(x)]+\langle Ax,\nabla f(x)\rangle+\langle G\widetilde C(t,x),\nabla f(x)\rangle$, for any $t\in[0,T]$ and any $x\in H$, and $f$ is a regular function.   
 Kolmogorov equations similar to \eqref{formal_PDE} are considered in \cite{dapr-fl} for the study of semilinear  parabolic SPDEs (however, note that in equation (6) of \cite{dapr-fl} the term $-e^{(\mathcal T-t)A_n} G \widetilde C$ is replaced by  $G \widetilde    C$).        
 
 On the other hand, the function $v_n:=e^{-(\mathcal T-t)A_n}u_n^{\mathcal T}$ formally solves
\begin{align*}  
\left\{ 
\begin{array}{ll}
\displaystyle \frac{\partial v_n^{\mathcal T}(t,x)}{\partial t}+\mathcal L_t[v_n^{\mathcal T}(t,\cdot)](x)
= A_n v_n^{\mathcal T}(t,x)-G\widetilde C(t,x), & x\in H,\ t\in[0,\mathcal T], \vspace{1mm}\\
v_n^{\mathcal T}(\mathcal T,x)=0,  & x\in H,
\end{array}
\right.
\end{align*}
which is similar to the equation (formally) solved by the function $v$ considered in \cite[Remark 6.2]{MasPri17}.
\end{remark}

\section{Ornstein-Uhlenbeck processes and   approximated FBSDEs}
\label{sub:FBSDE_app}

In this section we consider a family of FBSDEs on the time interval $[t,\mathcal T]$, with $0\leq t<\mathcal T\leq T$. Namely in a complete probability space $(\Omega, \calf, \P)$, for any $n\in\N$ and for $\tau\in[t,\mathcal T]$ we consider the following system of FBSDEs with forward and backward equations both taking values in $H$,  given by
\begin{align}
\label{abs_FBSDE_approssimato_n}
\left\{
\begin{array}{ll}
d\Xi_\tau^{t,x}=A\Xi_\tau^{t,x}d\tau+GdW_\tau,  & \tau\in[t,\T], \vspace{1mm}\\
\Xi_\tau^{t,x}=x, & \tau\in[0,t], \vspace{1mm}\\
 -dY_\tau^{t,x,n}=-A_nY_\tau^{t,x,n}d\tau+G\widetilde C(\tau,\Xi_\tau^{t,x})d\tau  
+Z_\tau^{t,x,n}\widetilde C(\tau,\Xi_\tau^{t,x})d\tau-Z_\tau^{t,x,n}dW_\tau,  & \tau\in[0,\T], \vspace{1mm} \\  
Y_{\mathcal T}^{t,x,n}=0.
\end{array}
\right.
\end{align}
The solution to the forward equation is the so called Ornstein-Uhenbeck process $\Xi=(\Xi_\tau^{t,x})$ and it is nothing else than equation \eqref{abs_1nonlinear_stochastic_eq} with drift $C$ equal to $0$. 
 Moreover $G$, $\widetilde C$ and $W$ are the same as in \eqref{abs_1nonlinear_stochastic_eq}. 
Finally $(A_n)_{n\geq 1}$ are given in Hypothesis \ref{hyp_2_1}, part $(A)$.

\begin{remark} \label{MaPr}  
 We recall that
under the general assumptions of the present paper we cannot consider the BSDE with $A$ instead of $A_n$ because $-A$ is not the generator of a semigroup of operators (see \cite{Gua1}). This case has been considered in \cite{MasPri17} since $A$ is the generator of a group of operators, and here we generalize the method introduced in \cite{MasPri17} to the case of $A$ generator of a semigroup of linear operators. 
\end{remark}
  
Let $n\in\N$. Concerning the backward equation in the FBSDE (\ref{abs_FBSDE_approssimato_n}), its precise meaning  is given by its mild formulation: $\P$-a.s. the pair of processes $(Y^{t,x,n}, Z^{t,x,n})$ satisfies
 \begin{align}
 Y_{\tau}^{t,x,n}= & \int_\tau^{\mathcal T}e^{-(s-\tau){A_n}}G \widetilde C(s,\Xi^{t,x}_s)\,ds+\int_\tau^{\mathcal T}e^{-(s-\tau){A_n}}  Z_s^{t,x,n}\, \widetilde C(s,\Xi^{t,x}_s)\,ds \notag \\
& -\int_\tau^{\mathcal T}e^{-(s-\tau){A_n}} Z^{t,x,n}_{s}\;d W_s,
 \label{abs_FBsystem_1-mild}
\end{align}
for any $\tau\in [t,\mathcal T]$   
 (cf. \cite{fute}, \cite{Gua1}, \cite{HuPeng} and the references therein).  Notice that in order to give sense to the BSDE in (\ref{abs_FBSDE_approssimato_n}) we need that $-A_n$ is the generator of a $C_0$-semigroup of bounded linear operators, and this is true if we assume that Hypothesis \ref{hyp_2_1}, part $(A)$ is satisfied. 
 
  Concerning equation \eqref{abs_FBsystem_1-mild} recall that  
 we endow $(\Omega, \calf, \P)$ with the natural filtration $({\cal F}_t^W)$ of $W$ (i.e., ${\cal F}_t^W$ is the smallest $\sigma$-algebra generated by $W^{(n)}_s$, ${n \ge 1}$ and $0 \le s \le t$)  augmented in the usual way with the family of $\P$-null
 sets of $\calf$. All the concepts of measurability, e.g. predictability, are referred to this filtration.
 
 The solution of \eqref{abs_FBsystem_1-mild} will be a pair of processes $(Y^{t,x,n}, Z^{t,x,n})\in L^2_\calp(\Omega; C([0,\T];H))\times L^2_\calp(\Omega\times[0,\T];L_2(U;H))$   (see Proposition \ref{Teo:ex-regdep-BSDE}), where $L^2_\calp(\Omega; C([0,\T],H))$ is the Banach space of all predictable $H$-valued processes $Y$ with continuous paths and such that
$$
 \E [ \sup_{\tau\in[0,\T]}\vert Y^{}_\tau\vert^2] = \| Y\|_{L^2_\calp(\Omega, C([0,\T];H))}^2 < \infty, 
$$
and $L^2_\calp(\Omega\times[0,\T];L_2(U,H))$ is the  usual $L^2$-space    of predictable processes $Z$ with values in $L_2(U,H)$.
 We sum up in the following proposition existence results for equation \eqref{abs_FBsystem_1-mild}.
\begin{proposition}\label{Teo:ex-regdep-BSDE}
Assume Hypotheses \ref{hyp_1} and \ref{hyp_2_1} hold true. Then, for any $n\in\N$ the BSDE \eqref{abs_FBsystem_1-mild} admits a unique solution  
$(Y^{t,x,n},Z^{t,x,n})\in L^2_\calp(\Omega; C([0,\T];H)) \times L^2_\calp(\Omega\times[0,\T];L_2(U;H))$, satisfying 
 \begin{equation}\label{est-YZ}
  \E\big [ \sup_{\tau\in[0,\T]}\vert Y^{t,x,n}_\tau\vert^2 \big ]+\E\int_0^{\T}\Vert Z_\tau^{t,x,n}\Vert^2_{L_2(U,H)} d \tau
  \leq C_{\T,n} \| C\|_{\infty},
   \end{equation}
   where $C_{\T,n}$ is a positive constant which depends  also on $\T$ and $n$, and the map:  $(t,x)\mapsto Y_t^{t,x}  $,    $[0,\T]\times  H \to H$, is deterministic.
 \end{proposition}
\begin{proof} Existence and uniqueness of a solution directly come from Lemma 2.1 and Proposition 2.1
in \cite{HuPeng}, that we can apply since
$C$ is bounded. Estimate (\ref{est-YZ}) follows from \cite{GuaTess}, Remark 4.5, estimate (4.19). Since the process $\Xi^{t,x}$ is ${\cal F}_{t,T}^W$-measurable (where ${\cal F}_{t,T}^W$ is the $\sigma$-algebra generated by $W_r - W_t$, $r \in [t,T]$, augmented with the $\P$-null sets), it turns out that $Y^{t,x}_t$ is measurable both with respect to ${\cal F}_{t, T }^W$ and ${\cal F}_t$; it follows that $Y^{t,x}_t$ is deterministic. 
\end{proof} 

Next we prove an identification property that in the present paper we will apply to $u_n^{\T}$, solution to equation \eqref{d5}, which is the analogous of the  identification formulae proved e.g. in \cite{fute} for real valued  functions  (see also \cite{Mas-Ban} for the case of functions defined on Banach spaces). 

\begin{lemma}
\label{prop:abs_identification_formulae_tilde} Let $v:\left[  0,\mathcal T\right]  \times H\longrightarrow H$ be a continuous function such that for every $t\in\left[0,\T\right]$, $v\left(  t,\cdot\right)$ is G\^ateaux differentiable and the map $\left(  t,x\right)  \mapsto\nabla v\left(  t,x\right)  $ is Borel measurable.  Let us fix $(t,x) \in [0, \T] \times H$ and   let $\Xi^{t,x} $ be the Ornstein-Uhlenbeck process defined in Section  \ref{sub:abstract_equation}. If $\psi$ is a square integrable predictable process and  $\bar Z \in L^2_\calp(\Omega \times [0,\T]; L_2(U,H))$ and if $v\left(  \tau,\Xi_{\tau}^{t,x}\right)  $ admits the representation  
\begin{equation} \label{v identific con Z}
v\left(  \tau,\Xi_{\tau}^{t,x}\right)  =v\left( \mathcal  T,\Xi_{\mathcal T}^{t,x}\right)  +\int_{\tau}^{\mathcal T}\psi_{s}ds-\int_{\tau}^{\mathcal T} \bar Z_{s}dW_{s},\quad\tau\in\left[  {0},\mathcal T\right]  ,
\end{equation}
then {  $\P$-a.s},     $\nabla^G v\left(  \tau,\Xi^{t,x}_{\tau}\right)  =\bar Z_{\tau}^{}$,  for a.e. $\tau\in[{0},\T]$.
\end{lemma}
\begin{proof} The result can be seen as an extension of  \cite[Proposition 5.6]{fute} to the case of an $H$-valued BSDE, and for this extension we use techniques similar to the ones in \cite{MasPri17}.
Let  $\xi  \in  U$ and consider the real Wiener process $(W^\xi_\tau)_{\tau\ge 0}$, where
 \[
  W^\xi_\tau:= \langle \xi , W_{\tau} \rangle_U. 
\] 
Let $h \in H$, we set  
$$
v^h(\tau, x):=\langle v(\tau, x),h\rangle_H,\quad 0\leq t\leq \tau\leq \mathcal T,\, x\in H,
$$
and we study  the joint quadratic variation between the real process $v^h\left(  \cdot,\Xi^{t,x}\right) $ and $ W^\xi$. Since
\begin{align*}
v^h\left(  \tau,\Xi_{\tau}^{t,x}\right)  
=& v^h\left( \mathcal  T,\Xi_{\mathcal T}^{t,x}\right)  +\int_{\tau}^{\mathcal T}\langle\psi_{s},h\rangle_H\,ds-\int_{\tau}^{\mathcal T} \langle \bar Z_{s},h\rangle_H\,dW_s \\
= & v^h\left( 0,\Xi_{0}^{t,x}\right)  -\int_{0}^{\tau}\langle\psi_{s},h\rangle_H\,ds+\int_{0}^{\tau} \langle \bar Z_{s},h\rangle_H\,dW_s,\quad\tau\in\left[  0,\mathcal T\right] ,
\end{align*}
we find 
\begin{equation} \label{joint}
\<v^h\left(  \cdot,\Xi^{t,x}\right), W^{\xi}\>_{\tau} = \int_{{0}}^\tau  \langle \bar Z^{}_{s} \xi ,   h \rangle_H ds, \;\; \tau \in [0,\T], \; \P\text{-a.s.}  
\end{equation}
Now we compute the joint quadratic variation in a different way, arguing as in \cite[Lemmata 6.3 \& 6.4]{fute} and we obtain that the real process $v^h\left(  \cdot,\Xi^{t,x}\right) $ admits joint quadratic variation with $W^{\xi}$ given by 
\[
 \< v^h(\cdot, \Xi^{t,x}), W^{\xi}\>_{\tau}=
 \int_0^\tau \nabla^G v^h(s,\Xi_s^{t,x}) \xi \,ds 
 = \int_0^\tau  \langle \nabla^G  v (s,\Xi_s^{t,x}) \xi, h \rangle \,ds, \;\; \tau \in [0,\T].
\]
Comparing this formula with \eqref{joint} we get that for a.e. $s \in [0,T]$ we have, $\P$-a.s.,
$$
\langle \bar Z^{}_{s} \xi ,   h \rangle 
  = \langle  \nabla^G  v(s,\Xi_s^{t,x}) \xi, h \rangle.
$$
Since $H$ is separable, for any $\xi\in U$ we get   { $\P$-a.s.}  $\bar Z_s^{}\xi=\nabla^G  v(s,\Xi_s^{t,x}) \xi $,   for a.e. $s \in [0,\T]$. The assertion now follows.
\end{proof}

Next we are going to apply the previous Lemma \ref{prop:abs_identification_formulae_tilde} to $u_n^{\T}$ solution to equation \eqref{d5}. We assume that Hypothesis \ref{hyp_2_1}, part $(B)$,} holds true and recall that the operator $A_n$ generates a group of linear bounded operators $(e^{\sigma A_n})_{\sigma\in\R}$. We want to show that the pair of processes 
\begin{equation}\label{pair}
(e^{-(\mathcal T-\tau )A_n}u_n^{\T}(\tau, \Xi_\tau^{t,x}), \nabla^G e^{-(\mathcal T-\tau )A_n} u^{\T}_n(\tau, \Xi_\tau^{t,x})),\, \tau \in [t,\mathcal T],
\end{equation}
satisfy the BSDE \eqref{abs_FBsystem_1-mild}. This is the content of the following proposition.
\begin{proposition}\label{prop_identific}
Assume Hypotheses \ref{hyp_1} and \ref{hyp_2_1} hold true. Let $u^{\T}_n$ be the unique solution to \eqref{d5}.  
 Then the pair of processes defined in \eqref{pair} is the unique solution of the BSDE \eqref{abs_FBsystem_1-mild}. As a consequence, if the pair of processes $(Y^{t,x,n}, Z^{t,x,n})$  is solution to the BSDE    \eqref{abs_FBsystem_1-mild}, then we have $Y^{t,x,n}_t=e^{-(\mathcal T-t )A_n}u^{\T}_n(t,x),\,Z^{t,x,n}_t= e^{-(\mathcal T-t )A_n} \nabla^Gu^{\T}_n(t,x)$, so that
\begin{align}
\label{abs_identification_formulae_tilde_2} 
  \nabla^Gu _n^{\T}   (\tau,\Xi_\tau^{t,x})=e^{(\mathcal T-\tau)A_n}Z_\tau^{t,x,n},  \quad \P\textup{-a.s.},   \text{for a.e.} \; \tau\in[t,\mathcal T].
\end{align}

Moreover, if a pair of processes $(Y^{t,x,n},Z^{t,x,n})$ is solution to \eqref{abs_FBsystem_1-mild}, then by setting
$$ 
 \widetilde v^n(t,x):=e^{(\mathcal T-t)A_n}Y^{t,x,n}_t,\, t\in [ 0,\T],
$$
we get that 
$$ 
  \nabla ^G\widetilde v^n(t,x):= e^{(\mathcal T-t)A_n}Z^{t,x,n}_t,\, t\in [ 0,\T],
$$
and $\widetilde v^n(t,x),\, t\in [ 0,\T],\, x \in H$, is solution to equation \eqref{d5}, i.e., $\widetilde v_n\equiv u_n^{\T}$. 
\end{proposition}
\begin{proof} The arguments are similar to those used in the proof of the uniqueness part of \cite[Theorem 6.2]{fute}, and are adequated to this different context. For reader's convenience, we simply write $u_n$ instead of $u_n^{\T}$.
\newline Starting from equation \eqref{d5}, we notice that for any $\tau\in[t,\T]$, 
$$
u_n(\tau, x)  = \E\int_\tau^{\mathcal T}\left[e^{(\mathcal T-s)A_n}G\widetilde C(s,\Xi_s^{\tau,x})\right]ds+\E \int_\tau^{\mathcal T}\left[\nabla^Gu_n(s,\Xi_s^{\tau,x})\widetilde C(s,\Xi_s^{\tau,x})\right]ds. 
$$
Hence, denoting the conditional expectation $\E[\cdot|\mathcal F_\tau]$ as $\E^{\mathcal F_\tau}[\cdot]$, the random variable $u_n(\tau,\Xi_\tau^{t,x})$ satisfies $\mathbb P$-a.s.
\begin{align}\label{d5-bis}
u_n(\tau,\Xi_\tau^{t,x}) := &      \Etau\int_\tau^{\mathcal T}\left[e^{(\mathcal T-s)A_n}G\widetilde C(s,\Xi_s^{\tau,\Xi_\tau^{t,x}})\right]ds+\Etau \int_\tau^{\mathcal T}\left[\nabla^Gu_n(\tau,\Xi_s^{\tau,\Xi_\tau^{t,x}})\widetilde C(s,\Xi_s^{\tau,\Xi_\tau^{t,x}})\right]ds   \nonumber 
\\
= & \Etau\int_\tau^{\mathcal T}\left[e^{(\mathcal T-s)A_n}G\widetilde C(s,\Xi_s^{t,x})\right]ds+\Etau \int_\tau^{\mathcal T}\left[\nabla^Gu_n(s,\Xi_s^{t,x})\widetilde C(s,\Xi_s^{t,x})\right]ds,
\end{align} 
since  $\Xi_s^{t,x} $ is $\calf_s$-measurable $\forall\, \tau\leq s\leq \mathcal T$, and  $\Xi_s^{\tau,\Xi_\tau^{t,x}}=\Xi_s^{t,x}$.
 Setting $$\xi:=\int_{{t}}^{\mathcal T}\left[e^{(\mathcal T-s)A_n}G\widetilde C(s,\Xi_s^{t,x})\right]ds+ \int_{{t}}^{\mathcal T}\left[\nabla^Gu_n(s,\Xi_s^{t,x})\widetilde C(s,\Xi_s^{t,x})\right]ds,$$ we can rewrite \eqref{d5-bis} as 
\begin{align*}
u_n(\tau,\Xi_\tau^{t,x}) &=\Etau\xi- \int_t^\tau\left[e^{(\mathcal T-s)A_n}G\widetilde C(s,\Xi_s^{t,x})\right]ds- \int_t^\tau\left[\nabla^Gu_n(s,\Xi_s^{t,x})\widetilde C(s,\Xi_s^{t,x})\right]ds,
\end{align*} 
for   
 $\tau\in[t,\T]$. By the representation theorem for martingales, see e.g. \cite{DPsecond}, there
exists  $(\bar Z^{t,x}_s)_{s \in [0. \T]}    \in  L^2_{\calp}(\Omega \times [0,\T]; L_2(U,H))$ such that for any $\tau\in[0,\T]$ 
\begin{align*}
u_n(\tau,\Xi_\tau^{t,x}) &=u_n(t,x)+\int_t^{\tau\wedge t} \bar Z_s^{t,x}\, dW_s\\&- \int_t^{\tau\wedge t}\left[e^{(\mathcal T-s)A_n}G\widetilde C(s,\Xi_s^{t,x})\right]ds- \int_t^{\tau\wedge t}\left[\nabla^Gu_n(s,\Xi_s^{t,x})\widetilde C(s,\Xi_s^{t,x})\right]ds, \quad \P\textup{-a.s.}
\end{align*}
We conclude that the process $u_n(\tau,\Xi_\tau^{t,x}),\,\, t\leq \tau \leq \mathcal T$ is a continuous semimartingale with canonical decomposition. By Lemma \ref{prop:abs_identification_formulae_tilde}, we have that $\bar Z_s^{t,x}=\nabla^G  u_n(s, \Xi_s^{t,x})$ a.e. $s\in [t,\mathcal T]$. Then the previous semimartingale reads as
\begin{align*}
u_n(\tau,\Xi_\tau^{t,x}) &=-\int_\tau^{\mathcal T} \nabla^G  u_n(s, \Xi_s^{t,x})\, dW_s\\&+\int_\tau^{\mathcal T}\left[e^{(\mathcal T-s)A_n}G\widetilde C(s,\Xi_s^{t,x})\right]ds+ \int_\tau^{\mathcal T}\left[\nabla^Gu_n(s,\Xi_s^{t,x})\widetilde C(s,\Xi_s^{t,x})\right]ds.
\end{align*}
If we apply $e^{-(\mathcal T-\tau)A_n}$ to both the sides of this equation we get
\begin{align*}
&e^{-(\mathcal T-\tau)A_n}u_n(\tau,\Xi_\tau^{t,x}) =-\int_\tau^{\mathcal T} e^{-(s-\tau)A_n}\left[e^{-(\T-s)A_n} \nabla^G  u_n(s, \Xi_s^{t,x})\,dW_s\right]
\\
&\;+ \int_\tau^{\mathcal T}  \left[e^{-(s-\tau)A_n}G\widetilde C(s,\Xi_s^{t,x})\right]ds+   \int_t^\tau e^{-(s-\tau)A_n}\left[e^{-(\mathcal T-s)A_n} \nabla^G   u_n(s,\Xi_s^{t,x})\widetilde C(s,\Xi_s^{t,x})\right]ds,  
\end{align*}
and by comparing this expression with \eqref{abs_FBsystem_1-mild}, it is immediate to see that the pair of processes 
\begin{align*}
(e^{-(\mathcal T-\tau )A_n}u_n(\tau, \Xi_\tau^{t,x}),  e^{-(\mathcal T-\tau )A_n} \nabla^G u_n(\tau, \Xi_\tau^{t,x})),\quad \tau \in [t,\mathcal T],
\end{align*}
solves equation  \eqref{abs_FBsystem_1-mild}. 

The ``Moreover" part follows from the fact that, by Proposition \ref{Teo:ex-regdep-BSDE}, equation \eqref{abs_FBsystem_1-mild} admits a unique solution.
\end{proof}

\section{Pathwise uniqueness for the nonlinear SDE}
\label{sec:path_uniq}
Let's go back to the nonlinear SPDE  \eqref{abs_1nonlinear_stochastic_eq}, which we rewrite with initial time $t=0$:
\begin{align*}
\left\{
\begin{array}{ll}
dX_\tau^{0,x}=AX_\tau^{0,x}d\tau+G\widetilde C(\tau,X_\tau^{0,x})d\tau +GdW_\tau & \tau\in[0,T], \\
X_0^{0,x}=x\in H.
\end{array}
\right.
\end{align*}
The existence of a weak solution to \eqref{abs_1nonlinear_stochastic_eq} has been already discussed in Proposition \ref{prop:esist-weak}. Let us set $X_\tau^x:=X_\tau^{0,x}$ for any $\tau\in[0,T]$ and   $x\in H$. Then, $\P{\textup{-a.s.}}$ we have
 \begin{align}
X_\tau^x= 
& e^{\tau A}x+\int_0^\tau \left(e^{(\tau-s)A}-e^{(\tau-s)A_n}\right)G\widetilde C(s,X_s^{x})ds+\int_0^\tau e^{(\tau-s)A_n}G\widetilde C(s,X_s^{x})ds \notag \\
& +\int_0^\tau e^{(\tau-s)A}GdW_s,\qquad \forall \tau\in[0,T].
\label{abs_1mild-formulation-n}
\end{align}
 Let  $\tau \in [0, T]$. In the next result, for any $n\in\N$ we consider  $u_n^{\tau}$,   the regular solution of \eqref{d5} with $\T=\tau$,  whose properties are listed in Hypothesis \ref{hyp_2_1}, part $(B)$.

\begin{proposition}\label{propo:abs_new-mildwaveeq} 
 Let Hypotheses \ref{hyp_1} and \ref{hyp_2_1} hold true. Then, for any $n\in\N$ and any $ \tau \in [0,T]$ we have
\begin{align}
X_\tau^x
=&  e^{\tau A}x+\int_0^\tau \left(e^{(\tau-s)A}-e^{(\tau-s)A_n}\right)G\widetilde C(s,X_s^{x})ds \notag \\
& +u_n^\tau(0,x)+\int_0^\tau \nabla^Gu_n^\tau(s,X_s^x)dW_s+\int_0^\tau e^{(\tau-s)A}GdW_s, \quad \P{\textup{-a.s.}}.
\label{abs_1mild-formulation-n-1bis} 
\end{align}
\end{proposition} 

\begin{proof}  
Let $0\leq t\leq \tau\leq T$ and $x \in H$.  We consider a   (weak) mild solution $X^{t,x}$ to  
\begin{align*}
dX^{t,x}_\sigma=AX_\sigma^{t,x} d \sigma+G\widetilde C(\sigma,X^{t,x}_\sigma)d\sigma +GdW_\sigma, \quad \sigma\in[t,\tau], \quad X^{t,x}_\sigma=x, \quad \sigma\in[0,t],
\end{align*}
which is given by 
\begin{align*}
  X_\sigma^{t,x}
=&  e^{(\sigma -t) A}x+\int_t^\sigma e^{(\sigma-s)A}G\widetilde C(s,X_s^{t,x})ds 
+ \int_t^\sigma    e^{(\sigma-s)A}GdW_s, \quad \P{\textup{-a.s.}}, \quad \sigma\in[t,\tau], \\
 X_\sigma^{t,x}   = & x, \quad \sigma\in[0,t].
\end{align*}
Such solution is defined on a stochastic basis     $(
\Omega,$ $ {\mathcal F},
 ({ \mathcal F}_{t}), \P) $,  
  on {which it is defined}  a 
 {cylindrical}  {${\mathcal F}_{t}$}-Wiener process  $W$ on $U$. 
  
 Let us set  
$$
\widehat W_{\sigma}=W_\sigma+\int_0^\sigma \widetilde C(s,X_s^{t,x})ds.
$$ 
By the Girsanov theorem  (see, for instance, \cite[Section 10.3]{DPsecond} or the Appendix in \cite{DPFPR13}) there exists a probability measure $\widehat\P = \widehat \P_{\tau}$ on $(\Omega,$ 
 ${ \mathcal F}_{\tau}),$  
 such that in $(\Omega, \calf_{\tau}, \widehat \P)$ the process $(\widehat W_{\sigma})_\sigma$ is a cylindrical Wiener process up to time $\tau$ (it is not difficult to prove  that $\widehat \P$ and $\P$ are equivalent). In $(\Omega, \calf_{\tau}, \widehat \P)$ the process $X^{t,x}$ 
 solves  
 $$ 
 dX^{t,x}_\sigma=A X^{t,x}_\sigma d\sigma+Gd\widehat W_\sigma, \;\; X^{t,x}_t=x,\;\;\; \sigma \in [t, \tau].   
 $$
 Since for OU stochastic equations pathwise uniqueness holds, we have, in particular, that  $(X^{t,x})$  is a  predictable process  with respect to the completed  natural filtration $(\calf_{s}^{\widehat W})_{0 \le s \le \tau}$ generated by $\widehat W$.    
 Let us consider the FBSDE system 
\begin{equation}
  \left\lbrace\begin{array}{ll}
 dX^{t,x}_\sigma  =AX^{t,x}
 
 _\sigma d\sigma+Gd \widehat W_\sigma
 , & \sigma\in\left[  t,\tau\right], \vspace{1mm} \\ 
 X_{\sigma}^{t,x}=x, & \sigma\in[0,t], \\ 
  -d Y_{\sigma}=-A_nY_\sigma d\sigma+G\widetilde C(\sigma,X^{t,x}_\sigma)\;d\sigma
  + Z_{\sigma} \, \widetilde C(\sigma,X^{t,x}_\sigma)  d\sigma
  -Z_{\sigma}\;d\widehat W_\sigma,
  &  \sigma\in  [0, \tau], \vspace{1mm} \\
Y_{\tau}=0.
 \end{array}
 \right.  \label{abs_FBSDE-wp2}
 \end{equation} 
From Proposition \ref{prop_identific}, system \eqref{abs_FBSDE-wp2} admits a unique solution 
$(X_\sigma^{t,x},\widehat Y_\sigma^{t,x,n},\widehat Z_\sigma^{t,x,n})$ in $(\Omega, \calf_{\tau},$ $(\calf_{s}^{\widehat W})_{0 \le s \le \tau},  $ $ \widehat \P)$. Recall that $\calf_{s}^{\widehat W } \subset \calf_{s},$ $s \in [0, \tau]$.  Moreover, see \eqref{pair},  we have 
 $\widehat \P$-a.s in $\Omega$ (or equivalently $\P$-a.s. in $\Omega$)
$$
 \widehat Y_\sigma^{t,x,n}=e^{-(\tau- \sigma)A_n}u_n^{\tau}(\sigma,X_\sigma^{t,x}), \;\; \sigma\in[0,\tau],  \;\;  \text{and} \;\; \widehat Z_\sigma^{t,x,n}=e^{-(\tau-\sigma)A_n}\nabla^Gu _n^{\tau}      (\sigma,X_\sigma^{t,x}),  
 $$ 
 for any $\sigma\in[0,\tau]$, a.e.. Hence, $\P$-a.s. for any $s\in[0,\tau]$, a.e.,  
 we have   
\begin{equation} \label{def-tildeYtildeZ_1}
e^{(\tau-s)A_n}\widehat Y_s^{t,x,n}= u_n^\tau (s,X_{s}^{t,x}), \quad {e^{(\tau-s)A_n}\widehat  Z_s^{t,x,n}=\nabla^G  u_n^\tau (s,X_{s}^{t,x}),}  
\end{equation} 
We recall that  
 \begin{align}  
 \widehat Y_{\sigma}^{t,x,n}
  = & \int_\sigma^{\tau}e^{-(s-\sigma)A_n }G \widetilde C(s,X^{t,x}_s)\,ds+\int_\sigma^\tau e^{-(s-\sigma)A_n}\widehat Z_s^{t,x,n}\widetilde C(s,X_s^{t,x})ds \notag \nonumber\\
 & -\int_\sigma^{\tau}e^{-(s-\sigma)A_n }\widehat Z^{t,x,n}_{s}\;d \widehat W_s \notag \nonumber\\
 = &  \int_\sigma^{\tau}e^{-(s-\sigma)A_n }G \widetilde C(s,X^{t,x}_s)\,ds-\int_\sigma^{\tau}e^{-(s-\sigma)A_n }\widehat Z^{t,x,n}_{s}\;d W_s, 
  \quad \P\textup{-a.s.}, \ \sigma\in[0,\tau]. \label{abs_bsde-markovian-wave-mild_3}
    \end{align}
Let us consider now the initial time $t=0$,  and let us write \eqref{abs_bsde-markovian-wave-mild_3} with $\sigma=0$ ( we write $X^x$ instead of $X^{0,x}$). We get
\begin{align*}  
\int_0^{\tau} e^{-sA_n}G\widetilde C(s,X_s^x)ds=  
 e^{-\tau A_n}u_n^{\tau}(0,x) 
+\int_0^{\tau}e^{-sA_n}\widehat Z_s^{0,x,n}dW_s, \quad \P{\textup{-a.s.}}, 
\end{align*}
for any $\tau\in[0,T]$, and by applying $e^{\tau A_n}$ to both  sides, from \eqref{def-tildeYtildeZ_1} we deduce that, $\P$-a.s.,
\begin{align*}
\int_0^{\tau}e^{(\tau-s)A_n}G\widetilde C(s,X_s^x)ds
&=  u_n^{\tau}(0,x)  
+\int_0^{\tau} e^{(\tau-s)A_n}\widehat Z_s^{0,x,n} dW_s\\
&= u^{\tau}_n(0,x)
+\int_0^{\tau} \nabla^Gu^{\tau}_n(s,X_s^x)dW_s, \quad \forall \tau\in[0,T].
\end{align*} 
By replacing in \eqref{abs_1mild-formulation-n} we infer that
\begin{align}
X_\tau^x
=&  e^{\tau A}x+\int_0^\tau \left(e^{(\tau-s)A}-e^{(\tau-s)A_n}\right)G\widetilde C(s,X_s^{x})ds \notag 
\\
& +u_n^\tau(0,x)+\int_0^\tau \nabla^Gu_n^\tau(s,X_s^x)dW_s+\int_0^\tau e^{(\tau-s)A}GdW_s,\quad  \P{\textup{-a.s.}},
\label{abs_1mild-formulation-n-1}
\end{align}
for any $\tau\in[0,T]$, and the proof is  finished.
\end{proof}

 The next result is our main uniqueness theorem.  

 \begin{theorem}\label{teo:abs_uni1}
Let Hypotheses  \ref{hyp_1} and \ref{hyp_2_1} hold true.  Then, there exists a positive constant $c=c(T)>0$ such that for any $x_1,x_2\in H$ we have 
\begin{equation} 
\label{abs_lip-dependence}
  \sup_{t \in [0,T]} \E  [   \vert  X_t^{x_1}-X_t^{x_2}\vert_H^2] \leq c \vert x_1-x_2\vert_H^2, 
\end{equation} 
where $X^{x_1}$ and $X^{x_2}$ denote    (weak) mild solutions to \eqref{abs_1nonlinear_stochastic_eq} starting at $x_1$ and $x_2$, respectively, and defined on the same stochastic basis. 
In particular, for equation (\ref{abs_1nonlinear_stochastic_eq}) pathwise uniqueness holds. 
\end{theorem}  
\begin{proof}  
 Let $x_1,x_2\in H$ and let us denote by $X^1$ and $X^2$ the mild solutions to \eqref{abs_1nonlinear_stochastic_eq} starting at $x_1$ and $x_2$, respectively.
We fix $t\in]0,T]$. From \eqref{abs_1mild-formulation-n-1bis} with $\tau =t$, for any $n\in\N$ we have
\begin{align} 
\label{abs_x1-x2}  
\left(X_t^1-X_t^2\right)
= & e^{t A}(x_1-x_2)+\left(u_n^{t}(0,x_1)-u_n^{t}(0,x_2)\right) 
+\delta_n^1(t) \, +\, \delta_n^2(t
)  \nonumber \\
& +\int_0^t\left(\nabla^Gu_n^t(s,X_s^1)-\nabla^Gu_n^t (s,X_s^2)\right)dW_s, \quad \P{\textup{-a.s.}}, 
\end{align}
where
\begin{align*}
\delta_n^i(t)=\int_0^t \left(e^{(t-s)A}-e^{(t-s)A_n}\right)G\widetilde C(s,X^i_s)ds, \quad  \P{\textup{-a.s.}}, \  i=1,2, \ n\in\N.
\end{align*}
Notice that in \eqref{abs_x1-x2}   we consider the function $u_n^{t}$   such that $u_n^{t}(t, \cdot)\equiv 0$. 

The crucial point is that estimates on  $u_n^t$ are uniform in $t$ (cf. Hypothesis \ref{hyp_2_1}), part $(B)$.
We also note that $x\mapsto e^{t A}x$ is Lipschitz continuous with respect to $x$, uniformly with respect to $t$, and taking into account Hypothesis \ref{hyp_2_1}, part $(B)$, point $(i)$, we know that $x\mapsto u_n^{t}(0,x) $ is Lipschitz continuous. For what concerns the stochastic integral, note that 
using Hypothesis
\ref{hyp_2_1}, part $(B)$, point $(ii)$ {with $\T=t$,}   by the It\^o isometry 
 we find 
 \begin{align} \label{doob}
  \E \left[  \left|\int_0^t\left(\nabla^Gu_n^t(s,X_s^1)-\nabla^G u_n^t (s,X_s^2)\right)dW_s\right|_H^2 \right] 
   \le  \int_0^{t} {h (t-s)} \, \E \left[  |X_s^1-X_s^2|_H^2 \right] ds. 
\end{align}
Using \eqref{doob} in \eqref{abs_x1-x2}, for any $n\in\N$ 
we get
\begin{gather*} 
\mathbb E \left[ \left|X_t^1-X_t^2\right|_H^2\right]
\\
\leq C_{T}\left(|x_1-x_2|_H^2
+ 
\int_0^t  {h (t-s)} \E \left[  |X_r^1-X_r^2|_H^2 \right] ds 
 +\mathbb E\left[ \sup_{t \le T} |\delta_n^1(t)|_H^2\right] 
 +\mathbb E\left[\sup_{t \le T} |\delta_n^2(t)|_H^2\right] \right),
\end{gather*}
where $C_{T}$ is a positive constant independent  of $n$ and $t$. 
Note that
$$
t \mapsto \E \left[ |X_t^1-X_t^2|_H^2\right],
$$
is a bounded function on $[0, T]$. Thus applying the generalized Gronwall lemma we infer
\begin{equation*}
\mathbb E \left[ |X_t^1-X_t^2|_H^2\right] 
\le 
K_T
(\mathbb E\left[ \sup_{t \le T} |\delta_n^1(t)|_H^2\right] 
 +\mathbb E\left[\sup_{t \le T} |\delta_n^2(t)|_H^2\right] + 
  |x_1-x_2|_H^2  ) ,
\end{equation*}
where $K_T$ is a positive constant independent of $n$ and { $t$.} We need to prove that 
\begin{align} \label{asd}
\mathbb E\left[ \sup_{t \le T} |\delta_n^1(t)|_H^2\right] 
 +\mathbb E\left[\sup_{t \le T} |\delta_n^2(t)|_H^2\right] \to 0, \quad n \to \infty.
\end{align} 
Using the dominated convergence theorem,  we get the assertion if we show that, $\P$-a.s.,   
\begin{equation}\label{efg}
\lim_{n \to \infty} \sup_{t \le T} |\delta_n^1(t)|_H^2 =0,\;\;\; 
 \lim_{n \to \infty}
\sup_{t \le T} |\delta_n^2(t)|_H^2 =0.
\end{equation}
 Let us  consider the first limit in \eqref{efg}
 (the proof of the second limit is similar). 
 
Let us fix $\omega$, $\P{\textup{-a.s.}}$; for any  $n \in \N$, we have   
\begin{gather*}
 \delta_n^1(t) = \int_0^t \left(e^{(t-s)A}-e^{(t-s)A_n}\right)   g(s)ds= \int_0^t \left(e^{rA}-e^{r A_n}\right)   g(t-r)dr,
 \  \ n\in\N. 
\end{gather*}
with $g(s) = G\widetilde C(s,X^1_s(\omega))$, $s \in [0,T]$, which is continuous from $[0,T]$ with values in $H$. 
It is enough to prove that   $ \sup_{t \le T} |\delta_n^1(t)|_H \to 0$ as $n \to \infty$.
We note that, for any compact set $K \subset H$, $r \in [0,T]$, there exists ${ y=y_{K,r} }$ such that
\begin{align*}
\sup_{y\in K}\left|\left(e^{rA}-e^{rA_n}\right)y\right|=\left|\left(e^{rA}-e^{rA_n}\right)y_{K, r}\right|,
\end{align*}
and so 
\begin{equation}\label{comp1}
\lim_{n \to \infty} \sup_{y \in K}\left| \left(e^{rA}-e^{rA_n}\right)  y\right|_H =0.
\end{equation} 
 Let us  introduce the compact sets $K_t= \{g(s) \}_{s \in [0,t]}$ $\subset H$, $t \in [0,T]$.  From \eqref{comp1} it follows that
 \begin{gather*}
| \delta_n^1(t)|_H \le \int_0^t \sup_{y \in K_t}\left| \left(e^{rA}-e^{rA_n}\right)   y\right|_H dr
\le  \int_0^T \sup_{y \in K_T}\left| \left(e^{rA}-e^{rA_n}\right)   y\right|_H dr \to 0,\quad  n \to \infty.   
\end{gather*}
 This shows  assertion \eqref{asd} and completes the proof.
\end{proof}

 \begin{remark}\label{gen} We point out that, using a localization argument as in  \cite{DPFPR15} the boundeness of $\widetilde C$ can be dispensed.  In particular, one can prove strong well-posedness of \eqref{nonlinear_stochastic_eq_intro}, for any $x \in H$, under 
  Hypotheses  \ref{hyp_1} and \ref{hyp_2_1} 
 but replacing the condition on $\widetilde C$ with the  weaker assumption:  $\widetilde C: [0,T] \times H \to U$  is continuous on $[0,T] \times H $,  there exists $K_T$ such that 
\begin{equation}\label{qre}
 |\widetilde C(t,x)|_U \le K_T (1 + |x|_H),\;\;\; t \in [0,T], \; x \in H,
\end{equation} 
 and moreover, for any ball $B = B(z,r) $, $z \in H$, $r>0$, the function $\widetilde C : [0,T] \times B \to U$ is $\beta$-H\"older continuous, uniformly in $t \in [0,T]$ (the index $\beta \in (0,1)$ should be the same for any ball $B$ but the H\"older norm may depend on the ball we consider). 
\end{remark}

\begin{remark} \label{strong} By Theorem  \ref{teo:abs_uni1}, using   a generalization of  the Yamada-Watanabe theorem 
 (see \cite{Ondre04} and \cite{LR15}), one deduces that equation \eqref{nonlinear_stochastic_eq_intro}
has a unique {  strong mild solution,}   {\sl for  any $x \in H$.} 
\end{remark}

\section{Analytic results on the associated Kolmogorov equation \eqref{d5} }
\label{sec:OU_smgr_reg}

In this section, assuming Hypothesis  \ref{hyp_1} we give {\it  sufficient }conditions on {\ $A,G$ and $\widetilde C$} such that Hypothesis \ref{hyp_2_1} is satisfied. This will imply the pathwise uniqueness result for equation \eqref{abs_1nonlinear_stochastic_eq} according to Theorem \ref{teo:abs_uni1}.

{ We split this section into two parts: in the former we provide preliminaries results on  the equation \eqref{d5_A_0} which involves a generator $A_0$ of a strongly continuous semigroup $e^{tA_0}$ on $H$;
in the second part we will consider Hypothesis \ref{hyp_2_1}.
}

\subsection{Preliminary results on equation \eqref{d5_A_0}
} 
\label{subsec:abs_OUsemigroup_0}

 Let us assume the following condition on the operator $Q_t$, $t>0$, introduced in \eqref{abstr_conv_stocastica}.
\begin{hypothesis}[Controllability]
\label{hyp_22} 
\begin{enumerate}[(i)]
\item For any $t>0$ we have ${\rm Im}(e^{tA})\subseteq {\rm Im}(Q_t^{1/2})$. 
\item For any $t>0$ we set $\Gamma(t):=Q^{-1/2}_te^{tA}:H\rightarrow H$. From $(i)$ it follows that $\Gamma(t)$ is a bounded linear operator for any $t>0$. We assume that there exist a positive  constant $C=C_T$ and measurable functions $\Lambda_1,\,\Lambda_2:(0,T]\rightarrow \R_+$ such that { $\inf_{t\in(0,T]}\Lambda_1(t), \inf_{t\in(0,T]}\Lambda_2(t)>0$;}  further,  for any $ \varepsilon\in(0,T)$, $\Lambda_1,\,\Lambda_2$ are bounded in the interval $[\varepsilon,T]$, and
\begin{align}
\label{abs_Gamma_OU}
|\Gamma(t)z|_H & \leq C_T\Lambda_1(t)|z|_H,\ \  |\Gamma(t)Gk|_{H}   \leq C_T\Lambda_2(t)|k|_U,  \ \ 
 z\in H, \ k\in U,  \ t\in(0,T].
\end{align}
\end{enumerate}
Note that if $U=H$  and $G =I$  we consider $\Lambda_1 = \Lambda_2$. 
\end{hypothesis}
\begin{remark}
\label{rmk_Lamba_1-Lambda_2}  
Since $G\in L(U;H)$ it follows that the best choice of $\Lambda_1$ and of $\Lambda_2$ in \eqref{abs_Gamma_OU} and a suitable choice of $C_T$ give $\Lambda_2(t)\leq \Lambda_1(t)$ for any $t\in(0,T]$.
\end{remark}
It is well known that for any $t>0$ condition ${\rm Im}(e^{tA})\subseteq {\rm Im}(Q_t^{1/2})$ is related to the null-controllability of the abstract controlled equation
\begin{align*}
\left\{
\begin{array}{ll}
\dot{Y}(t)=AY(t)+Gu(t), & t\in[0,T], \\
Y(0)=y\in H.
\end{array}
\right.
\end{align*}

In the sequel we will also need the following assumption.

\begin{hypothesis}
\label{hyp_3}
\noindent  The function  { $\Lambda_1^{1-\beta}\Lambda_2\in L^1(0,T)$, where $\beta\in(0,1)$ is the constant in Hypothesis \ref{hyp_1}$(iv)$ (cf. \eqref{abs_Gamma_OU}).  }   
\end{hypothesis}
By the dominated convergence  theorem the previous assumption implies that  
\begin{equation}\label{int_L1L2beta}
 C_{\gamma, T}:=\int_0^Te^{-\gamma t} \left(\Lambda_1(t)\right)^{1-\beta}\Lambda_2(t)\,dt \rightarrow0, \quad \gamma\rightarrow+\infty.
\end{equation}

\begin{remark}
\label{rmk:integrability} The function $\Lambda_1(t)=t^{-\sigma_1},\,\Lambda_2(t)=t^{-\sigma_2},$ with $\sigma_1>0$ and $\sigma_2\in(0,1)$ satisfy hypotheses \ref{hyp_22} and \ref{hyp_3} if $\beta>\max\{0,(\sigma_1)^{-1}(\sigma_1+\sigma_2-1)\}$. Indeed since $\sigma_2\in(0,1)$ it follows that $\sigma_1+\sigma_2-1<\sigma_1$, which implies that $\frac{\sigma_1+\sigma_2-1}{\sigma_1}<1$. Further, with the condition $\beta>\max\{0,(\sigma_1)^{-1}(\sigma_1+\sigma_2-1)\}$ the product $\Lambda_1^{1-\beta}(t)\Lambda_2(t)=t^{-((1-\beta)\sigma_1+\sigma_2)}$ is integrable as required in \eqref{int_L1L2beta}, since $(1-\beta)\sigma_1+\sigma_2<1$. 
\end{remark} 

{We recall that $(R_t)$ is the Ornstein-Uhlenbeck semigroup defined by $R_t[\Phi](x):=\mathbb E[\Phi(\Xi_t^{0,x})]$, for any $\Phi\in B_b(H;H)$, any $t\geq0$ and any $x\in H$.} Under Hypothesis \ref{hyp_22}, from \cite[Theorem 9.26]{DPsecond} we infer that for any $\Phi\in B_b(H;H)$ the map $x\mapsto R_t[\Phi](x)\in C^\infty_b(H;H)$ and
\begin{align}
\label{abs_gateaux_derivative}
\nabla_kR_t[\Phi](x)
= & \int_H\langle \Gamma(t)k,Q_t^{-1/2}y\rangle_H\Phi(e^{tA}x+y)\mathcal N(0,Q_t)(dy), 
\end{align}
for any $x,k\in H$ and any $t>0$, where $\mathcal N(0,Q_t)$ is the Gaussian measure on $H$ with mean $0$ and covariance operator $Q_t$ (see for instance \cite[Chapter 1]{DP3}). Estimates \eqref{abs_Gamma_OU} allow us to repeat verbatim the statements and the proofs of \cite[Lemmata 4.1-4.3]{MasPri17}, and we collect these results in a unique Lemma. 
\begin{lemma}
\label{abs_lemma_collezione}
Let Hypotheses \ref{hyp_1} and \ref{hyp_22} hold true, and let $R_t$ be the Ornstein-Uhlenbeck operator defined in Section \ref{sub:abstract_equation}, acting on vector-valued functions. 
\begin{itemize}
\item[(i)] 
For any $\Phi\in B_b(H;H)$ and any $t>0$, the function $x\mapsto R_t[\Phi](x)$ is G\^ateaux differentiable and its G\^ateaux derivative $\nabla R_t[\Phi](x)\in L(H;H)$ is given by \eqref{abs_gateaux_derivative}. Moreover, for any $T>0$ there exists a positive constant $C=C_T$ such that 
\begin{align}
\label{abs_stima_OU_grad_completa}
\sup_{x\in H}|\nabla_z R_t[\Phi](x) |_H& \leq
C\|\Phi\|_\infty\Lambda_1(t)|z|_H, \quad z\in H\\   
\label{abs_stima_OU_grad_G}
\sup_{x\in H}|\nabla^G_k R_t[\Phi](x) |_H& \leq
C\|\Phi\|_\infty\Lambda_2(t)|k|_U, \quad k\in U.
\end{align}
If $\Phi\in C_b(H;H)$ then $R_t[\Phi]$ is Fr\'echet differentiable on $H$ and we have $\nabla R_t[\Phi]\in C_b(H;L(H;H))$ and $\nabla^GR_t[\Phi]\in C_b(H;L(U;H))$.
\item[(iii)]
For any $\Phi\in C_b(H;H)$, any $t>0$ and any $k\in U$ the function $x\mapsto \nabla^G_k R_t[\Phi](x)$ is Fr\'echet differentiable on $H$ and
\begin{align}
\nabla_y& \nabla^G_kR_t[\Phi](x) \notag \\
\label{abs_second_order_derivatives}
= & \int_H\left(\langle \Gamma(t)y,Q^{-1/2}z\rangle_H\langle \Gamma(t)Gk,Q^{-1/2}z\rangle_H-\langle \Gamma(t)y,\Gamma(t)Gk\rangle_H\right)\Phi(e^{tA}x+z)\mathscr N(0,Q_t)(dz),
\end{align}
\begin{align}
\label{abs_stima_der_seconda_direz}
 \sup_{x\in H} |\nabla_y\nabla^G_{k}R_t[\Phi](x)|_{H}\leq C\|\Phi\|_\infty\Lambda_1(t)\Lambda_2(t)|y|_H|k|_U, \quad t\in(0,T],
\end{align}
and 
\begin{align}
\lim_{x\rightarrow 0}& \sup_{y\in H}\|\nabla_\cdot\nabla^G_kR_t[\Phi](x+y)-\nabla_\cdot\nabla_k^GR_t[\Phi](y)\|_{L(H;H)} \notag \\
= & \lim_{x\rightarrow 0}\sup_{y\in H}\sup_{|z|_H=1}|\nabla_z\nabla^G_kR_t[\Phi](x+y)-\nabla_z\nabla_k^GR_t[\Phi](y)|_{H}=0, \quad k\in U.
\label{abs_convergenza_derivate_seconde}
\end{align}
\end{itemize}
\end{lemma}
  
The last result we need follows from interpolation theory. Let $\beta\in(0,1)$. From \cite[Theorem 2.3.3 \& example 2.3.4]{DP3} it follows that
\begin{align}
\label{abs_interpolation_result}
(C_b(H),C_b^1(H))_{\beta,\infty}=C_b^\beta(H), \quad \beta\in(0,1),
\end{align}
with equivalence of the norms {(here, $(X,Y)_{\beta,\infty}$ denotes the real interpolation space between the Banach spaces $X$ and $Y$, for more details see \cite{Lun})}. Further, we denote by $(\calr_t)$ the transition semigroup of the Ornstein Uhlenbeck process $\Xi^{0,x}$ acting on Borel measurable real valued functions $\phi: H\rightarrow \R$ as
$$
\calr_t [\phi](x)=\mathbb{E}[\phi(\Xi^{0,x}_t)].
$$
There is a link between the $H$-valued Ornstein-Uhlenbeck transition semigroup $(R_t)_{t\geq0}$ and the scalar Ornstein-Uhlenbeck transition semigroup $(\calr_t)_{t\geq0}$: for any $\Phi\in B_b(H;H)$ and $h\in H$ we set $\Phi_h(x):=\langle \Phi(x),h\rangle_H$ for any $x\in H$. From \cite[Section 3]{dapr-fl} it follows that 
\begin{align}
\label{abs_smgr_sc_smgr_vet}
\langle \nabla_y R_t[\Phi](x), h \rangle_H=\nabla_y\calr_t[\Phi_h](x), \quad t>0, \ x,y,h\in H.
\end{align}

With computations similar to the ones in the proof of \cite[Lemma 4.4]{MasPri17} 
we can prove the following result. For reader's convenience we provide a detailed proof in Appendix \ref{appendix B}. 

\begin{lemma}
\label{abs_lem_interpolazione}
Let Hypotheses \ref{hyp_1} and \ref{hyp_22} be satisfied. Then, for any $T>0$ there exists a positive constant $C=C_T$ such that for any $\beta\in(0,1)$ any $\Phi\in C_b^\beta(H;H)$ we have
\begin{align}
\label{abs_stima_holder_1}
\sup_{x\in H}|\nabla_y R_t[\Phi](x)|_H\leq 
& C\|\Phi\|_{C^\beta(H;H)} \Lambda_1^{1-\beta}(t)|y|_H, \quad y\in H, \\
\label{abs_stima_holder_3}
\sup_{x\in H}|\nabla_y\nabla^G_k R_t[\Phi](x)|_{H}
\leq &  {C\|\Phi\|_{C^\beta(H;H)}}{\Lambda_1(t)^{1-\beta}\Lambda_2(t)}|y|_H|k|_U, \quad y\in H, \ k\in U,
\end{align}
for any $ t\in(0,T]$.
\end{lemma}

\begin{remark}
\label{rmk:sc_smbr_prop}
We recall that estimates \eqref{abs_stima_OU_grad_completa}, \eqref{abs_stima_OU_grad_G}, \eqref{abs_stima_der_seconda_direz}, \eqref{abs_stima_holder_1} and \eqref{abs_stima_holder_3} hold true with $R_t$ replaced by $\calr_t$ and $\Phi$ being a real-valued function.
\end{remark} 

We go back to the integral equation \eqref{d5_A_0} and for any $\mathcal T\in[0,T]$ we introduce the spaces $E_{0}^{\mathcal T}$ and $E_{0,\gamma}^{\mathcal T}$ as follows.
\begin{definition}
\label{abs_def:E_0}
$E_0^{\mathcal T}$ is the space of functions $u\in C_b([0,\T]\times H;H)$ such that $u(t,\cdot)$ is Fr\'echet differentiable on $H$ for any $t\in[0,\mathcal T]$ and the map $\nabla u:[0,\T]\times H\to L(H,H)$ is strongly continuous and globally bounded. 
 Moreover, for any $k\in U$ and $t\in[0,\mathcal T]$ the map $x\mapsto \nabla^G_ku(t,x)$ is Fr\'echet differentiable on $H$.

For any $\gamma\geq 0$, we set
\begin{align}
\label{abs_spazio_punto_fisso}
E_{0,\gamma}^{\mathcal T}:=\{u\in E_0^{\mathcal T}:\|u\|_{\gamma,\mathcal T}<+\infty\},
\end{align}
where
\begin{align}
\notag \|u\|_{\gamma,\mathcal T}
:= & \sup_{(t,x)\in[0,\mathcal T]\times H}e^{\gamma t}|u(t,x)|_H
+ \sup_{(t,x)\in[0,\mathcal T]\times H}e^{\gamma t}\|\nabla_{} u(t,x)\|_{L(H;H)} \\
\label{abs_norma_gamma}
& 
+ \sup_{(t,x)\in[0,\mathcal T]\times H}\sup_{|k|_U=1}e^{\gamma t}\|\nabla_\cdot\nabla^G_k u(t,x)\|_{L(H;H)}.
\end{align}
\end{definition}
It is easy to prove that $E_{0,\gamma}^{\mathcal T}$ is a Banach space for any $\gamma\geq0$.       Some further properties of functions $u \in E_0^{\mathcal T} $ are collected in the next remark.

\begin{remark} \label{see}
(1)   If $u \in E_0^{\mathcal T} $, $t \in [0, \T]$ and $k \in U$ then 
\begin{equation}\label{2ww} 
 \| \nabla_k^G  u(t, \cdot ) \|_{C^{\beta}(H,H)} \le 3 \sup_{x \in H}\| \nabla_k^G  u(t, x ) \|_{H} +  \sup_{x \in H}\|\nabla_\cdot\nabla^G_k u(t,x)\|_{L(H;H)}
\end{equation}
To verify the previous inequality we write for $x \not =y$ (we have to consider $|x-y| \le 1$ and $|x-y| >1$)
\begin{gather*}
|\nabla_k^G  u(t, x )
 - \nabla_k^G  u(t, y ) |_H |x-y|^{-  \beta} \le 2  \sup_{x \in H}\| \nabla_k^G  u(t, x ) \|_{H} + \sup_{x \in H}\|\nabla_\cdot\nabla^G_k u(t,x)\|_{L(H;H)}.
 \end{gather*}
 (ii) If $u \in E_0^{\mathcal T} $ then the mapping 
\begin{equation}\label{s34}
 t \mapsto \| \nabla^G  u(t, \cdot ) \|_{C^{\beta}(H,L(U;H))}  
\end{equation}
is Borel measurable on $[0,T]$ (with values in $\R_+$). It is not difficult to prove the measurability of   $  t \mapsto  \sup_{x \in H}\| \nabla^G  u(t, x ) \|_{L(U;H)}$. In order to show that
\begin{equation}\label{meas}
t \mapsto [ \nabla^G  u(t, \cdot ) ]_{C^{\beta}(H,L(U;H))}  \;\; \text{is measurable},
\end{equation}
we consider a countable dense subset $D$ of $\{ u \in U \, : \, |u|_U =1 \}$. Let $S$ be a countable dense subset of $H$. 
We note that by the continuity property of $\nabla^G_k  u$
\begin{gather*}
[ \nabla^G  u(t, \cdot ) ]_{C^{\beta}(H,L(U;H))} =
\sup_{x, y \in H, \, x \not = y}  \sup_{|k|_U = 1} \frac{|\nabla^G_k  u(t, x ) -  \nabla^G_k  u(t, y )|_H}{|x-y|^{\beta}_H}
\\ 
= \sup_{x, y \in S, \, x \not = y}  \sup_{k \in D} \frac{|\nabla^G_k  u(t, x ) -  \nabla^G_k  u(t, y )|_H}{|x-y|^{\beta}_H}. 
\end{gather*} 
Since for fixed $x, y \in S, \, x \not = y$, $k \in D$, the mapping:
 $ t \mapsto \frac{|\nabla^G_k  u(t, x ) -  \nabla^G_k  u(t, y )|_H}{|x-y|^{\beta}_H}$ is continuous on $[0,T]$ we get assertion \eqref{meas}.  
\end{remark}
 In the next result we will also use the H\"older continuity of $\widetilde C(t, \cdot)$, {$t\in[0,T]$.} 

\begin{theorem}
\label{teo:abs_solmildPDEn} 
Let Hypotheses \ref{hyp_1}, \ref{hyp_22} and \ref{hyp_3} hold true and let $A_0$ be the generator of a strongly continuous semigroup $e^{t A_0}$ on $H.$

Then,
there exists a unique solution $u^{\mathcal T}$ to \eqref{d5_A_0} in the sense of Definition \ref{defn:soluzione_int} which belongs to $E_{0}^{\mathcal T}$ and there exists a positive constant $M=M_T$ which only depends on $T$,  $\sup_{t \in [0,T]}\| e^{tA_0}\|_{L(H)}$   and  $\sup_{s\in[0,T]}\|\widetilde C(s,\cdot)\|_{C_b^\beta(H;U)}$  but not on $\mathcal T$, such that $\|u^{\mathcal T}\|_{0,\mathcal T}\leq M$. 
\end{theorem}
\begin{proof} 
Let us introduce the operator $\mathscr G$ defined on  $E_{0,\gamma}^{\mathcal T}$  by
\begin{align*}
\label{abs_operatore_contrazione}
(\mathscr G u)(t,x)
:= & \int_t^{\mathcal T}R_{s-t}\left[e^{(\mathcal T-s)A_0}G\widetilde C(s,\cdot)\right](x)ds
+ \int_t^{\mathcal T}R_{s-t}\left[\nabla^Gu(s,\cdot)\widetilde C(s,\cdot)\right](x)ds,
\end{align*}
for any $(t,x)\in[0,\mathcal T]\times H$, with $\gamma>0$ to be chosen. We proceed in some steps. 

\textbf{\textit{Step I.}}   We have to verify that $\mathscr G :  E_{0,\gamma}^{\mathcal T} \to E_{0,\gamma}^{\mathcal T}$. We only check the more difficult part, i.e. we only verify that if $u \in E_{0,\gamma}^{\mathcal T}$, for a fixed  $y \in H$, 
\begin{equation}\label{dqq}
 \nabla_y (\mathscr G u) :[0,\T]\times H\to H \;\;\; \text{ is  continuous}
 \;\; \text{on} \; [0,\T]\times H.
\end{equation}    
 We will only prove that   
 \begin{equation}\label{dqq1}
 \nabla_y \int_t^{\mathcal T}R_{s-t}\left[\nabla^Gu(s,\cdot)\widetilde C(s,\cdot)\right](x)ds \;\;\; \text{ is  continuous}
 \;\; \text{on} \; [0,\T]\times H,
 \end{equation}
the other term $ \nabla_y \displaystyle\int_t^{\mathcal T}R_{s-t}\left[e^{(\mathcal T-s)A_0}G\widetilde C(s,\cdot)\right](x)ds$  can be treated in a similar way.  
 
First note that  $B(s,x) := \nabla^Gu(s, x)\widetilde C(s,x)$ is a bounded continuous function on $[0, \T] \times H$ with values in $H$. 
We also define 
 $B(s,x) = 0$ for $s \ge \T$, $x \in H$.   
 Using the estimate \eqref{abs_stima_holder_1} and the fact that 
\begin{gather*}
 \int_0^T \Lambda_1^{1-\beta}(s) ds < \infty 
\end{gather*}
we  consider  the function
\begin{gather*}
 v(t,x):= \int_t^{\mathcal T} \nabla_y R_{s-t}\left[\nabla^Gu(s,\cdot)\widetilde C(s,\cdot)\right](x)ds =
   \int_0^{\T -t } \nabla_y R_{r} B(r+ t,\cdot) (x)dr, \;\;
\end{gather*}  
$ (t, x) \in [0, \T] \times H. $ It is enough to prove that $v$ is continuous on $[0, \T]  \times H$. Let us prove the continuity   at a fixed $(t_0,x_0)$. 
 We write
 \begin{gather*}
|v(t,x) - v(t_0,x_0)| \le 
 \Big | \int_0^{\T - t} \nabla_y R_{r} B(r+ t,\cdot) (x)dr - \int_0^{\T - t_0} \nabla_y R_{r} B(r+ t,\cdot) (x)dr \Big|
\\ + 
\Big |   \int_0^{\T - t_0}   \big[\nabla_y R_{r}  B(r+ t,\cdot) (x) -  \nabla_y R_{r} B(r+ t_0,\cdot) (x_0) \big]dr  
\Big|   = J_1 (t,x) + J_2 (t,x).
\end{gather*}
Now 
\begin{gather*}
  J_1 (t,x) \le \Big |   \int_{\T-t_0}^{\T-t} |\nabla_y R_{r} B(r+ t,\cdot) (x)|dr \Big |
\le C{\sup_{t\in[0,\T]}\|B(t,\cdot)\|_{C^\beta(H;H)}} |y|_H\Big |   \int_{\T-t_0}^{\T-t}   \Lambda_1^{1-\beta}(r) dr \Big | 
 \end{gather*}
and so $\lim_{t \to t_0}$ $\sup_{x \in  H}  J_1 (t,x) =0$. Concerning $J_2$ we note that, for any $r \in ]0, \T - t_0[$ the mapping   
$$
(t,x) \mapsto (\nabla_y R_{r}  B(r+ t,\cdot) (x) -  \nabla_y R_{r} B(r+ t_0,\cdot) (x_0) )
$$
$$=\int_H\langle \Gamma(r)y,Q_r^{-1/2}z\rangle_H [B(r+t , e^{tA}x+z)
- B(r+t_0 , e^{t_0 A}x_0+z)] \mathcal N(0,Q_r)(dz)
$$
verifies $\lim_{(t,x) \to (t_0, x_0)} |\nabla_y R_{r}  B(r+ t,\cdot) (x) -  \nabla_y R_{r} B(r+ t_0,\cdot) (x_0) |$ $=0$ by the dominated convergence theorem. Moreover, using the estimate \eqref{abs_stima_holder_1} and again the Lebesgue theorem we infer 
\begin{gather*}
 \lim_{(t,x) \to (t_0, x_0) } \int_0^{\T- t_0} |\nabla_y R_{r}  B(r+ t,\cdot) (x) -  \nabla_y R_{r} B(r+ t_0,\cdot) (x_0) | dr =0
\end{gather*} 
and so $ \lim_{(t,x) \to (t_0, x_0) } J_2(t,x) =0$. This shows \eqref{dqq}.

\vskip 2mm
\textbf{\textit{Step II.}}
We claim that a suitable choice of $\gamma$ implies that $\mathscr G $ is a contraction on $E_{0,\gamma}^{\mathcal T}$. For any $u_1,u_2\in E_{0,\gamma}^{\mathcal T}$ we have to estimate the difference $\|\mathscr G u_1-\mathscr G u_2\|_{\gamma,\mathcal T}$. Let us only estimate the term
\begin{align*}
e^{\gamma t}\|\nabla_{\cdot}\nabla^G_k\mathscr G u_1(t,x)-\nabla_{\cdot}\nabla^G_k\mathscr G u_2(t,x)\|_{L(H;H)}, \quad t\in[0,\mathcal T],  \ x\in H, \ k\in U,  
\end{align*}
since the other addends can be estimated in a similar way. We have, for  $ t\in[0,\mathcal T],$ $y \in H$, $|y |_H \le 1$,
\begin{gather}
\notag  
e^{\gamma t} |\nabla_{y}\nabla^G_k\mathscr G u_1(t,x)-\nabla_{y}\nabla^G_k\mathscr G u_2(t,x) |_{H} \le  
 \\
\label{abs_stima_per_punto_fisso_1}
\leq \int_t^{\mathcal T}e^{-\gamma(s-t)} |\nabla_{y} \nabla^G_kR_{s-t}\left[e^{\gamma s}\left(\nabla^G u_1-\nabla^G u_2\right)\widetilde C(s,\cdot)\right](x)   |_{H}ds.
\end{gather}
Since $u_1,u_2\in E_{0,\gamma}^{\mathcal T}$, 
the map $x\mapsto \nabla^Gu_i(s,x)\widetilde C(s,x)$ is $\beta$-H\"older continuous from $H$ into $H$, $i=1,2$, uniformly with respect to $s\in [0,\mathcal T]$, and
\begin{align*}
 \|\nabla^Gu_i(s,\cdot)\widetilde C(s,\cdot)\|_{C^\beta(H;H)}\leq 2\sup_{|k|_U =1}\|\nabla^G_ku_i(s,\cdot)\|_{C_b^\beta(U;H)}\|\widetilde C(s,\cdot)\|_{C_b^\beta(H;U)}, 
\end{align*}
for any $s\in[0,\mathcal T]$, with $i=1,2$. By applying \eqref{abs_stima_holder_3} to   \eqref{abs_stima_per_punto_fisso_1} and taking into account \eqref{dqq1} we get, uniformly in $y$ and in $ (t,x)\in[0,\mathcal T]\times H$ 
\begin{align} \notag
 \int_t^{\mathcal T} e^{-\gamma(s-t)}&  | \nabla_y \nabla^G_kR_{s-t}\left[e^{\gamma s}\left(\nabla^Gu_1-\nabla^Gu_2\right)\widetilde C(s,\cdot)\right](x) |_{H}  ds 
 \\ \notag
 \leq & C\int_t^{\mathcal T}e^{-\gamma(s-t)}\left(\Lambda_1(t-s)\right)^{1-\beta}\Lambda_2(s-t)\,ds \, \cdot  \\ \notag
 & \sup_{r\in[0,T]}\|\widetilde C(r,\cdot)\|_{C_b^\beta(H;U)}\sup_{r\in[0,\mathcal T]}e^{\gamma r}\sup_{k\in U, \; |k|_U=1}  \|\nabla^G_k(u_1-u_2)(r,\cdot)\|_{C_b^\beta(H;H)} 
 \\ \label{miao}
 \leq&  C_{\gamma,\mathcal T}\sup_{r\in[0,T]}\|\widetilde C(r,\cdot)\|_{C_b^\beta(H;U)} \|u_1-u_2\|_{\gamma,\mathcal T}
\end{align} 
(see also \eqref{2ww})  where, from Hypothesis \ref{hyp_3}, $C_{\gamma,\mathcal T}$ is a positive constant which goes to $0$ as $\gamma\rightarrow +\infty$, uniformly with respect to $\mathcal T\in[0,T]$. We get
\begin{gather*}  
 e^{\gamma t}\|\nabla_{\cdot}\nabla^G_k\mathscr G_nu_1(t,x)-\nabla_{\cdot}\nabla^G_k\mathscr G_nu_2(t,x)\|_{L(H;H)} \leq  C_{\gamma,\mathcal T}\sup_{r\in[0,T]}\|\widetilde C(r,\cdot)\|_{C_b^\beta(H;U)}\|u_1-u_2\|_{\gamma,\mathcal T}
\end{gather*}
Similar arguments applied to the other terms of the norm $\|\mathcal G u_1-\mathcal G u_2\|_{\gamma,\mathcal T}$ give
\begin{align*}
\|\mathscr G u_1-\mathscr G u_2\|_{\gamma,\mathcal T}
\leq C_{\gamma,\mathcal  T}\|u_1-u_2\|_{\gamma,\mathcal T},
\end{align*}
 Choosing $\gamma$ large enough we deduce that $\mathscr G $ is a contraction on $E_{0,\gamma}^{\mathcal T}$ and therefore it admits a unique fixed point $u^{\mathcal T}$. 


\vskip 2mm \textbf{\textit{ Step III.}} 
 Let us prove the last part of the statement.  We will estimate the crucial term
\begin{align*}
 \|\nabla_{\cdot}\nabla^G_k  u^{\T}(t,x) \|_{L(H;H)}, \quad t\in[0,\mathcal T],  \ x\in H, \ k\in U,  
\end{align*}
  with $|k|_U=1$, since the other addends can be estimated in a similar way.
 We have, using \eqref{2ww}, for any $t \in [0,\T]$,
 \begin{equation*}    
 \| \nabla^G  u_{}^{\mathcal T}(t, \cdot ) \|_{C^{\beta}(H;U  )} \le 3 \sup_{x \in H}\| \nabla^G  u^{\mathcal T} (t, x ) \|_{U} +  \sup_{x \in H} \sup_{|w|_U =1}\|\nabla_\cdot\nabla^G_w u^{\mathcal T}(t,x)\|_{H}.
 \end{equation*} 
 Hence, starting from 
 \begin{align*}
   u^{\T}(t,x)
 := & \int_t^{\mathcal T}R_{s-t}\left[e^{(\mathcal T-s)A_0}G\widetilde C(s,\cdot)\right](x)ds
 + \int_t^{\mathcal T}R_{s-t}\left[\nabla^G u^{\T}(s,\cdot)\widetilde C(s,\cdot)\right](x)ds
 \end{align*}
 and arguing as before (using also that  $\inf_{t\in(0,T]}\Lambda_2(t)>0$) we arrive at
 \begin{align*}
 \| \nabla^G  u_{}^{\mathcal T}(t, \cdot ) \|_{C^{\beta}(H;U  )}
 \leq M_1+M_1\int_t^{\mathcal T}h(s-t) \,\| \nabla^G  u_{}^{\mathcal T}(s, \cdot ) \|_{C^{\beta}(H,U  )} ds , \quad t\in[0,\mathcal T],   
 \end{align*}
 where the function $r\mapsto h(r) = \Lambda_1(r)^{1-\beta}\Lambda_2(r)\in   L^1(0,T)$ by Hypothesis \ref{hyp_3}   and $M_1$  is a  positive constant which depend on $T$, $\sup_{t \in [0,T]}\| e^{tA_0}\|_{L(H)}$ and  $\sup_{s\in[0,T]}\|\widetilde C(s,\cdot)\|_{C_b^\beta(H;U)}$, but not on  $\mathcal T$. The generalized Gronwall's lemma \ref{lem:gron_mod} gives  
 \begin{gather*} 
   \| \nabla^G  u_{}^{\mathcal T}(t, \cdot ) \|_{C^{\beta}(H;U  )} 
 \leq M_1[ 1+ \exp\left(\|h\|_{L^1(0,T)}\right ) \, \|h\|_{L^1(0,T)}],    
 \quad t\in[0,\mathcal T].      
 \end{gather*}
 Using the previous estimate we can bound 
 $|\nabla_{y}\nabla^G_k  u^{\T}(t,x) |_{H}$
 for any $y \in H$, $|y|_H \le 1$, $|k|_U =1$, arguing as in \eqref{miao}. We obtain 
  \begin{gather*}
  \sup_{(t,x) \in [0, \T] \times H}\|\nabla_{\cdot}\nabla^G_k  u^{\T}(t,x) \|_{L(H;H)} 
\leq \tilde M.       
\end{gather*}
 Arguing in a similar way, we obtain 
 \begin{gather*} 
\|u^{\mathcal T}\|_{0,\mathcal T} \le   \tilde M_1 
\end{gather*}
with $\tilde M_1$  depending on $T$, $\sup_{t \in [0,T]}\| e^{tA_0}\|_{L(H)}$,  $\sup_{s\in[0,T]}\|\widetilde C(s,\cdot)\|_{C_b^\beta(H;U)}$, $\Lambda_1$ and $\Lambda_2$  but not on  $\mathcal T$.  
\end{proof}

 \begin{corollary}
\label{teo:abs_solmildPDEnq1} 
Let Hypotheses \ref{hyp_1}, \ref{hyp_22} and \ref{hyp_3} hold true. Consider,
 for any $n\in\N$, the  operator $A_n$ which generates a {\rm strongly continuous  group} of linear and bounded operators $(e^{tA_n})\subset L(H)$ such that for any $T>0$ we have
\begin{equation}\label{d333}
 \sup_{t \in [0,T]} \sup_{n \ge 1} \| e^{t A_n}\|_{L(H)} < \infty, \quad 
\lim_{n \to \infty} e^{t A_n} x = e^{t A} x, \;\; x \in H,\;\; t \ge 0.
\end{equation}
(cf. (A) in Hypothesis \ref{hyp_2_1}). Then,
there exist  unique solutions $u_n^{\mathcal T}$ to \eqref{d5}   which belong to $E_{0}^{\mathcal T}$ and there exists a positive constant $M=M_T$, 
independent of $n$ and $\mathcal T$, 
such that $\|u_n^{\mathcal T}\|_{0,\mathcal T}\leq M$.
\end{corollary}

\begin{remark} \label{ma} Note that the previous result implies the validity of \eqref{sg} in Hypothesis \ref{hyp_2_1} for any choice of generators $(A_n)$ verifying assertion (A) in Hypothesis \ref{hyp_2_1}.  
\end{remark}

\vskip 1mm In the next two sections  we provide sufficient conditions for
the validity of \eqref{df} in  Hypothesis \ref{hyp_2_1}. We will consider 
two different types of approximations $(A_n)_{n\in\N}$  for $A$: the Yosida approximations which we use to treat  semilinear  stochastic damped equations and the 
finite dimensional approximations which we use to deal with  semilinear stochastic heat equations. 

\subsection{Sufficient conditions to ensure Hypothesis \ref{hyp_2_1},
 using  the Yosida approximations for $A$}
 \label{sub:yosida_appr_hilb_sc_norm}  
   

Let us show that the solutions $u_n^{\mathcal T}$ of \eqref{d5}  satisfies \eqref{df} of Hypothesis \ref{hyp_2_1} when 
 \begin{equation} \label{yosida}
A_n:=nAR(n,A)
 \end{equation}
for any $n\in\N$. We notice that with this choice of  $A_n$ Hypothesis \ref{hyp_2_1}, part $(A)$, is fulfilled (cf Remark \ref{ma}). We introduce  the following additional assumption which will be verified in Section \ref{Sec:stoc_damp_we} for the damped equation.
\begin{hypothesis}
\label{hyp:sigma_3}
For any $T>0$ there exists a positive constant $C=C_T$ such that
\begin{align}
\label{stima_hyp_HS}
\|\Gamma(t)G\|_{L_2(U;H)}\leq \Lambda_2(t){C_T}, \quad t\in(0,T],
\end{align}
where $\Lambda_2(t)$ is the function introduced in Hypothesis   \ref{hyp_22}.
\end{hypothesis}

The first estimate of the Hilbert-Schmidt norm of $u_n^{\mathcal T}$ follows from the following lemma.
\begin{lemma} \label{ci2}
Let Hypotheses \ref{hyp_1}, \ref{hyp_22} and \ref{hyp:sigma_3} hold true. Then:
\begin{itemize}
\item[(i)] for any $\Phi\in B_b(H;H)$ we have $\nabla^G_\cdot R_t[\Phi](x)\in L_2(U;H)$, $x\in H$, $t>0$, and for any $T>0$ there exists a positive constant $C=C_T$ such that
\begin{align}
\label{abs_stima_OU_HS}
\sup_{x\in H}\|\nabla^G_\cdot R_t[\Phi](x)\|_{L_2(U;H)}\leq \Lambda_2(t){C\|\Phi\|_\infty}, \quad t\in(0,T].   
\end{align}
If $\Phi\in C_b(H;H)$ then $\nabla^GR_t[\Phi]\in C_b(H;L_2(U;H))$.
\item [(ii)] The map $U\ni k\mapsto \nabla_y\nabla^G_k R_t[\Phi](x)\in L_2(U;H)$ and for any $T>0$ there exists a positive constant $C=C_T$ such that (cf. Hypothesis \ref{hyp_22})) 
\begin{align}
\label{abs_stima_HS_derivata_seconda}  
\sup_{x\in H}\|\nabla_y\nabla^G_{\cdot}R_t[\Phi](x)\|_{L_2(U;H)}\leq \Lambda_1(t)\Lambda_2(t){C\|\Phi\|_\infty}|y|_H, \quad t\in(0,T].
\end{align}
\item[(iii)] For any $T>0$ there exists a positive constant $C=C_T>0$ such that for any $\beta\in(0,1)$ we have
\begin{align}
\label{stima_HS_der_second_holder}
\sup_{x\in H}\|\nabla_y\nabla^G_{\cdot}R_t[\Phi](x)\|_{L_2(U;H)}\leq\left( \Lambda_1(t)\right)^{1-\beta}\Lambda_2(t){C\|\Phi\|_{C^\beta(H;H)}}|y|_H, \quad t\in(0,T],
\end{align}
for any $\Phi\in C_b^\beta(H;H)$.
\end{itemize}
\end{lemma}
\begin{proof}
Let $T>0$ and let $\{e_k:k\in\N\}$ be an orthonormal basis of $U$. From \eqref{abs_gateaux_derivative} and \eqref{stima_hyp_HS} we get
\begin{align*}
\|\nabla^G_\cdot R_t[\Phi](x)\|_{L_2(U;H)}^2
= & \sum_{k\in\N}|\nabla R_t[\Phi](x)Ge_k|_H^2
\leq C\|\Phi\|_\infty\sum_{k\in\N}|\Gamma(t)Ge_k|_H^2
\leq \Lambda_2(t){C\|\Phi\|_\infty},
\end{align*}
for any $\Phi\in C_b(H;H)$, any $x\in H$ and any $t\in(0,T]$, and $(i)$ follows.

\noindent To prove $(ii)$ it is enough to consider \eqref{abs_second_order_derivatives} and \eqref{stima_hyp_HS}, and to argue as in the proof of $(i)$. 

\noindent It remains to prove $(iii)$. Analogous computations as for \eqref{abs_stima_holder_3} in the proof of Lemma \ref{abs_lem_interpolazione} (see Appendix \ref{appendix B}) give
\begin{align*}
|\nabla_y(\nabla_{k}^G R_t[\Phi])(x)|_{H}
 \leq (\Lambda_1(t))^{1-\beta}{C\|\Phi\|_{C^\beta(H;H)}}|y|_H|\Gamma(t)Gk|_H, \quad t\in(0,T], \quad \phi\in C_b^\beta(H),
\end{align*}
for any $T>0$, anu $\beta\in(0,1)$ and any $\Phi\in C_b^\beta(H;H)$, where $C=C_T$ is a positive constant which only depends on $T$. To conclude, let us consider an orthonormal basis $\{e_k:k\in\N\}$ of $\N$. It follows that, see also the calculations \eqref{drr} 
\begin{gather*} 
\|\nabla_y(\nabla_{\cdot}^GR_t[\Phi])(x)\|_{L_2(U;H)}^2 
 =  \sum_{k\in\N}|\nabla_y(\nabla_{f_k}^GR_t[\Phi])(x)|_{H}^2 \\
\leq  \Lambda_1(t)^{2-2\beta}{C^2\|\phi\|_{C^\beta(H;H)}^2}|y|_H^2\sum_{k\in\N}|\Gamma(t)G f_k|_H^2 \\ 
=   \Lambda_1(t)^{2-2\beta}{C^2\|\phi\|_{C^\beta(H;H)}^2}|y|_H^2\|\Gamma(t)G\|_{L_2(U;H)}^2 
\leq    \Lambda_1(t)^{2-2\beta}\Lambda_2(t)^2{C^2\|\phi\|_{C^\beta(H;H)}^2}|y|_H^2,
\end{gather*}
which gives the thesis. 
\end{proof}
 In the next Theorem we investigate further properties of  $u_n^{\mathcal T}$, the solutions to \eqref{d5} with $A_n = nA R(n,A)$; see Corollary \ref{teo:abs_solmildPDEnq1}. Note that \eqref{abs_stima_der_sec} gives \eqref{dfdf} with $h =c$.  
\begin{theorem}
\label{thm:HS_wave_equations}
Let Hypotheses \ref{hyp_1}, \ref{hyp_22}, \ref{hyp_3} and \ref{hyp:sigma_3} hold true, and let   $u_n^{\mathcal T}$ be the solutions to \eqref{d5} with $A_n = n A R(n,A)$.
\newline Then,  $\nabla^G u_n^{\T}\in C_b([0,\T]\times H;L_2(U;H))$. Further, for any $t\in[0,\mathcal T]$ and any $x,y\in H$, the map $U\ni k\mapsto \nabla_y\nabla^G_ku^{\mathcal T}_n(t,x)$ belongs to $L_2(U;H)$ 
and there exists a positive constant $c=c(T)$ which depends on $T$ but neither on $\mathcal T$ nor on $n$ such that
\begin{align}
\label{abs_stima_der_sec}
\sup_{(t,x)\in[0,\mathcal T]\times H}\|\nabla_y\nabla^G_{\cdot}u^{\mathcal T}_n(t,x)\|_{L_2(U;H)}\leq c|y|_H, \quad y\in H.
\end{align}   
\end{theorem}
\begin{proof} The fact that  for each $t \in [0,T]$, $\nabla^G u_n^{\T} (t, \cdot) \in C_b(H, L_2(U;H))$ follows by (i) in Lemma  \ref{ci2} taking into account that
 $B(s,x) := \nabla^Gu(s, x)\widetilde C(s,x)$ is a bounded continuous function on $[0, \T] \times H$ with values in $H$.
    
Arguing as in the  first step of the proof of Theorem \ref{teo:abs_solmildPDEn} one can show that 
  $$   
  \nabla^G u_n^{\T} : [0,\T] \times H \to L_2(U;H) \;\; \text{is continuous and bounded.}
$$  
From Theorem \ref{teo:abs_solmildPDEn} and estimate \eqref{stima_HS_der_second_holder} we infer (see also \eqref{2ww})     
\begin{align}
\label{stimaHS_der_seconde_mild_sol}
\|\nabla_y\nabla^G_\cdot u_n^{\mathcal T}(t,x)\|_{L_2(U;H)}\leq C_1|y|_H+C_2\|u_n^{\mathcal T}\|_{0,\mathcal T}|y|_H\int_t^{\mathcal T}{ \left(\Lambda_1(s)\right)^{1-\beta}\Lambda_2(s)} ds, \quad y\in H, 
\end{align}
where $C_1$ and $C_2$ are positive constants which depend on $T$, $K_T$ in \eqref{d3} and  $\sup_{s\in[0,T]}$  $\|\widetilde C(s,\cdot)\|_{C_b^\beta(H;U)}$, but neither on $n$ and $\mathcal T$.  Corollary   \ref{teo:abs_solmildPDEnq1} and \eqref{stimaHS_der_seconde_mild_sol} give the thesis.
\end{proof}

\subsection{Sufficient conditions to ensure Hypothesis \ref{hyp_2_1},
 using  the   finite dimensional approximations for $A$ }
\label{subsec:abs_OUsemigroup_finite} 
Here we assume that 
$$U = H$$ 
 and Hypotheses \ref{hyp_1}, \ref{hyp_22} and \ref{hyp_3}, where in particular it is assumed that $\widetilde C (t, \cdot ) \in  C_b^\beta(H;H)$  for some $0<\beta<1$ uniformly in $t \in [0,T]$ (see \eqref{ey}). Moreover we require   
 the following condition:
\begin{hypothesis}\label{ip_finite}    
\begin{enumerate}  
\item $A$ is self-adjoint, with compact resolvent, $\{e_n:n\in\N\}$ is a complete orthonormal system in $H$ which satisfies $Ae_n=-\alpha_n e_n$, with non-decreasing positive $(\alpha_n)_{n\geq 1}$. 
\item
 We require    $G\in L_2(H)$ or, setting $(\widetilde C)_n:=\< \widetilde C,e_n\>$,
\begin{equation}
\label{sommatoria_k_beta}
\sum_{n=1}^\infty\frac{\sup_{t\in(0,T)}\Vert (\widetilde C(t,\cdot))_n\Vert_{C^\beta(H;\R)}^2}{\alpha_n}<\infty; 
\end{equation} 
\end{enumerate}
\end{hypothesis} 
 \begin{remark}
 \label{rmk:diff_hyp_heat}
We point out that  Hypotheses \ref{hyp_1},   \ref{hyp_22}, \ref{hyp_3} and  \ref{ip_finite}  extend  assumptions $1-6$ in \cite{dapr-fl} in the following way. 

 (i) Assumption (ii) in Hypothesis \ref{hyp_3}      is weaker than assumption $6$ in \cite{dapr-fl} (such assumption $6$ corresponds to the case when $\Lambda_1=\Lambda_2$). Recall that, in general we have $\Lambda_2\leq\Lambda_1$ (see Remark \ref{rmk_Lamba_1-Lambda_2}). 
The main consequence of this fact is that our results apply to semilinear stochastic heat equation in dimension $d=3$ (see Section \ref{sub:heat_equation}), while examples in \cite{dapr-fl} only cover the cases $d=1$ and $d=2$.

\vskip 1mm (ii) Following   \cite{dapr-fl} one should  require a condition like  $
\int_0^T\left(\Lambda_1(t)\right)^{\beta}\Lambda_2(t)dt <+\infty.$
 However we will not impose such condition. 
\end{remark}  

\begin{remark}
In this section we consider the case when $G$ is not necessarily a trace class operator and \eqref{sommatoria_k_beta} holds true, since if $G\in L_2(H)$ then \eqref{stima_hyp_HS} is satisfied with $U$ replaced by $H$. Indeed, for any $\{h_n:n\in\N\}$ orthonormal basis of $H$ we have
\begin{align*}
\sum_{n\in\N}|\Gamma(t)Gh_n|_H^2
\leq C_T\Lambda_2(t)\sum_{n\in\N}|Gh_n|_H^2\leq C_T\Lambda_2(t)\|G\|_{L_2(H)}^2,
\end{align*}
and the estimate \eqref{abs_stima_der_sec} follows at once. This implies that condition $G\in L_2(H)$ allows to get strong uniqueness by using Yosida approximations and the computations developed in Section \ref{sub:yosida_appr_hilb_sc_norm},  but for semilinear stochastic heat equation this does not lead to the sharp result.
\end{remark}
Let $n\in\N$. We consider $E_n:={\rm span}\{e_1,...,e_n\}$ the finite dimensional linear span generated by $e_1,...,e_n$  (see Hypothesis \ref{ip_finite}) and we let $\Pi_n$ be the projection of $H$ onto $E_n$:
 \begin{equation} \label{qq}
 \Pi_n:H\rightarrow E_n, \quad x\mapsto \sum_{k=1}^n \<x,e_k\>_H e_k.
 \end{equation}
As approximants of $A$ we will consider in this section the finite dimensional truncations of $A$, given by
\begin{gather}\label{A_approx_fin}
A_n = A \Pi_n, \quad n\in\N. 
\end{gather}
 Let us notice that this family of operators satisfies Hypothesis \ref{hyp_2_1}, part $(A)$. { For any $\T\in(0,T]$} we consider the integral equation \eqref{d5} which we rewrite here for the reader's convenience: 
\begin{align}  
\label{ciao}
u(t,x)
:= & \int_t^{\mathcal T}R_{s-t}\left[e^{(\mathcal T-s)A_n}G \widetilde C(s,\cdot)\right](x)ds
+ \int_t^{\mathcal T}R_{s-t}\left[\nabla^G u(s,\cdot) \widetilde C(s,\cdot)\right](x)ds.
\end{align}
We denote by $u_n^{\T}$ the solution of this equation  (see Theorem \ref{teo:abs_solmildPDEn}). Following Remark \ref{rm:comp_PDE} 
$u^{\mathcal T}_n$ solves 
\begin{align}\label{pde_formale} 
\left\{
\begin{array}{ll}
\displaystyle \frac{\partial u (t,x)}{\partial t}+\mathcal L_t[u(t,\cdot)](x)
= -e^{(\mathcal T-t)A_n} G\widetilde C(t,x), & x\in H,\ t\in[0,\mathcal T], \\
u(\mathcal T,x)=0,  & x\in H.
\end{array}
\right.
\end{align}
where
$$\mathcal L_tf(x):=\frac12{\rm Tr}[GG^*\nabla^2f(x)]+\langle Ax,\nabla f(x)\rangle+\langle \widetilde C(t,x),\nabla^G f(x)\rangle, \quad t\in[0,T],\,x\in H.$$
 Indeed 
 in mild formulation equation \eqref{pde_formale} can be rewritten as
\begin{align}
\label{pde_formale_forward_mild}
u_{n}^{\mathcal T} (t,x)
= & \int_t^{\mathcal T}R_{s-t}\left[e^{(\T-s)A_n}G\widetilde C(s,\cdot)\right](x)ds
+ \int_t^{\mathcal T}R_{s-t}\left[\nabla^G  u_{n}^{\mathcal T}(s,\cdot) \widetilde C(s,\cdot)\right](x)ds, 
\end{align} 
For every fixed $n$ we let $u_{n,k}^{\mathcal T}:=\<u_{n}^{\mathcal T},e_k\>:[0,\T]\times H\rightarrow \R$ its $k$-component, with $k\in\N$. The following lemma states that for any $n\in\N$ we have $u_n^{\T}(t,x)\in E_n$ for any $(t,x)\in[0,\T]\times H$.

\begin{lemma}
\label{lemma:dec_unT}
Let $u_{n,k}^{\T}$ be as above. Then:
\begin{itemize}
\item[(i)] For any  $k=1,\ldots,n$ we have
\begin{align}
\label{pde_formale_forward_k_mild}
u_{n,k}^{\mathcal T} (t,x)
= & \int_t^{\mathcal T}\calr_{s-t}\left[e^{-(\T-s)\alpha_k}(G\widetilde C(s,\cdot))_k\right](x)ds
+ \int_t^{\mathcal T}\calr_{s-t}\left[\nabla^G  u_{n,k}^{\mathcal T}(s,\cdot) \widetilde C(s,\cdot)\right](x)ds,
\end{align}
\item [(ii)] For any  $k\geq n+1$ we have $u_{n,k}^{\T}=  0.$
\end{itemize}
\end{lemma}
\begin{proof}
Since
\begin{align*}
R_{s-t}\left[e^{(\T-s)A_n}G\widetilde C(s,\cdot)\right](x)
= & \Pi_n\mathbb E[e^{(\T-s)A_n}G\widetilde C(s,\Xi^{0,x}_s)], 
\end{align*} 
formula \eqref{pde_formale_forward_k_mild} follows. Further, for any $j\geq n+1$, we have
\begin{align}
\label{pde_int_repre_j_0}
u_{n,j}^{\mathcal T} (t,x)
=\int_t^{\mathcal T}\calr_{s-t}\left[\nabla^G  u_{n,j}^{\mathcal T}(s,\cdot) \widetilde C(s,\cdot)\right](x)ds.
\end{align}
Indeed
$$
 \langle R_{s-t}\left[e^{(\T-s)A_n}G\widetilde C(s,\cdot)\right](x), e_k \rangle =0,\;\;\; k \ge n+1. 
$$
From \eqref{abs_stima_OU_grad_G} and Remark \ref{rmk:sc_smbr_prop} we infer that, for $k \in H$, $|k|_H =1$, 
\begin{align*}
\|\nabla^G_k u_{n,j}^{\T}(t,\cdot)\|_{C_b(H)}\leq C\int_t^{\T}\Lambda_2(s-t) \|\nabla^G_k u_{n,j}^{\T}(s,\cdot)\|_{C_b(H)}ds \cdot  \sup_{s\in[0,T]}\|\widetilde C(s,\cdot)\|_{C_b(H;H)}. 
\end{align*}
The generalized Gronwall lemma \ref{lem:gron_mod} gives $\nabla^G u_{n,j}^{\T}(t,x)=0$ for any $t\in[0,\T]$ and any $x\in H$  and from \eqref{pde_int_repre_j_0} we get $(ii)$.
\end{proof}

We notice that, up to revert time, equations \eqref{ciao} and \eqref{pde_formale_forward_k_mild} coincide with the mild integral equations (16) and (15) in the paper \cite{dapr-fl}, respectively, with $G$ and $G_k$ which are given here by $G=e^{s A_n }G\widetilde C(\T -s , \cdot) $ and $G_k(s,\cdot) =  e^{-s \alpha_k }(G\widetilde C)_k(\T -s , \cdot)$. We stress that from  Corollary  \ref{teo:abs_solmildPDEnq1}    we already know that $u_{n,k}^{\mathcal T}\in $     
 $E_{0}^{\mathcal T}$ and there exists a positive constant $M=M_T$, 
independent of  $k=1,\ldots,n$, $\T$ and  $n$ such that 
\begin{equation}\label{qee}
 \|u_{n,k}^{\mathcal T}\|_{0,\mathcal T}\leq M.
\end{equation}  

 The next   result gives a new   estimate which is not present in  the regularity results  of Section 4 in  \cite{dapr-fl}. Indeed in such section estimates on the second derivatives of solutions  
are given using  the operator norm; instead here we consider   the  stronger Hilbert-Schmidt norm.     
  
  
\begin{theorem}
\label{prop:HS_heat}   
Let Hypotheses \ref{hyp_1}, \ref{hyp_22}, \ref{hyp_3} and \ref{ip_finite} be satisfied, { and let $\T\in(0,T]$}. Then, 
$\nabla^G u^{\mathcal T}_{n}\in B_b([0,\T]\times H;L_2(H))$ and there exists ${h :(0,\T)\rightarrow \R_+\in L^1(0,\mathcal T)}$, independent of $n$, such that for any $n\in\N$ and any $t\in(0,\mathcal T)$ we have  
\begin{gather}
 \label{dfdf2}   
 \sup_{x\in H}
 \|\nabla_y\nabla^G_{\cdot}u^{\mathcal T}_n (t,x)\|^2_{L_2({ H};H)}\leq \, 
{h(\T-t)} \, |y|^2_H, \quad y\in H.   
\end{gather} 
Further, there exists a positive constant $C=C_T$,  depending on
 $\|\Lambda_1^{1-\beta}\Lambda_2\|_{L^1(0,T)}$,  
 such that for any $\mathcal T\in(0,T]$ we have 
\begin{align} \label{www}
\|h\|_{L^1(0,\T)}
\leq & 2C\sum_{k=1}^\infty\frac{\|(\widetilde C(s,\cdot))_k\|_{C^\beta(H;H)}^2}{\alpha_k}<+\infty.
\end{align}
\end{theorem}  
\begin{remark}
\label{rmk:h}
 Note that  \eqref{dfdf2}  implies   (cf. \eqref{df}$)$ 
  \begin{equation}\label{df3}  
\sup_{x\in  H}\|\nabla^G_{\cdot}u^{\mathcal T}_n (t,x + y) -\nabla^G_{\cdot}u^{\mathcal T}_n (t,x ) \|_{L_2(H)}^2 \, \leq \, {h(\T-t)} \, |y|_H^2, \quad t\in(0,\T), \ y\in H.
\end{equation} 
\end{remark}

\begin{proof}    Let $\T\in(0,T]$ and $n \ge 1$. Let us prove that  $\nabla^G u_n^{\T}$ belongs to  $B_b([0,\T]\times H;L_2(H))$.
 We have 
\begin{gather*}
 \nabla^G u_n^{\T} (t,x) = \sum_{k=1}^n \nabla^G u_{n,k}^{\mathcal T}(t,x) e_k
\end{gather*}
(see Lemma \ref{lemma:dec_unT}).  
 Arguing as in \eqref{drr} we have
\begin{align*}
  \sum_{k \ge 1} |\nabla^G_{{ e}_k}u^{\mathcal T}_n (t,x)|^2_H
 = & \sum_{k \ge 1} |\nabla_{G {e}_k}\, 
  u^{\mathcal T}_n (t,x)|^2_H  \notag
=   \sum_{j =1}^n\sum_{k \ge 1}   \langle \nabla_{G {e}_k}\, 
 u^{\mathcal T}_n (t,x) , e_j\rangle^2_H  \\
 = & 
 \sum_{j =1}^n\sum_{k \ge 1} \, [\nabla_{G { e}_k} 
 u^ {\mathcal T }_{n,j} (t,x)]^2      
=  \sum_{j = 1}^n \, |\nabla^{G}    
 u^{{\mathcal T }}_{n,j} (t,x)|^2_{ H}.    
\end{align*}  
This shows that, for any $(t, x)$,  the map:  $k \mapsto \nabla^G_{k}u^{\mathcal T}_n (t,x)$ is a Hilbert-Schmidt operator from ${ H}$ into $H$.


For any $N \ge 1$, $(t, x ) \in [0, \T] \times H$, we introduce the   approximating mappings:
\begin{gather*}
 k \mapsto  
 \nabla^G_{\Pi_N k} \, u_n^{\T}(t,x) = F_N(t,x,k),
\end{gather*}  
{ where $\Pi_n$ has been defined in \eqref{qq}.}   By the previous calculations and   by Theorem \ref{teo:abs_solmildPDEn} we deduce  that $F_N \in C_b([0,\T]\times H;L_2(H))$ for any $N \ge 1$. Since
\begin{gather*}
\lim_{N \to \infty }  \|F_N(t,x, \cdot ) -  \nabla^G_{}u^{\mathcal T}_n (t,x) \|_{L_2(H)} =0,
\end{gather*} 
 for any $(t, x) \in [0,\T] \times H$, we get the desired measurability property.

In the sequel  $C$ is a positive constant which may vary from line to line and which does not depend on $n,k$ and $\T$. 
 In order to prove \eqref{dfdf2} we write 
 as in  \eqref{drr}
 \begin{gather*}
\|\nabla_y\nabla^G_{\cdot}u^{\mathcal T}_n (t,x)\|_{L_2(H) }^2
 =  
 \sum_{j \ge 1} \, |\nabla_y\nabla^{G}  _\cdot
 u^{{\mathcal T }}_{n,j} (t,x)|^2_{ H}.  
\end{gather*} 
By Lemma \ref{lemma:dec_unT}   assertion \eqref{dfdf2}
  follows if we prove
\begin{align}     
\label{abs_stima_der_sec_finite_new} 
\sum_{k=1}^n \sup_{x \in H}\Vert \nabla_\cdot \nabla^G_\cdot  u_{n,k}^{\mathcal T}(s,x)\Vert_{L(H;{ H})}^2\leq h^{\mathcal T}(s),\;\;\; s \in [0, \T], 
\end{align}
for any $n\in\N$.  Let us prove estimate \eqref{abs_stima_der_sec_finite_new}. 
 Arguing as in \eqref{2ww} we find, for any $t \in [0,\T]$,
\begin{equation}\label{2ww1}    
 \| \nabla^G_\cdot  u_{n,k}^{\mathcal T}(t, \cdot ) \|_{C^{\beta}(H;{ H}  )} \le 3 \sup_{x \in H}\| \nabla^G_\cdot  u_{n,k}^{\mathcal T} (t, x ) \|_{{ H}} +  \sup_{x \in H} \sup_{|w|_{ H} =1}\|\nabla_\cdot\nabla^G_w u_{n,k}^{\mathcal T}(t,x)\|_{H}.
\end{equation} 
 Let us apply $\nabla^G$ to \eqref{pde_formale_forward_k_mild}. We will  take into account the regularizing properties of $\calr_t$ (see \eqref{abs_stima_OU_grad_G}, \eqref{abs_stima_holder_3} and Remark \ref{rmk:sc_smbr_prop}), 
 and \eqref{2ww1}.  


For any $k\in\{1,\ldots,n\}$, define 
$$
U_{n,k}(t) = \| \nabla^G_\cdot  u_{n,k}^{\mathcal T}(t, \cdot ) \|_{C^{\beta}(H;{ H}  )}, \;\; t \in [0, \T].
$$
Using that  $\inf_{t\in(0,T]}\Lambda_2(t)>0$ and taking into account \eqref{pde_formale_forward_k_mild} we get 
\begin{align*}
 U_{n,k}(t) 
\leq  &\; C\sup_{s\in[0,T]}\|(\widetilde C(s,\cdot))_k\|_{C^\beta(H;H)}\int_t^{\T}e^{-(\T-r)\alpha_k}\left(\Lambda_1(r-t)\right)^{1-\beta} \Lambda_2(r-t)dr \\
& +C\int_t^{\T}\left(\Lambda_1(r-t)\right)^{1-\beta}\Lambda_2(r-t) U_{n,k}(r) dr.  
\end{align*} 
 Let 
 $$
 g(t) = \left(\Lambda_1(t)\right)^{1-\beta}\Lambda_2(t), 
 $$
 $ t \in [0, \T]$. By applying the generalized Gronwall lemma \ref{lem:gron_mod} we infer 
\begin{align}\label{stima-g}
  \| \nabla^G_\cdot  u_{n,k}^{\mathcal T}(t, \cdot ) \|_{C^{\beta}(H; { H} )} \le f(t)+
   e^{\|g \|_{L^1(0,T)}} \int_t^{\T}f(s) g(s-t) ds,
\end{align}
  for any $t\in[0,\T]$ and $k=1,\ldots,n$,  where $f(t)= C \sup_{s\in[0,T]}\|(\widetilde C(s,\cdot))_k\|_{C^\beta(H;H)} \; \cdot $ $\, I_{{\T}}^k(t) $ and 
 \begin{align}\label{Ik}
I_{{\T}}^k(t):= \int_t^{\T}e^{-(\T-r)\alpha_k} g(r-t)dr, \quad t\in[0,\T].
 \end{align}
   Let us estimate $\nabla_\cdot \nabla_{\cdot}^Gu^{\T}_{n,k}$; by the integral equation verified by $u^{\T}_{n,k}$ we get
  \begin{gather*}
  \|\nabla_{\cdot} \nabla_{\cdot}^Gu^{\T}_{n,k}(t,\cdot)\|_{C_b(H;L(H;H))}
  \leq  C\sup_{s\in[0,T]}\|(\widetilde C(s,\cdot))_k\|_{C^\beta(H;H)}\int_t^{\T}e^{-(\T-r)\alpha_k} g(r-t)dr	 \\
  +  C\int_t^{\T} g(r-t)\,   \|\nabla_{\cdot}^G u_{n,k}^{\T}(r,\cdot)\|_{C^\beta(H;H)}dr \\
  \leq   C\sup_{s\in[0,T]}\|(\widetilde C(s,\cdot))_k\|_{C^\beta(H;H)}(I_{{\T}}^k(t)+J_{{\T}}^k(t)),
  \end{gather*}
  where in the last passage we have used the definition \eqref{Ik}. Moreover setting $K_T = e^{\|g \|_{L^1(0,T)}}$ and
\begin{align*}
J_{{\T}}^k(t):=  
 \int_t^{\T} g(r-t) \Big(
 I_{{\T}}^k(r)+
   K_T \int_r^{\T} I_{{\T}}^k(\xi) g(\xi -r)d\xi  \Big)dr,
\end{align*}  
by \eqref{stima-g} we have performed the second part of the last inequality. Therefore,
\begin{align*}
\sum_{k=1}^n \|\nabla_{\cdot} \nabla_{\cdot}^Gu^{\T}_{n,k}(t,\cdot)\|_{C_b(H;L(H;H))}^2
\leq  C\sum_{k=1}^n \sup_{s\in[0,T]}\|(\widetilde C(s,\cdot))_k\|_{C^\beta(H;H)}^2 \left((I_{{\T}}^k(t))^2+(J_{{\T}}^k(t))^2\right).
\end{align*}
Let us prove that for any $k\in\N$ the functions $t\mapsto I_{\T}^k(t)$ and $t\mapsto J_{\T}^k(t)$ only depend on $\mathcal T-t$.  Indeed, setting $\mathcal T-r=s$ in formula \eqref{Ik} where $I_{{\T}}^k(t)$ is defined we get
\begin{align*}
I_{{\T}}^k(t):= \int_0^{\T-t}e^{-s\alpha_k} g(\T-s-t)ds, \quad t\in[0,\T], 
 \end{align*}
and we see that $ I_{\T}^k(t)$ depends only on $\T-t$.  Analogously setting $\mathcal T-r=s, \,\mathcal T-\xi=\eta$ in the definition of $J_{{_\T}}^k(t)$ we get
\begin{align*}
J_{{\T}}^k(t):=  
 \int_0^{\T-t} g(\T-s-t) \Big(
 I_{{\T}}^k(\T-s)+
   K_T \int_s^{\T-t} I_{{\T}}^k(\T-\eta) g(s-\eta)d\eta \Big)ds,
\end{align*}  
and we see that $ J_{\T}^k(t)$ depends only on $\T-t$.


Let us set   
\begin{align} \label{eu}
h(t):=C\sum_{k=1}^{+\infty} \sup_{s\in[0,T]}\|(\widetilde C(s,\cdot))_k\|_{C^\beta(H;H)}^2\left((I^k(t))^2+(J^k({t}))^2\right), \quad t\in(0,\T),
\end{align}
then \eqref{dfdf2} is satisfied. It remains to prove \eqref{www}. To  this purpose 
note that, for each fixed $k \ge 1$, $I^k$ and $J^k$ are bounded function on $[0,\T]$ (uniformly in $k$). Indeed  
\begin{gather*}
0 \le  I^k({\T-t})
\leq  \int_t^\T  g(r-t)dr
\leq   \|g\|_{L^1(0,T)} = \|\Lambda_1^{1-\beta}\Lambda_2\|_{L^1(0,T)}, 
\end{gather*}
and similarly, with a constant $M_0 = M_0 (\|g\|_{L^1(0,T)}) >0$:
\begin{gather*}
 0 \le J^k({\T-t}) \le M_0, \;\; t \in [0, \T].
\end{gather*}
We need more precise estimates of the $L^1$-norms of 
$(I^k({\T-t}))^2$ and $(J^k(\T-t))^2$ to get the  estimate \eqref{abs_stima_der_sec_finite_new} on  $
 \|\nabla_{\cdot} \nabla_{\cdot}^Gu^{\T}_{n,k}(t,\cdot)\|_{C_b(H;L(H;H))}
$. 

\noindent Using that  $I^k$ and $J^k$ are uniformly bounded, we concentrate on giving bounds  for 
$$
 \| I^k \|_{L^1(0,\T)}  \;\; \text{and} \;\; \| J^k \|_{L^1(0,\T)}; 
$$
by the boundednes of $I^k$ and $J^k$ this will imply estimates for  $\| (I^k)^2 \|_{L^1(0,\T)}$ and $ \| (J^k) ^2\|_{L^1(0,\T)}. $ \newline We have  by the Fubini theorem
\begin{align*}
\int_0^{\T}I^k({\T-t})dt
 &= \int_0^{\T } dt \int_t^{\T}e^{-(\T-r)\alpha_k} g(r-t)dr=  
 \int_0^{\T}e^{-(\T-r)\alpha_k}\left(\int_0^{r} g(r-t)dt\right)dr  \notag \\
& \leq  \| g\|_{L^1(0,T)}\int_0^{\T}e^{-(\T-r)\alpha_k}dr \notag  \leq \frac{ \|\Lambda_1^{1-\beta}\Lambda_2\|_{L^1(0,T)}}{\alpha_k},
 \;\; k \ge 1,
\end{align*}
 which is the required dependence on $\alpha_k$.
 On the other hand,
\begin{gather*}
\int_0^{\T} J^k({\T-t})dt = \int_0^{\T} dt \int_t^{\T} g(r-t) \Big(
 I^k({\T-r})+
   K_T \int_r^{\T} I^k({\T-\xi}) g(\xi -r)d\xi  \Big)dr
\\=
  \int_0^{\T} dt \int_t^{\T} g(r-t) 
 I^k({\T-r}) dr  +  K_T \int_0^{\T} dt \int_t^{\T} g(r-t) 
   \Big( \int_r^{\T} I^k({\T-\xi}) g(\xi -r)d\xi \Big) dr  
   \\
  = \int_0^{\T} I^k({\T-r})  dr \int_0^{r} g(r-t) dt
  +  K_T \int_0^{\T} \Big( \int_r^{\T} I^k({\T-\xi}) g(\xi -r)d\xi \Big) dr \int_0^{r} g(r-t) dt
 \\
 \le \| g\|_{L^1(0,T)} \Big ( \int_0^{\T} I^k({\T-r})  dr + K_T \int_0^{\T} dr  \int_r^{\T} I^k({\T-\xi}) g(\xi -r)d\xi   \Big ).
\end{gather*}
Hence, denoting by $C_1$ a constant depending on $ \| g\|_{L^1(0,T)}$ we find  
\begin{gather*}
\int_0^{\T} J^k(\T-t)dt
\le \| g\|_{L^1(0,T)} \Big ( \int_0^{\T} I^k({\T-r})  dr + K_T 
 \| g\|_{L^1(0,T)}
\int_0^{\T}  I^k(\T-\xi) d\xi   \Big )
\le  \frac{C_1}{\alpha_k}.
\end{gather*}
Recalling the definition of $h$ in  \eqref{eu}, collecting the previous estimates  it follows that
\begin{align*} 
\|h\|_{L^1(0,\T)}
\leq & 2C\sum_{k=1}^\infty\frac{\|(\widetilde C(s,\cdot))_k\|_{C^\beta(H;H)}^2}{\alpha_k}<+\infty.
\end{align*} 
The proof is complete.
\end{proof}

\section{Applications}\label{sec:applications}

In this section we consider two concrete models to which our results apply:  stochastic semilinear  damped {Euler-Bernoulli beam} equation and  stochastic semilinear  heat equations.

\subsection{ An example of stochastic damped {Euler-Bernoulli beam} equation}  
\label{sub:example_stoc_wave_eq}

First we consider a  nonlinear stochastic damped {Euler-Bernoulli beam} equation with the nonlocal term $\left(-\frac{\partial^2}{\partial \xi^2} \right)^{{2\alpha}}$ and hinged boundary conditions:
\begin{align}
\label{sto_damped_WE}
\left\{
\begin{array}{ll}
\frac{\partial ^2}{\partial t^2}y(t,\xi)=-\frac{\partial^{{4}}}{\partial \xi^{{4}}}y(t,\xi)-\rho\left(-\frac{\partial^2}{\partial \xi^2} \right)^{{2\alpha}}\frac{\partial}{\partial t}y(t,\xi)\\
\phantom{\frac{\partial ^2}{\partial t^2}y(t,\xi)=}
+{\left(-\frac{\partial^2}{\partial \xi^2} \right)^{{-2\gamma}}}c(t,\xi,y(t,\xi)) 
+{\left(-\frac{\partial^2}{\partial \xi^2} \right)^{{-2\gamma}}}\dot{W}(t,\xi), & (t,\xi)\in(0,T]\times (0,1), \vspace{3mm}\\
y(t,0)={y(t,1)=\frac{\partial^2}{\partial x^2}y(t,0)=\frac{\partial^2}{\partial x^2}y(t,1)=0}, & t\in(0,T], \vspace{3mm} \\
y(0,\xi)=y_0(\xi), \ \ \frac{\partial}{\partial t}y(0,\xi)=y_1(\xi)&  \xi\in[0,1],
\end{array}
\right.  
\end{align}
with $\rho>0$ and $\alpha\in[0,{1/2}]$.   Using the terminology of \cite{chen-russ}
 this equation is in the class 
stochastic {Euler-Bernoulli beam} equations  which describe elastic systems with structural damping.

Here, $y_0$ describes the initial position and $y_1$ the initial velocity of the particle, and $\dot W(\tau,\xi)$ is a space-time white noise on $[0,T]\times [0,1]$ which describes external random forces. 

In Section \ref{Sec:stoc_damp_we} we show that
equation \eqref{sto_damped_WE} can be reformulated in an abstract way as a stochastic evolution equation in {a suitable space $H$} of the form \eqref{nonlinear_stochastic_eq_intro}.

If the term $c(\cdot,\cdot,\cdot)$ satisfies the next  conditions, then we are able to apply our results to equation \eqref{sto_damped_WE} which give the pathwise uniqueness for  mild solutions.
  
\vskip 1mm   
{\it The function $c : [  0,T ]  \times [  0,1 ]
\times \R \to \R $  is measurable and,  for  $s $
$\in [  0,T ]  ,$ a.e. $\xi$ $\in [  0,1 ]  ,$ the map $c (
s,\xi,\cdot ) $ $ :\R  \rightarrow \R$ is continuous. 
 There exists $c_1$ bounded and measurable 
on $[  0,1] $, $\beta \in (0,1)$,  such that,
 for  $s\in [
0,T ]  $ and a.e. $\xi$ $ \in [  0,1 ],$   
 \begin{align} \label{w33} 
|  c (  s,\xi,x )  - c(  s,\xi,y )  |
\leq c_{1} (  \xi )  |  x-y |^\beta,
\end{align}
$ x,\,y \in \R$. Moreover
$|  c (  s,\xi,x )  |  \leq $ $ d_{2} (  \xi)
,$  for $s \in [0,T]$, $x \in \R$ and a.e. $\xi \in [0,1]$, with $d_{2} \in L^2 ([  0,1] ) $.
}

\vskip 1mm 
  
To deal with equation \eqref{sto_damped_WE} first  we have to show the well-posedness when  $c=0$, by proving that the stochastic convolution is well defined in $H$: {this is an easy consequence of the fact that $\Lambda^{-\gamma}$ is a trace class operator on $U$}.

Once that the well-posedness of the linear stochastic damped {Euler-Bernoulli beam} equation 
is proved, we investigate the regularizing effects of the associated transition semigroup $(R_t)$ (cf. Section \ref{sec:OU_smgr_reg}).  These effects can be proved by means of optimal blow-up rates for the minimal energy associated to  null controllability of  related  linear deterministic control systems. 
To this purpose we use a spectral approach to the damped elastic operators introduced in \cite{chen-russ} and recovered in \cite{lasi-trig} and \cite{trig}.
For the stochastic damped {Euler-Bernoulli beam} equation we will get  
\begin{align*}  
{\Lambda_1(t)=\Lambda_2(t)=t^{-1/2-2\gamma}}
\end{align*} 
(cf. Hypotheses  \ref{hyp_22} and \ref{hyp:sigma_3}) for every $\gamma\in(1/8,1/4)$.  
So we are able to we prove that if   
\begin{gather*}
\rho^2\neq 4 (\pi^2 n^2)^{1-2\alpha} 
\end{gather*}
and  $\beta$ given in \eqref{w33} belongs to $\left(\overline \beta,1\right)$ where {$\overline \beta=8\gamma/(1+4\gamma)$}, then {\it pathwise uniqueness holds true for equation \eqref{sto_damped_WE}.}  

\subsection{Semilinear stochastic damped {Euler-Bernoulli beam} equations {in general form}}
\label{Sec:stoc_damp_we}

\subsubsection{ {Setting and  assumptions} }
 \label{sec6:set_ass}

We consider the following nonlinear stochastic {damped beam} equation which is a general form of \eqref{sto_damped_WE}:
\begin{align} 
\label{abst_sto_damped_WE_1}
\left\{  
\begin{array}{ll}
\displaystyle \frac{\partial ^2y}{\partial t^2}(t)=-\Lambda y(t)-\rho \Lambda^\alpha\frac{\partial y}{\partial t}(t)+  {\Lambda^{-\gamma}}\tilde C\left(t,y(t), \frac{\partial y}{\partial t}(t)\right)+{\Lambda^{-\gamma}}\dot{W}_t, & t\in(0,T], \vspace{1mm}\\
y(0)=y_0, \vspace{1mm} \\
\displaystyle \frac{\partial y}{\partial t}(0)=y_1,
\end{array}
\right.
\end{align}
 with $\rho>0$ and $\alpha\in[0,{1/2}]$ {and $\gamma\in(1/8,1/4)$}. Here, $\Lambda:D(\Lambda)\subset U\rightarrow U$ is a {\it positive self-adjoint operator} on a separable Hilbert space $U$ such that 
\begin{equation}\label{lamb1}
\text{ $\Lambda^{-{2\gamma}}$ {\it which is a trace class operator from $U$ into $U$} }
\end{equation}
 and $W=\{W(\tau):\tau\geq0\}$ is a cylindrical Wiener process on $U$.

We aim at formulating this equation as an abstract stochastic  evolution equation in the product space $H:=V\times U=: D(\Lambda^{1/2})\times U$. 

About  $\tilde C: [0,T]\times H\rightarrow U$ in \eqref{abst_sto_damped_WE_1} we will  assume that it is continuous and bounded and there exists a positive constant $K$ and  $\beta\in ( 0,1)$ such that 
\begin{align} \label{ey33}
|\widetilde C(t, h)-\widetilde C (t,h')|_U\leq K|h-h'|^\beta_H, \quad h,h'\in H, \ t\in[0,T].
\end{align}
(cf. (i) in Hypothesis \ref{hyp_3}).    

{We will assume the following hypothesis.
\begin{hypothesis}
\label{hyp_non_linear_damped_1}
For every $\alpha\in[0,1/2]$ the index $\beta$ belongs to $\left(\displaystyle \frac{8\gamma}{1+4\gamma},1\right)$. 
\end{hypothesis}}
We notice that, since $\gamma<1/4$, the quantity $8\gamma/(1+4\gamma)<1$.

Let $(\mu_n)$ be the family of eigenvalues of $\Lambda$. We are assuming that  
\begin{align}
\label{cond_traccia_fin_mu_n}
\sum_{n \ge 1} \mu_n^{-{2\gamma}}  < \infty.
\end{align} 
We also require 
\begin{hypothesis} 
\label{hyp_rho}
For any $n\in\N$ $\rho^2\neq 4\mu_n^{1-2\alpha}$.  
\end{hypothesis}

\begin{remark}   Concerning the basic example \eqref{sto_damped_WE} in Section \ref{sub:example_stoc_wave_eq}, we have
\begin{align*}
& \mathcal{D}\left(  \Lambda\right)  =\{y\in\
H^{2}\left(\left[  0,1\right]\right)\cap H_0^1\left(\left[  0,1\right]\right):
y''\in
H^{2}\left(\left[  0,1\right]\right)\cap H_0^1\left(\left[  0,1\right]\right)\}, \\
& \Lambda y    =y^{(iv)}=(-y'')''
\in L^2([0,1]) ,\text{ \ for every }y\in\mathcal{D}\left(
\Lambda\right).
\end{align*}
Moreover   
 $y_0 \in  V = H^{2}\left(\left[  0,1\right]\right)\cap H_0^1\left(\left[  0,1\right]\right)$, $y _1 \in U=L^2([0,1])$.   Note that $\Lambda^{{-2\gamma}}$   has {\it finite trace} for $\gamma>1/8$ since the eigenvalues of $\Lambda$ are  $\lambda_n = \pi^4 n^4$, $n \ge 1$.
\end{remark}

By considering  $\widetilde G: U  \longrightarrow H$,
$\widetilde G u=\left(
\begin{array}
[c]{c}%
0\\
u
\end{array}
\right)  =\left(
\begin{array}
[c]{c}%
0\\
I
\end{array}
\right)  u
$
 for any $u\in U$, $G=\widetilde G\Lambda^{-\gamma}$, and for any $ h  =(h_1, h_2) \in H$ we define
\begin{equation}\label{B}
C(\tau, h) = G  \, \tilde C (\tau, h) (\xi):=
\left(\begin{array}{l}
     0\\ 
c(\tau, \xi, h_1(\xi)) 
 \end{array}\right), \;\;\; \xi \in [0,1],\; \tau \in [0,T]. \; 
\end{equation}
It is easy to see that $\tilde C  (\tau, h) =  
c(\tau, \cdot , h_1(\cdot )) $ with values in  $U$ is $\beta$-H\"older continuous  
 in $h$ uniformly in $\tau$ (cf.  \eqref{ey33}).
We consider also the operators $\cala_{\alpha,\rho}:D(\cala_{\alpha,\rho})\subset H\rightarrow H$  
\begin{align}
\label{def:A_alfarho-G-1_bis}
\cala_{\alpha,\rho}
:=\left(\begin{matrix}
0 & I \\ -\Lambda & -\rho \Lambda^\alpha
\end{matrix}\right).
\end{align}
Writing 
$ X_\tau(\xi):=
\Big(
\begin{array}{l}
  y(\tau, \xi)\\
  \frac{\partial}{\partial \tau}y(\tau,\xi)
 \end{array}
 \Big),$ it follows that equation \eqref{abst_sto_damped_WE_1} can be reformulated as a stochastic evolution equation in $H$:
\begin{align}
\label{abs_stoc_damped_wave_eq_2_bis}
\left\{
\begin{array}{ll}
dX_\tau= {\mathcal A_{\alpha,\rho}}X_\tau d\tau+G\widetilde C(\tau,X_\tau)d\tau+G dW_\tau, & \tau\in[0,T], \vspace{1mm} \\
X_0=\overline x_0:=\left(\begin{matrix}
x_0^1\\x_0^2
\end{matrix}
\right)\in H,
\end{array} \right.
\end{align}
which has the form of \eqref{abs_1nonlinear_stochastic_eq}.

In the next Sections \ref{subsec:mathscrA_alpha-rho} and \ref{1sub_app_spectral_dec} we provide
 preliminary results. Then in Section \ref{subsec:linear_wav_eq} we prove that stochastic linear equation \eqref{abs_stoc_damped_wave_eq_2_bis} or \eqref{abst_sto_damped_WE_1} with  $\widetilde C\equiv 0$ is well-posed in $H$.  To this purpose we need Hypothesis \ref{hyp_rho} and  condition \eqref{lamb1}. 
  Finally, in section \ref{sub:strong_uniq_damped_wave} we will formulate our main result on well-posedness for \eqref{abst_sto_damped_WE_1}.

\subsubsection{The operator {$\widetilde {\mathcal{A}}_{\alpha,\rho}$}}
\label{subsec:mathscrA_alpha-rho}

The techniques of \cite{chen-russ} and of \cite{lasi-trig} can be adapted   to our situation.
In order to study properties of the operator $\mathcal{A}_{\alpha,\rho}$, we introduce the operator
\begin{align}\label{tilde_A}
\widetilde {\cala}_{\alpha,\rho}:=\left(
\begin{matrix}
0 & \Lambda^{1/2} \\
-\Lambda^{1/2} & -\rho \Lambda^\alpha
\end{matrix}
\right)
\end{align}
on the space $\widetilde H:=U\times U$, with $D(\widetilde {\cala}_{\alpha,\rho}):=D(\Lambda^{1/2})\times D(\Lambda^{(1/2)\vee \alpha})$. We recall that $ D(\Lambda^{1/2})=V$.
When $\rho$ satisfies suitable assumptions (see Hypothesis \ref{hyp_rho}) we have 
two generation results: from \cite[Appendix A]{chen-trig-1} the operator $\cala_{\alpha,\rho}$ generates a strongly continuous semigroup $(e^{t\cala_{\alpha,\rho}})_{t\geq0}$ on $H$ which is also analytic for $\alpha\in[\frac12,1)$, and from \cite[Section 3]{lasi-trig} the operator $\widetilde {\cala}_{\alpha,\rho}$ generates a strongly continuous semigroup $(e^{t\widetilde {\cala}_{\alpha,\rho}})_{t\geq0}$ on $\widetilde H$. 
\newline 
Let us introduce the operator $M:H\rightarrow \widetilde H$ defined as
\begin{align}\label{eq:M}
M:=
\left(
\begin{matrix}
\Lambda^{1/2} & 0 \\ 0 & I
\end{matrix}
\right), \quad M\left(   
\begin{matrix}
x_1 \\ x_2 
\end{matrix}
\right)=
\left(
\begin{matrix}
\Lambda^{1/2}x_1 \\ x_2 
\end{matrix}
\right).
\end{align}
We notice that $\widetilde H=MH$. Further, since
\begin{align*}
 \langle Mx,y\rangle_{\widetilde H}
= \langle \Lambda^{1/2}x_1,y_1\rangle_{U}
+ \langle x_2,y_2\rangle_{U}
= \langle x_1,\Lambda^{-1/2}y_1\rangle_{V}+\langle x_2,y_2\rangle_{U}=\langle x,M^*y\rangle_{H},
\end{align*}
it follows that
\begin{align*}
M^*=\left(
\begin{matrix}
\Lambda^{-1/2} & 0 \\
0 & I
\end{matrix}
\right)=M^{-1}:\widetilde H\rightarrow H.
\end{align*}
The operator $M$ is the link between $(e^{t\cala_{\alpha,\rho}})_{t\geq0}$ and $(e^{t\widetilde {\cala}_{\alpha,\rho}})_{t\geq0}$, as the following lemma states.
\begin{lemma}
\label{1lemma_app_uguag_smgr} Let $\cala_{\alpha,\rho},\,\widetilde{\cala}_{\alpha,\rho}$ and $M$ be defined in (\ref{def:A_alfarho-G-1_bis}), in \eqref{tilde_A} and in \eqref{eq:M}, respectively. Then, for any $t>0$ and any $y\in H$ we have $Me^{t\cala_{\alpha,\rho}}y=e^{t\widetilde{ \cala}_{\alpha,\rho}}My$ .
\end{lemma}
\begin{proof} By density, to get the thesis it is enough to prove the equality for $y\in D(\Lambda)\times D(\Lambda^{(1/2)\vee\alpha})$. Let $y\in D(\Lambda)\times D(\Lambda^{(1/2)\vee\alpha})$. We set $f(t):=Me^{t\cala_{\alpha,\rho}}y$ and $g(t):=e^{t\widetilde {\cala}_{\alpha,\rho}}My$. For any $t>0$ we have
\begin{align*}
f'(t)=M\cala_{\alpha,\rho}e^{t\cala_{\alpha,\rho}}y
= \left(
\begin{matrix}
0 & \Lambda^{1/2} \\
-\Lambda & -\rho \Lambda^{\alpha}
\end{matrix}
\right)M^{-1}f(t)=\widetilde {\cala}_{\alpha,\rho}\ \!f(t),
\end{align*}
and $f(0)=My$. Let us set $h:=f-g$, since $g'(t)=\widetilde {\cala}_{\alpha,\rho}\ \!g(t)$ for any $t>0$ and $g(0)=My$ it follows that $h\in C^1([0,+\infty),H)$, $h'(t)=\widetilde {\cala}_{\alpha,\rho}\ \!h(t)$ on $(0,T]$ and $h(0)=0$. This gives $h\equiv0$ which implies the thesis.
\end{proof}

\subsubsection{Spectral decomposition in {$\widetilde H$} $= U \times U$   }
\label{1sub_app_spectral_dec}

{Only in this subsection and in Appendix \ref{appendix A} we consider complexified spaces. We do not change the notation to not weigh down them.}

Let $\Lambda:D(\Lambda)\subset U\rightarrow U$ be  as in Section \ref{sec6:set_ass},  let $(\mu_n)$ be the family of eigenvalues of $\Lambda$ (without loss of generality we can assume that they are simple, see \cite{lasi-trig}), and let $\{e_n\}_{n\in\N}$ be a family of corresponding eigenvectors (not-normalized), i.e., $\Lambda e_n=\mu_n e_n$, $n\in\N$; we know that  $\mu_n\nearrow+\infty$ as $n\rightarrow+\infty$. 
We notice that $\{e_n\}_{n\in\N}$ forms an orthogonal basis of $U$. 
Further, let $\widetilde G:U\rightarrow \widetilde H$ be defined as before, see \eqref{B}. 

Now we extend the computations introduced in \cite{chen-russ} and recovered in \cite{lasi-trig} and \cite{trig} to the space $\widetilde H$; it follows that the operator $\widetilde{\mathcal A}_{\alpha,\rho}$ has eigenvalues $\lambda_n^+,\lambda_n^-$ in $H$, $n\in\N$, where 
\begin{align*}
\lambda_n^{\pm}:=\frac{-\rho\mu_n^{\alpha}\pm\sqrt{\rho^2\mu_n^{2\alpha}-4\mu_n}}{2}, \quad \lambda_n^+\lambda_n^-={\mu_n}, \quad \lambda_n^++\lambda_n^-=-\rho\mu_n^\alpha, \quad n\in\N,
\end{align*}
 Recall that Hypothesis \ref{hyp_rho}   ensures that the eigenvalues  $\{\lambda_n^{\pm} :n\in\N\}$ of $\widetilde {\cala}_{\alpha,\rho}$ are simple  (see \cite[Section 3]{lasi-trig} and also Section 2.3 in \cite{trig}).
 The normalized corresponding eigenvectors are given by
\begin{align}
\label{1autovet_phi}
\Phi_n^{+}:=\left(
\begin{matrix}
\mu_n^{1/2} e_n \\ \lambda_n^+ e_n
\end{matrix}
\right), \quad 
\Phi_n^{-}:=\chi_n\left(
\begin{matrix}
\mu_n^{1/2} e_n \\ \lambda_n^- e_n
\end{matrix}
\right), 
\end{align}
 where (possibly replacing { $e_n$ by $\gamma_n e_n$})  we may assume that 
$$
(\mu_n+|\lambda_n^+|^2)|e_n|^2_{U}=1, \quad  \; \;
\chi_n^2(\mu_n+|\lambda_n^-|^2)|e_n|_{U}^2=1, \quad n\in\N,
$$ 
with $\chi_n^2 =  \frac{\mu_n+|\lambda_n^+|^2}{\mu_n+|\lambda_n^-|^2}$.
{\it Note that   $\{\Phi_n^+:n\in\N\}$ and $\{\Phi_n^-:n\in\N\}$
each forms an orthonormal family on $\widetilde  H$.} 
  
On the other hand, $\{\Phi_n^+:n\in\N\} \cup \{\Phi_n^-:n\in\N\}$  is a complete family on $\widetilde  H$ under Hypothesis \ref{hyp_rho}.  
Let us set $\widetilde {H}^+:=\overline{\rm span}\{\Phi_n^+:n\in\N\}$, $\widetilde {H}^-:=\overline{\rm span}\{\Phi_n^-:n\in\N\}$ and let us consider the decomposition 
$$\widetilde {H}=\widetilde {H}^+\oplus \widetilde {H}^-  
  \;\;\; \text{ (non-orthogonal, direct sum).}  
  $$
We denote by $x^+$ the projection of $x$ on $H^+$ and on $x^-$ the projection of $x$ on $H^-$.   From the previous decomposition we have
\begin{align}
\widetilde{\mathcal A}_{\alpha,\rho} x & =\sum_{n=1}^\infty \left(\lambda_n^+\langle x^+,\Phi_n^+\rangle_{\widetilde H} \Phi_n^++\lambda_n^-\langle x^-,\Phi_n^-\rangle_{\widetilde H}\Phi_n^-\right), \quad x\in D(\widetilde{\mathcal A}_{\alpha,\rho}), 
\label{formula_a_tilde}\\
e^{t\widetilde{\mathcal A}_{\alpha,\rho}}x & =\sum_{n=1}^\infty \left(e^{\lambda_n^+t}\langle x^+,\Phi_n^+\rangle_{\widetilde H}\Phi_n^++e^{\lambda_n^-t}\langle x^-,\Phi_n^-\rangle_{\widetilde H}\Phi_n^-\right), \quad t\in[0,+\infty), \ x\in {\widetilde H}. \label{formula_semgr_e_tAtilde}
\end{align}
Further,  for any $a\in U$ we have     
\begin{align} \label{wqqq}
\widetilde Ga=\left(
\begin{matrix} 
0 \\ a
\end{matrix}
\right)=:\sum_{n=1}^\infty\left(c_n^+\Phi_n^++c_n^-\Phi_n^-\right);
\end{align}
by considering the orthonormal basis $\{e_n/|e_n|_{U}\}_{n\in\N}$ of $U$ it follows that
\begin{align*}
c_n^++\chi_nc_n^-=0, \quad (c_n^+\lambda_n^++\chi_nc_n^-\lambda_n^-)|e_n|_{U}^2=\langle a,e_n\rangle_{U}, \quad n\in\N,
\end{align*}
which implies
\begin{align*}
c_n^-=-\frac{c_n^+}{\chi_n}, \quad c_n^+=\frac{1}{(\lambda_n^+-\lambda_n^-)|e_n|_{U}}\langle a,e_n/|e_n|_{U}\rangle_{U}, \quad n\in\N.
\end{align*}
Further, let $\widetilde{\mathcal A}_{\alpha,\rho}^+:=\widetilde{\mathcal A}_{\alpha,\rho}^{\widetilde {H}^+}:D(\widetilde{\mathcal A}_{\alpha,\rho}^+)(:=D(\widetilde{\mathcal A}_{\alpha,\rho})\cap \widetilde {H}^+)\subset \widetilde {H}^+\rightarrow \widetilde {H}^+$ and $\widetilde{\mathcal A}_{\alpha,\rho}^-:=\widetilde{\mathcal A}_{\alpha,\rho}^{\widetilde {H}^-}:D(\widetilde{\mathcal A}_{\alpha,\rho}^-)(:=D(\widetilde{\mathcal A}_{\alpha,\rho})\cap \widetilde {H}^-)\subset \widetilde {H}^-\rightarrow \widetilde {H}^-$ be the restrictions of $\widetilde {\cala}_{\alpha,\rho}$ to ${\widetilde H^+}$ and ${\widetilde H^-}$, respectively. For any $h\in \widetilde {H}$ we denote by $h^+$ and by $h^-$ its projection on $\widetilde {H}^+$ and $\widetilde {H}^-$, respectively. With respect to this decomposition, the operators
\begin{align*}
\widetilde{\mathcal A}_{\alpha,\rho}=\left(
\begin{matrix}
\widetilde{\mathcal A}_{\alpha,\rho}^+ & 0 \\
0 & \widetilde{\mathcal A}_{\alpha,\rho}^-
\end{matrix}
\right), \quad
e^{t\widetilde{\mathcal A}_{\alpha,\rho}}=
\left(
\begin{matrix}
e^{t\widetilde{\mathcal A}_{\alpha,\rho}^+} & 0 \\
0 & e^{t\widetilde{\mathcal A}_{\alpha,\rho}^-}
\end{matrix}
\right), \ t\geq0,
\quad \widetilde G=\left(
\begin{matrix}
\widetilde G^+ \\ \widetilde G^-
\end{matrix}
\right),
\end{align*}
admit the following explicit formulae:
\begin{align*}
\widetilde{\mathcal A}_{\alpha,\rho}^+x^+ & =\sum_{n=1}^\infty \lambda_n^+\langle x^+,\Phi_n^+\rangle_{\widetilde {H}}\Phi_n^+, \ x^+\in D(\widetilde{\mathcal A}_{\alpha,\rho}^+),
&\widetilde{\mathcal A}_{\alpha,\rho}^-x^-  &  = \sum_{n=1}^\infty \lambda_n^-\langle x^-,\Phi_n^-\rangle_{\widetilde {H}}\Phi_n^-, \ x^-\in D(\widetilde{\mathcal A}_{\alpha,\rho}^-), \\
e^{t\widetilde {\cala^+}_{\alpha,\rho}}x^+ & =\sum_{n=1}^\infty e^{t\lambda_n^+}\langle x^+,\Phi_n^+\rangle_{\widetilde {H}}\Phi_n^+, \quad x^+\in \widetilde {H}^+, 
&e^{t\widetilde {\cala^-}_{\alpha,\rho}}x^-  & = \sum_{n=1}^\infty e^{t\lambda_n^-}\langle x^-,\Phi_n^-\rangle_{\widetilde {H}}\Phi_n^-, \quad x^-\in \widetilde {H}^-, \\
\widetilde G^+a & =\sum_{n=1}^\infty b_n^+a_n\Phi_n^+, \quad a\in U, 
&\widetilde G^-a  & = \sum_{n=1}^\infty b_n^- a_n\Phi_n^-, \quad a\in U,
\end{align*}
where { (cf. \eqref{wqqq})}
\begin{align}
\label{formula_b_n}
a_n:=\langle a,e_n/|e_n|_{U}\rangle_{U}, \quad b_n^+:=\frac{1}{(\lambda_n^+-\lambda_n^-)|e_n|_{U}}, \quad b_n^-:=-\frac{b_n^+}{\chi_n}, \quad n\in\N.
\end{align}
From the definition of $e_n$, 
$\lambda_n^{\pm},\mu_n$ and $b_n^{\pm}$ (cf. formulae \cite[(2.3.14)-(2.3.18)]{trig}) we have
\begin{align}
\label{1stime_coefficienti_avl}
|\lambda_n^{\pm}|\sim \mu_n^{1/2}, \quad  \,  |e_n|_{U}\sim \mu_n^{-1/2}, \quad |\lambda_n^+-\lambda_n^-|\sim \mu_n^{1/2}, \quad b_n^{\pm}\sim {\rm cost}, \quad \chi_n\sim {\rm cost},
\end{align}
definitively with respect to $n\in\N$. {Finally, since $|\lambda_n^+|,|\lambda_n^-|$ blows up as $n\rightarrow +\infty$ and $\lambda_n^+,\lambda_n^-$ has negative real part, from \eqref{formula_a_tilde} and \eqref{formula_semgr_e_tAtilde} it follows that $t\mapsto e^{t\widetilde {\mathcal A}_{\alpha,\rho}}x$ belongs to $C^1((0,+\infty);\widetilde H)\cap C((0,+\infty);D((\widetilde {\mathcal A}_{\alpha,\rho})^\eta))$ for any $\eta>0$ and any $x\in \widetilde H$, and for any $T>0$ and $\eta>0$ there exists a positive constant $L=L_{T,\eta}$ such that $\|(\widetilde {\mathcal A}_{\alpha,\rho})^\eta e^{t\widetilde{\mathcal A}_{\alpha,\rho}}\|_{L(\widetilde H)}\leq t^{-\eta}L_T$.}
 
\subsubsection{Linear stochastic damped {Euler-Bernoulli beam} equations}
\label{subsec:linear_wav_eq}
Let us consider the problem
\begin{align}\label{lin_stoc_problem_1}
\left\{
\begin{array}{ll}
dX_t= {\mathcal A_{\alpha,\rho}}X_tdt+GdW_t, & t\in[0,T], \vspace{1mm} \\
X_0=\left(\begin{matrix}
x_0^1\\x_0^2
\end{matrix}
\right)=x \in H = V\times U.\end{array} \right.
\end{align}
where $\mathcal A_{\alpha,\rho}$ has been defined in (\ref{def:A_alfarho-G-1_bis}). 
 In the following we will refer to the solution of equation \eqref{lin_stoc_problem_1} as the Ornstein-Uhlenbeck process. We have $\alpha\in[0,1)$ and $\rho$ satisfying Hypothesis \ref{hyp_rho}. 
 
 {\it We will show that equation (\ref{lin_stoc_problem_1}) is well-posed in $H$,}  i.e.,   for any $x \in H$, there exists a unique mild solution  having continuous paths with values in $H$  (cf. \eqref{mild}). This  
is given by 
\begin{align}
\label{mild_solution_lin_prb_1}
X_t=e^{t{\mathcal A_{\alpha,\rho}}}x+\int_0^te^{(t-s){\mathcal A_{\alpha,\rho}}}GdW_s, \quad t\in[0,T],\;\; \P-a.s..
\end{align}
To this purpose we need to show that the stochastic convolution
\begin{align}
\label{stochastic_convolution_1} 
W_{\cala_{\alpha,\rho}}(t):=\int_0^te^{(t-s){\mathcal A_{\alpha,\rho}}}GdW_s, \quad t\in[0,T],
\end{align}
verifies condition \eqref{e5.17} (cf. Section 5.3 in  \cite{DPsecond}). {This is an easy consequence of \eqref{lamb1}.}
 \begin{proposition}
    \label{lemma_conv_stocastica}
Assume that Hypothesis \ref{hyp_rho} and condition \eqref{lamb1} are satisfied. Then, 
 \begin{align} 
\label{stima_somma_stoc_convolution}
    {  \sup_{t \ge 0 } \| e^{t \cala_{\alpha\,\rho}} G \|_{L_2(U, H)} < \infty }
\end{align} 
  which implies \eqref{e5.17}. 
\end{proposition}
{
\begin{proof}
Formula \eqref{stima_somma_stoc_convolution} follows from the fact that, from \eqref{cond_traccia_fin_mu_n}, $\Lambda^{-2\gamma}$ is a trace-class operator on $U$ and that $G=\widetilde G\Lambda^{-\gamma}$.
\end{proof}}
 
\subsubsection{Strong uniqueness for nonlinear damped equations}
\label{sub:strong_uniq_damped_wave}

Let us consider the nonlinear equation \eqref{abs_stoc_damped_wave_eq_2_bis} or \eqref{abst_sto_damped_WE_1}
under the hypotheses given in Section \ref{sec6:set_ass}.  We know that such assumptions implies in particular that Hypothesis \ref{hyp_1} is satisfied, with $A=\cala_{\alpha,\rho}$ and $G$ defined as in Section \ref{sec6:set_ass}. 
 Hence, there exists a unique weak solution $X=(X_t)_{t\geq0}$ to\  
 \eqref{abs_stoc_damped_wave_eq_2_bis}. 

Next we  show that the controllability assumption in Hypothesis \ref{hyp_22}$(i)$, the estimates on $\Lambda_1$ and $\Lambda_2$ in Hypotheses \ref{hyp_22}$(ii)$ and in \ref{hyp:sigma_3} and Hypothesis \ref{hyp_3} hold true, with $\Gamma(t)=\Gamma_{\alpha,\rho}(t)=\mathscr (Q_t^{\alpha,\rho})^{-1/2}e^{t\cala_{\alpha,\rho}}$ for any $t>0$.
\begin{proposition}
\label{pro:energy_est_damped}
Let us take $A=\cala_{\alpha,\rho}$ and $G$ as in Section \ref{sec6:set_ass}, let $\Gamma(t)=\Gamma_{\alpha,\rho}(t)=\mathscr (Q_t^{\alpha,\rho})^{-1/2}e^{t\cala_{\alpha,\rho}}$ for any $t>0$ and let $\beta$ satisfies Hypothesis \ref{hyp_non_linear_damped_1}.
Then, Hypotheses \ref{hyp_22}, \ref{hyp_3} and \ref{hyp:sigma_3} hold true with
\begin{align}
\label{stime_energia}  
{\Lambda_1(t)\sim\Lambda_2(t)
\sim 
 t^{-1/2-2\gamma}}
 \end{align}
for every $\alpha\in\left[0,\frac12\right]$
\end{proposition}
\begin{proof}
Estimates of $\Lambda_1$ and $\Lambda_2$ follows from Corollary \ref{app:stime_energia_direzionali_1} and Theorem \ref{thm:stime_globali_damped_beam_control}.
{
Further, 
\begin{align*}
|\Gamma(t)Ga|_H\leq 
{C}{t^{-1/2-2\gamma}}|\Lambda^{-\gamma} a|_U, \quad t\in(0,T],
\end{align*}
for every $\alpha\in[0,1/2]$.}
Therefore, for any orthonormal basis $\{h_n:n\in\N\}$ of $U$ we have
\begin{align*}
\|\Gamma(t)G\|_{L_2(U;H)}^2
= &\sum_{n\in\N}|\Gamma(t)Gh_n|_H^2
\leq \frac{C^2}{t^{1+4\gamma}} \sum_{n\in\N} |\Lambda^{-\gamma} h_n|_U^2=\frac{C^2 { {\rm Tr}(\Lambda^{-2\gamma}) } }{t^{1+4\gamma}},
\quad t\in(0,T].
\end{align*}
{
This implies that \eqref{stima_hyp_HS} holds true with $\Lambda_2(t)=t^{-(1/2+2\gamma)}$.}

{Let us conclude by proving that also Hypothesis \ref{hyp_3} are satisfied. Indeed, we have
\begin{align*}
\Lambda_1(t)^{1-\beta}\Lambda_2(t)
\sim\   t^{-(1/2+2\gamma)(2-\beta)},
\end{align*}
and from the choice of $\beta$ we get $-(1/2+2\gamma)(2-\beta)>-1$.}
\end{proof}
{ The previous results show that we can apply Theorem \ref{teo:abs_uni1} to equation \eqref{abs_stoc_damped_wave_eq_2_bis}. We finally get} 
\begin{theorem} \label{wff}
Let $\cala_{\alpha,\rho}$ and $G$ be defined as in Section \ref{sec6:set_ass}, 
and let the assumptions of  Section \ref{sec6:set_ass} be satisfied. Then, for the nonlinear damped beam equation \eqref{abs_stoc_damped_wave_eq_2_bis} the assertions of Theorem \ref{teo:abs_uni1} hold. In particular, we have  pathwise uniqueness for \eqref{abs_stoc_damped_wave_eq_2_bis}. 
\end{theorem}

\subsection{{ Semilinear stochastic heat equations} }
\label{sub:heat_equation}

Let us assume that $A=\Delta$ is the realization of the Laplacian operator in $H=U=L^2([0,\pi]^d)$ with periodic boundary conditions, and let $G=(-\Delta)^{-\gamma/2}$
with $\gamma\geq0$.  We are considering 
 \begin{align} 
\label{ww4}
\left\lbrace
\begin{array}{ll}dX^{x}_t= \Delta X^{x}_t dt + C(X^{x}_t)dt + (-\Delta)^{-\gamma/2} dW_t, & t \in[0,T], \vspace{1mm} \\ 
X_0^{x}=x\in H.  
 \end{array}\right.
\end{align} 
It is well known that $D(A)=H^2([0,2\pi]^d)_{\rm per}$, the classical Sobolev space with periodic boundary conditions. We notice that Hypothesis \ref{ip_finite}-$1.$ is fulfilled.
  Let $W$ be a cylindrical Wiener process on $H$, see Section \ref{sub:abstract_equation}.
  We first note that assumption \eqref{e5.17}  is verified if 
\begin{equation}\label{see2}
 1+\gamma> \frac d2
\end{equation}
(see also  the proof of Lemma 9 in \cite{dapr-fl}). 
 In particular,  if $\gamma =0$, i.e., $G=I$,  it is required  $d=1$, 
  as in \cite{dapr-fl}.  Let $R>0$. We consider, { for a fixed  $\beta \in (0,1)$, }
\begin{align} \label{s33}
 C(f)(\xi):=g(\xi)\int_{[0,2\pi]^d} h(\xi')\, {\left(|f(\xi')|\wedge R\right)^{\beta}}\ \! d\xi', \quad \xi\in [0,2\pi]^d,
\end{align}
for any $f\in H$  and $g\in H^{\gamma}([0,2\pi]^d)_{\rm per}$  (see, for instance, Section 6 in  \cite{hairer}) and $h\in L^\infty([0,2\pi]^d)$. The regularity of $g$ implies that   $C(f)\in D((-A)^{\gamma/2})={\rm Im}(G)$ for any $f\in H$. Hence we can set
\begin{gather*} 
 \widetilde C(f) = G^{-1}C(f) = (-\Delta)^{\gamma/2} C(f),\;\;\; f \in H.
\end{gather*}
and write $C(f)=   (-\Delta)^{-\gamma/2} \widetilde C(f) $ in \eqref{ww4}. 
 
 Arguing as in \cite[Lemma 8]{dapr-fl} it follows that there exists a positive constant $M$ such that
\begin{align*}
|\widetilde C(f_1)- \widetilde C(f_2)|_H
\leq M|(-\Delta)^{-\gamma/2} g|_{H}\|h\|_{\infty}|f_1-f_2|_H^{{ \beta}}, \quad f_1,f_2\in H. 
\end{align*} 
On the other hand, using the orthonormal basis $(e_n)$ in Hypothesis \ref{ip_finite},  we have the estimates
\begin{gather*}
 |\widetilde C(f)_n| = |\langle \widetilde C(f), e_n \rangle_H| 
 \\
 = |\langle (-\Delta)^{\gamma/2}g, e_n\rangle_H \, \int_{[0,2\pi]^d} h(\xi')\sqrt{|f(\xi')|\wedge R}\ \! d\xi' | \le C R |\langle (-\Delta)^{\gamma/2}g, e_n\rangle_H |\,  \|h\|_{\infty};
\\ 
|\widetilde C(f_1)_n - \widetilde C(f_2)_n | 
\le 
  C  |\langle (-\Delta)^{\gamma/2}g, e_n\rangle_H |\,   \|h\|_{\infty} \, 
  |f_1-f_2|_H^{{ \beta}},\;\;\; \;\;\; f, f_1, f_2 \in H.
\end{gather*}
 Hence, Hypothesis \ref{ip_finite}-$2.$ is satisfied. Indeed we have 
\begin{gather*}
\sum_{n\in\N}\frac{\|(\widetilde C(\cdot))_n\|^2_{{ \beta}}}{\alpha_n}
 \le 
 \frac{C_R}{\alpha_1} \, \|h\|_{\infty}^2  \, \sum_{n\in\N} |\langle (-\Delta)^{\gamma/2}g, e_n\rangle_H |^2  
 \le  \frac{C_R'}{\alpha_1} \|h\|_{\infty}^2 \, \| g \|_{H^{\gamma}([0,2\pi]^d)_{\rm per}}^2 < \infty.
\end{gather*}
Let us discuss (ii) in  Hypothesis \ref{hyp_3}. With our choice of $A,H$ and $G$ we have 
\begin{align*}
\Lambda_1(t)=t^{-1/2-\gamma/2},  \quad \Lambda_2(t)=t^{-1/2},  \quad t\in(0,T].
\end{align*}
Hence, Hypothesis \ref{hyp_3} $(ii)$ is satisfied if
\begin{align}
\label{ex_gamma_beta}
0\leq \gamma<\frac{\beta}{1-\beta}.
\end{align}
{ Since $\beta(1-\beta)^{-1}\rightarrow+\infty$ as $\beta\rightarrow1^-$, it follows that the bigger the H\"older exponent $\beta$ is, the bigger is the bound for $\gamma$. 
  In particular (see also the remark below):
  {
\begin{enumerate}[(i)]
\item  if we choose $\beta\in (\frac13,1)$ then \eqref{ex_gamma_beta} is satisfied for some $\gamma>\frac12$, hence $2(1+\gamma)>3$ and according to \eqref{see2} we can also consider $d=3$ for the SPDE \eqref{ww4};
\item if  we take $\beta\in(1/2,1)$ then \eqref{ex_gamma_beta} if fulfilled for some $\gamma>1$ and we get $2(1+\gamma)>4$; according to \eqref{see2}  we can also consider the case $d=4$ for the SPDE \eqref{ww4}.
\end{enumerate}}}
{ \sl In the previous two cases we can apply Theorem 
 \ref{teo:abs_uni1} 
 to equation \eqref{ww4} with $C$ given in \eqref{s33} and obtain 
 pathwise uniqueness and Lipschitz dependence of  initial conditions.  }

\begin{remark} \label{df11} Concerning \eqref{ww4}
 the main difference with \cite{dapr-fl} is that in such paper the authors  require   the assumption
\begin{align*}
\int_0^T\Lambda_t^{(1+\theta)}dt<+\infty, \quad \theta=\max\{\beta,1-\beta\}, \ \Lambda_t=\Lambda_1(t)=t^{-1/2-\gamma/2}.
\end{align*}
Hence, the best situation is $\beta=\frac12$ and above condition reads as
\begin{align*}
\frac32\left(\frac12+\frac\gamma2\right)<1 \Leftrightarrow \gamma<\frac13.
\end{align*}
{ According to \eqref{see2} we need 
$1+\gamma>\displaystyle \frac d2$.  By combining this condition with $\gamma<\frac13$ it follows that the case $d=3$ is not reached in \cite{dapr-fl}.} 
\end{remark}

\appendix
 \section{Energy estimates for the control problem associated to the beam equation }
\label{appendix A}

{As in Subsection \ref{1sub_app_spectral_dec}, here we consider complexified spaces.}
In $H=V\times U$ we consider the following control problem
\begin{align}\label{1contr_problem_agg}
\left\{
\begin{array}{ll}
z_t(t)=\mathcal A_{\alpha,\rho}z(t)+\widetilde G{\Lambda^{-\gamma}}u(t), & t\in[0,T], \vspace{1mm} \\
z(0)=(z_0^1,z_0^2)\in H,
\end{array}
\right.
\end{align}
where 
\begin{align*}
\mathcal A_{\alpha,\rho}
:=\left(
\begin{matrix}
0 & I \\
-\Lambda & -\rho \Lambda^\alpha
\end{matrix}
\right):D(\cala_{\alpha,\rho})\subset H\rightarrow H,\quad
\widetilde G:=\left(
\begin{matrix}
0 \\ I
\end{matrix}
\right):U\rightarrow H,
\end{align*}
and $\Lambda: {D(\Lambda)} \subset U\rightarrow U$, $\alpha$ and $\rho$ satisfy the assumptions in Section \ref{sec6:set_ass}. 

Let $z(t)=(z_1(t),z_2(t))\in H$ be the solution to the problem \eqref{1contr_problem_agg}; arguing as in Section \ref{subsec:mathscrA_alpha-rho} we infer that $y(t):=(y_1(t),y_2(t)):=(\Lambda^{1/2}z_1(t),z_2(t))\in \widetilde H:=U\times U$ is solution to 
\begin{align}
\label{1contr_problem}
\left\{
\begin{array}{ll}
y_t(t)=\widetilde {\mathcal A}_{\alpha,\rho}y(t)+\widetilde G{\Lambda^{-\gamma}}u(t), & t\in[0,T], \vspace{1mm} \\
y(0)=(y_0^1,y_0^2)\in \widetilde H,
\end{array}
\right.
\end{align}
where $y_0^1=\Lambda^{1/2}z_0^1$, $y_0^2=z_0^2$ and 
\begin{align*}
\widetilde {\mathcal A}_{\alpha,\rho}:=
\left(
\begin{matrix}
0 & \Lambda^{1/2} \\
-\Lambda^{1/2} & -\rho \Lambda^{\alpha} 
\end{matrix}
\right):D(\widetilde {\cala}_{\alpha,\rho})\subset \widetilde H\rightarrow \widetilde H.
\end{align*}
{For every $a \in U$ and $k= Ga=\widetilde G\Lambda^{-\gamma}a$, we get
 $|k|_{H}= |k|_{\widetilde H}= |{\Lambda^{-\gamma}} a|_{U}$.}

We notice that a control $u$ steers {$k$} to $0$ at time $t\in(0,T]$ in $H$ in (\ref{1contr_problem_agg}) if and only if $u$ steers {$k$}   to $0$ at time $t\in(0,T]$ in $\widetilde H$ in (\ref{1contr_problem}). 
Hence, the energy to steer {$k$} to $0$ at time $t$ in $\widetilde H$, which is given by
\begin{align*}
\widetilde {\mathcal E}_C(t,k):=\inf\left\{\int_0^t|u(s)|^2_{U}ds : u\in L^2(0,T;U), \ y \ {\textrm {solution to \eqref{1contr_problem}}}, \ k=(y_0^1,y_0^2), \ y(t)=0\right\},
\end{align*}
coincides with  the energy to steer {$k$} to $0$ at time $t$ in $H$, which is given by
\begin{align*}
{\mathcal E}_C(t,k):=\inf\left\{\int_0^t|u(s)|^2_{U}ds : u\in L^2(0,T;U), \ z \ {\textrm {solution to \eqref{1contr_problem_agg}}}, \ k=(z_0^1,z_0^2), \ z(t)=0\right\}.
\end{align*}
We will prove that (\ref{1contr_problem}) is null controllable and we provide an estimate to $\widetilde{\mathcal E}_C(T,k)$.
\begin{theorem}
\label{1thm:stima_energia} Let $\cala_{\alpha,\rho}$ be defined in (\ref{def:A_alfarho-G-1_bis}) and let Hypotheses \ref{hyp_rho} and condition  \eqref{lamb1} in Section \ref{sec6:set_ass}  hold true. 
{Let $T>0$. 
Then, there exists a positive constant $C=C(T)$ such that for any $a\in U$, $k=Ga$, the energy to steer $k$ to $0$ at time $t$ can be estimated by 
\begin{align*}
\widetilde{\mathcal E}_C(t,k)\leq 
\frac{C|\Lambda^{-\gamma}a|^2_{U}}{t^{1+4\gamma}}
\qquad t\in(0,T].
\end{align*}
In particular, the equality $\widetilde{\mathcal E}_C(t,k)={\mathcal E}_C(t,k)$ implies that for every $T>0$ there exists a positive constant $C=C(T)$ such that for any $a\in U$, $k=Ga$, the energy to steer $k$ to $0$ at time $t$ can be estimated by 
\begin{align*}
{\mathcal E}_C(t,k)\leq 
\frac{C|\Lambda^{-\gamma}a|^2_{U}}{t^{1+4\gamma}}
\qquad t\in(0,T].
\end{align*}
}

\end{theorem}
\begin{proof}
Here, we consider the decomposition introduced in Section \ref{1sub_app_spectral_dec} and we keep the same notation.
{We follow the method in \cite[Proposition 1.3]{Zab08}, which has been extended to infinite dimension in \cite{MasPri17}. The idea is the following. If we consider the matrix formulation for $\widetilde {\mathcal A}_{\alpha,\rho}$ and $G$, then the $2\times 2$-matrix $[G|\widetilde {\mathcal A}_{\alpha,\rho}G]$ is invertible. We denote by $K$ its inverse matrix and by $K_0$ and $K_1$ the rows of $K$. Then, the control
\begin{align*}
u(t)=K_0(t)\psi(t)+K_1\psi'(t), \qquad t\in[0,T],    
\end{align*}
steers $k$ to $0$, where $\psi(t)=-\Phi(t)e^{t\widetilde {\mathcal A}_{\alpha,\rho}}k$ and $\Phi(t)$ is a suitable smooth function. We construct our control adapting this approach to our situation.}  

{The mild solution to \eqref{1contr_problem} is
\begin{align*}
y(t)=e^{t\widetilde {\mathcal A}_{\alpha,\rho}}y(0)+\int_0^t e^{(t-s)\widetilde{\mathcal A}_{\alpha,\rho}}\widetilde G\Lambda^{-\gamma}u(s)ds
\end{align*}
for every $t\in[0,T]$.}
Let us fix $T>0$ and $a\in U$. Further, we set $f_T(t):=t^2(T-t)^2$ and $\phi_T(t):=\|f_T\|_{L^1(0,T)}^{-1}f_T(t)$.
It follows that $\|\phi_T\|_{L^1(0,T)}=1$, that $\phi_T$ vanishes at $0$ and $T$ and that {$|\phi_T(t)|\leq \tilde c T^{-3}t^2$, $|t\phi'_T(t)|\leq \tilde cT^{-3}t^2$ and $|\phi_T'(t)|\leq \tilde cT^{-3}t$, for some positive constant $\tilde c$ and any $t\in[0,T]$}.  
We claim that the control
\begin{align*}
v(t):=v_{0}(t)+v'_1(t), \quad t\in(0,T], \qquad v(0)=0,
\end{align*}
defined by
\begin{align}
\langle v_0(t),e_n/|e_n|_{U}\rangle_{U}
& :={\mu_n^{\gamma}}\left(\frac{\lambda_n^- e^{\lambda_n^+ t}}{\lambda_n^--\lambda_n^+}-\frac{\lambda_n^+ e^{\lambda_n^- t}}{\lambda_n^--\lambda_n^+}\right){\widetilde a_n}\phi_T(t),\quad  n\in\N, \label{v_0} \\
\langle v_1(t),e_n/|e_n|_{U}\rangle_{U}
& :={\mu_n^{\gamma}}\left(\frac{- e^{\lambda_n^+ t}}{\lambda_n^--\lambda_n^+}+\frac{ e^{\lambda_n^- t}}{\lambda_n^--\lambda_n^+}\right){\widetilde a_n}\phi_T(t), \quad n\in\N, \label{v_1}
\end{align}
for $t\in(0,T]$, steers the initial state $Ga$ at $0$ at time $T$, where {${\widetilde a_n}:=\langle {\Lambda^{-\gamma}}a,e_n/|e_n|_{U}\rangle_{U}$} for any $n\in\N$ {(we have  $v_1' = \frac{dv_1}{dt}$)}. 
{The series which define $v_0(t)$ and $v_1(t)$ are well-defined since $\mu_n$ grows $|\lambda_n^{\pm}|^2$ (see \eqref{1stime_coefficienti_avl}) and for every $t>0$ we have $e^{t\widetilde {\mathcal A}_{\alpha,\rho}}$ maps $\widetilde H$ onto $D(\widetilde {(\mathcal A}_{\alpha,\rho})^k)$ for every $k\in\N$.}

At first, we notice that $\widetilde G{\Lambda^{-\gamma}}v_1(t)\in D(\widetilde {\cala}_{\alpha,\rho})$ for any $t\in[0,T]$. Indeed, from the definitions of $G$ and $v_1$ we have
\begin{align*}
\widetilde G{\Lambda^{-\gamma}}v_1(t)=\phi_T(t)\sum_{n=1}^\infty \left[\left(-\frac{e^{\lambda_n^+t}}{\lambda_n^--\lambda_n^+}+\frac{ e^{\lambda_n^- t}}{\lambda_n^--\lambda_n^+}\right)b^+_n{\widetilde a_n}\Phi_n^+
+ \left(-\frac{ e^{\lambda_n^+t}}{\lambda_n^--\lambda_n^+}+\frac{e^{\lambda_n^- t}}{\lambda_n^--\lambda_n^+}\right)b^-_n{\widetilde a_n}\Phi_n^-\right],
\end{align*}
For any $N\in\N$ we set
\begin{align*}
(\widetilde G{\Lambda^{-\gamma}}v_1(t))_N=\phi_T(t)\sum_{n=1}^N \left(-\frac{e^{\lambda_n^+t}}{\lambda_n^--\lambda_n^+}+\frac{ e^{\lambda_n^- t}}{\lambda_n^--\lambda_n^+}\right)\left[b^+_na_n\Phi_n^+
+b^-_na_n\Phi_n^-\right].
\end{align*}
Clearly, $(\widetilde G{\Lambda^{-\gamma}}v_1(t))_N\rightarrow \widetilde G{\Lambda^{-\gamma}}v_1(t)$ as $N\rightarrow+\infty$ in $\widetilde H$. Further, from the previous decomposition we infer that $(\widetilde G{\Lambda^{-\gamma}}v_1(t))_N\in D(\widetilde {\mathcal A}_{\alpha,\rho})$ for any $N\in\N$ and
\begin{align*}
\widetilde {\mathcal A}_{\alpha,\rho}(\widetilde G{\Lambda^{-\gamma}}v_1(t))_N=\phi_T(t)\sum_{n=1}^N \left(-\frac{e^{\lambda_n^+t}}{\lambda_n^--\lambda_n^+}+\frac{ e^{\lambda_n^- t}}{\lambda_n^--\lambda_n^+}\right)\left[\lambda_n^+b^+_n{\widetilde a_n}\Phi_n^++\lambda_n^-b^-_n{\widetilde a_n}\Phi_n^-\right].
\end{align*}
Hence, recalling the definition of $\Phi_n^+$ and of $\Phi_n^-$ we get
\begin{align*}
|\widetilde {\mathcal A}_{\alpha,\rho}(\widetilde G{\Lambda^{-\gamma}}v_1(t))_N|_{\widetilde H}^2
= &  {\phi_T(t)^2} \sum_{n=1}^N{\widetilde a_n}^2 \left(-\frac{e^{\lambda_n^+t}}{\lambda_n^--\lambda_n^+}+\frac{ e^{\lambda_n^- t}}{\lambda_n^--\lambda_n^+}\right)^2 \\
& \times \left((\lambda_n^+)^2(b_n^+)^2+(\lambda_n^-)^2((\chi_nb_n^-)^2+2\chi_nb_n^+b_n^-\lambda_n^+\lambda_n^-(\mu_n|e_n|_{U}^2+\lambda_n^+\lambda_n^-
|e_n|_{U}^2)
\right).
\end{align*}
We notice that from \eqref{1stime_coefficienti_avl} it follows that for any $t\in(0,T]$
\begin{align*}
& \Big (-\frac{e^{\lambda_n^+t}}{\lambda_n^--\lambda_n^+}+\frac{ e^{\lambda_n^- t}}{\lambda_n^--\lambda_n^+}\Big)^2 (|\lambda_n^+|^2+|\lambda_n^-|^2+|\lambda_n^+||\lambda_n^-|) \sim  {\rm const};  \\
& (\mu_n|e_n|_{U}^2+\lambda_n^+\lambda_n^-|e_n|_{U}^2)\sim {\rm const}, 
\end{align*}
and that
\begin{align*}
\sum_{n=1}^\infty {\widetilde a_n}^2(b_n^+)^2, \;\;\; \sum_{n=1}^\infty {\widetilde a_n}^2(b_n^-)^2<+\infty.
\end{align*}
Since {
\begin{gather*}  
| \widetilde {\mathcal A}_{\alpha,\rho}(\widetilde G{\Lambda^{-\gamma}}v_1(t))_N - 
 \widetilde {\mathcal A}_{\alpha,\rho}(\widetilde G{\Lambda^{-\gamma}}v_1(t))_{N + p} 
 |_{\widetilde H}^2
\leq  M{ \phi_T(t)^2 } \sum_{n=N+1}^{N+p} \left(-\frac{e^{\lambda_n^+t}}{\lambda_n^--\lambda_n^+}+\frac{ e^{\lambda_n^- t}} {\lambda_n^--\lambda_n^+}\right)^2 \,  \cdot  \\
 (|\lambda_n^+|^2+|\lambda_n^-|^2+|\lambda_n^+||\lambda_n^-|)\left ({\widetilde a_n}^2(b_n^+)^2+{\widetilde a_n}^2(b_n^-)^2\right) 
\leq   M\sum_{n=N+1 }^{N+p} \left({\widetilde a_n}^2(b_n^+)^2+{\widetilde a_n}^2(b_n^-)^2\right),
\end{gather*}
 }
for some positive constant $M$, it follows that $\widetilde {\mathcal A}_{\alpha,\rho}(Gv_1(t))_N$ converges to 
\begin{align*}
\phi_T(t)\sum_{n=1}^\infty \left[\left(-\frac{\lambda_n^+ e^{\lambda_n^+ t}}{\lambda_n^--\lambda_n^+}+\frac{\lambda_n^+ e^{\lambda_n^- t}}{\lambda_n^--\lambda_n^+}\right)b^+_n{\widetilde a_n}\Phi_n^+
+ \left(-\frac{\lambda_n^- e^{\lambda_n^+ t}}{\lambda_n^--\lambda_n^+}+\frac{\lambda_n^- e^{\lambda_n^- t}}{\lambda_n^--\lambda_n^+}\right)b^-_n{\widetilde a_n}\Phi_n^-\right].
\end{align*}
in $\widetilde H$ as $N\rightarrow+\infty$. Since $\widetilde {\mathcal A}_{\alpha,\rho}$ is a closed operator, it follows that $\widetilde G{\Lambda^{-\gamma}}v_1(t)\in D(\widetilde {\mathcal A}_{\alpha,\rho})$, {for any $t \in [0,T]$,} and 
\begin{align}
\label{1Gv1}
\widetilde{\mathcal A}_{\alpha,\rho}\widetilde G{\Lambda^{-\gamma}}v_1(t)=\phi_T(t)\sum_{n=1}^\infty \left[\left(-\frac{\lambda_n^+ e^{\lambda_n^+ t}}{\lambda_n^--\lambda_n^+}+\frac{\lambda_n^+ e^{\lambda_n^- t}}{\lambda_n^--\lambda_n^+}\right)b^+_n{\widetilde a_n}\Phi_n^+
+ \left(-\frac{\lambda_n^- e^{\lambda_n^+ t}}{\lambda_n^--\lambda_n^+}+\frac{\lambda_n^- e^{\lambda_n^- t}}{\lambda_n^--\lambda_n^+}\right)b^-_n{\widetilde a_n}\Phi_n^-\right].
\end{align} 
 Let us consider the integral term in the mild solution. We get
\begin{align}
\int_0^Te^{(T-s)\widetilde{\mathcal A}_{\alpha,\rho}}\widetilde G\Lambda^{-\gamma}v(s)ds
= & \int_0^Te^{(T-s)\widetilde{\mathcal A}_{\alpha,\rho}}\widetilde G\Lambda^{-\gamma}v_0(s)ds+\int_0^Te^{(T-s)\widetilde{\mathcal A}_{\alpha,\rho}}\widetilde G\Lambda^{-\gamma}v'_1(s)ds \notag \\
= & \int_0^Te^{(T-s)\widetilde{\mathcal A}_{\alpha,\rho}}(\widetilde G{\Lambda^{-\gamma}}v_0(s)+\widetilde{\mathcal A}_{\alpha,\rho}\widetilde G{\Lambda^{-\gamma}}v_1(s))ds, \label{1controllo_spezzamento}
\end{align}
where we have integrated by parts under the second integral and we have used the fact that $\widetilde G{\Lambda^{-\gamma}}v_1(s)\in D(\widetilde{\mathcal A}_{\alpha,\rho})$ for any $s\in[0,T]$, and that ${\Lambda^{-\gamma}}v_1(0)={\Lambda^{-\gamma}}v_1(T)=0$. We notice that for any $s\in[0,T]$ we have
\begin{align}
\widetilde G{\Lambda^{-\gamma}}v_0(s)=\phi_T(s)\sum_{n=1}^\infty \left[\left(\frac{\lambda_n^- e^{\lambda_n^+ s}}{\lambda_n^--\lambda_n^+}-\frac{\lambda_n^+ e^{\lambda_n^- s}}{\lambda_n^--\lambda_n^+}\right)b^+_n{\widetilde a_n}\Phi_n^+
+\left(\frac{\lambda_n^- e^{\lambda_n^+ s}}{\lambda_n^--\lambda_n^+}-\frac{\lambda_n^+ e^{\lambda_n^- s}}{\lambda_n^--\lambda_n^+}\right)b^-_n{\widetilde a_n}\Phi_n^-\right].\label{1Gv0}
\end{align}
Hence, \eqref{1Gv1} and \eqref{1Gv0} give
\begin{align}
\widetilde G{\Lambda^{-\gamma}}v_0(s)+\widetilde{\mathcal A}_{\alpha,\rho}\widetilde G{\Lambda^{-\gamma}}v_1(s)
= &  -\phi_T(s)\sum_{n=1}^\infty\left( e^{\lambda_n^+s}b_n^+{\widetilde a_n}\Phi_n^+
+ e^{\lambda_n^-s}b_n^-{\widetilde a_n}\Phi_n^-\right) \notag \\
= & -\phi_T(s)e^{s\widetilde{\mathcal A}_{\alpha,\rho}}(Ga), \quad s\in[0,T].
\label{1conto_somma controlli}
\end{align}
Replacing \eqref{1conto_somma controlli} in \eqref{1controllo_spezzamento} we get
\begin{align*}
\int_0^Te^{(T-s)\widetilde{\mathcal A}_{\alpha,\rho}}\widetilde G{\Lambda^{-\gamma}}v(s)ds
= -e^{T\widetilde{\mathcal A}_{\alpha,\rho}}(Ga).
\end{align*}
The mild formulation of the solution $y$ to \eqref{1contr_problem} implies that
\begin{align*}
y(T)
= & e^{T\widetilde {\cala}_{\alpha,\rho}}(Ga)+\int_0^Te^{(T-s)\widetilde {\cala}_{\alpha,\rho}}\widetilde G{\Lambda^{-\gamma}}v(s)ds
= e^{T\widetilde{\mathcal A}_{\alpha,\rho}}(Ga)-e^{T\widetilde{\mathcal A}_{\alpha,\rho}}(Ga)=0,
\end{align*}
which gives the claim. 

Now we estimate the 
$L^2$-norm of the control $v$. We separately consider the two addends $v_0$ and $v_1$.

As far as $v_0$ is concerned, from \eqref{v_0} we get
\begin{align*}
\int_0^T|v_0(t)|_{U}^2dt
\leq \int_0^T(\phi_T(t))^2\sum_{n=1}^\infty
{\mu_n^{2\gamma}}\left|\frac{\lambda_n^- e^{\lambda_n^+ t}-{\lambda_n^+ e^{\lambda_n^- t}}}{\lambda_n^--\lambda_n^+}\right|^2|{\widetilde a_n}|^2dt.
\end{align*}
From \eqref{1stime_coefficienti_avl}, for any $n\in\N$ we get
{
\begin{align}
{\mu_n^{2\gamma}}\left|\frac{\lambda_n^- e^{\lambda_n^+ t}-{\lambda_n^+ e^{\lambda_n^- t}}}{\lambda_n^--\lambda_n^+}\right|^2
\leq & C \left|\mu_n^{\gamma}\left(e^{\lambda_n^+t}\right)+\mu_n^{\gamma}\left(e^{\lambda_n^-t}\right)\right|^2 \notag \\
\leq & C\left|(\lambda_n^+)^{2\gamma}\left(e^{\lambda_n^+t}\right)+(\lambda_n^-)^{2\gamma}\left(e^{\lambda_n^-t}\right)\right|^2 \notag \\
\leq & C |\widetilde{\mathcal A}_{\alpha,\rho}^{4\gamma}e^{t\widetilde{\mathcal A}_{\alpha,\rho}}\Phi_n^+|_{\widetilde H}^2+|\widetilde{\mathcal A}_{\alpha,\rho}^{4\gamma}e^{t\widetilde{\mathcal A}_{\alpha,\rho}}\Phi_n^-|_{\widetilde H}^2 \notag \\
\leq & CL_T^2t^{-4\gamma},
\label{stima_serie_v0}
\end{align}
where $C$ is a positive constant which may vary from line to line.
Hence,
\begin{align*}
\int_0^T|v_0(t)|_{U}^2dt
\leq {\rm const} \int_0^T|\phi_T(t)|^2t^{-4\gamma}dt|\Lambda^{-\gamma} a|_{U}^2
\leq C|\Lambda^{-\gamma}a|_{U}^2T^{-1-4\gamma},
\end{align*}
for some positive constant $C$.}

Let us consider {$v_1'$}. From \eqref{v_1} we get
\begin{align*}
\int_0^T| {v_1'(t)}|^2_{U}dt
\leq & 2\int_0^T(\phi_T(t))^2\sum_{n=1}^\infty{\mu_n^{2\gamma}}\left|\frac{-\lambda_n^+ e^{\lambda_n^+ t}}{\lambda_n^--\lambda_n^+}+\frac{ \lambda_n^-e^{\lambda_n^- t}}{\lambda_n^--\lambda_n^+}\right|^2|{\widetilde a_n}|^2dt \\
& +2\int_0^T(\phi_T'(t))^2\sum_{n=1}^\infty{\mu_n^{2\gamma}}\left|\frac{- e^{\lambda_n^+ t}}{\lambda_n^--\lambda_n^+}+\frac{ e^{\lambda_n^- t}}{\lambda_n^--\lambda_n^+}\right|^2 {|{\widetilde a_n}|^2 } 
dt:=I_1+I_2.
\end{align*} 
The first integral can be estimated arguing as for $v_0$. By taking $I_2$ into account, if we multiply and divide by ${t^2}$ under the integral sign we get
\begin{align*}
I_2
= 2\int_0^T(t\phi_T'(t))^2\sum_{n=1}^\infty{\mu_n^{2\gamma}}\left|\left(\frac{- e^{\lambda_n^+ t}+ e^{\lambda_n^- t}}{t(\lambda_n^--\lambda_n^+)}\right){\widetilde a_n}\right|^2dt.
\end{align*}
Since there exists a positive constant $C$ such that
\begin{align*}
\sup_{t\in(0,T]}\sup_{n\in\N}{\mu_n^{2\gamma}}\left|\frac{- e^{\lambda_n^+ t}+ e^{\lambda_n^- t}}{t(\lambda_n^--\lambda_n^+)}\right|^2
= \sup_{t\in(0,T]}\sup_{n\in\N}{\mu_n^{2\gamma}}\left|\frac{e^{\lambda_n^+t}(e^{(\lambda_n^--\lambda_n^+)t}-1)}{t(\lambda_n^--\lambda_n^+)}\right|^2\leq C,
\end{align*}
we infer that $I_2\leq c|\Lambda^{-\gamma}a|^2_{U}T^{-1}$ for some positive constant $c$. Therefore, we can conclude that
\begin{align*}
\|v\|_{L^2(0,T;U)}\leq \frac{c|\Lambda^{-\gamma}a|_{U}}{T^{1/2{+2\gamma}}}.
\end{align*}
\end{proof}

Theorem \ref{1thm:stima_energia} has an important consequence, due to the fact that
$\mathcal E_C(t,k)=|Q^{-1/2} e^{t{\mathcal A}_{\alpha,\rho}}k|_H^2$ for any $k\in H$ and any $t>0$.
{
\begin{corollary}
\label{app:stime_energia_direzionali_1} Under the assumptions of Theorem \ref{1thm:stima_energia}, there exists a positive constant $C$ such that for any $a\in U$ and any $t\in(0,T]$ we have
\begin{align*}
& |Q^{-1/2}_te^{t\widetilde{\mathcal A}_{\alpha,\rho}}Ga|_H^2
=\mathcal E_C(t,Ga)
=\mathcal E_C(t,\widetilde G\Lambda^{-\gamma}a)
\leq  \frac{C|\Lambda^{-\gamma}a|^2_{U}}{t^{1+4\gamma}}.
\end{align*}
\end{corollary}}

Now we provide energy estimates in case of general initial datum $k\in H$.
{
\begin{theorem}
\label{thm:stime_globali_damped_beam_control}
Let $\cala_{\alpha,\rho}$ be defined in (\ref{def:A_alfarho-G-1_bis}) and let Hypotheses \ref{hyp_rho} and condition  \eqref{lamb1} in Section \ref{sec6:set_ass}  hold true. 
There exists a positive constant $c$ such that for every $h\in H$ we have
\begin{align*}
\mathcal E_C(t,h)\leq 
\frac{c|h|^2_{H}}{t^{1+4\gamma}} \qquad t\in(0,T]. 
\end{align*}
In particular, there exists a positive constant $c$ such that for any $h\in H$ and any $t\in(0,T]$
\begin{align*}
& |Q^{-1/2}_te^{t\widetilde{\mathcal A}_{\alpha,\rho}}h|_H^2
=\mathcal E_C(t,h)
\leq  \frac{c|h|^2_{H}}{t^{1+4\gamma}}.
\end{align*}
\end{theorem}
\begin{proof}
Let us apply the method exploited in the proof of Theorem \ref{1thm:stima_energia}. Let $h\in H$. As above, we provide an estimate of $\widetilde {\mathcal E}_C(t,h)$. We set
\begin{align*}
\langle v_0(t),e_n/|e_n|_{U}\rangle_{U}
& :={\mu_n^{\gamma}}\left({\lambda_n^- e^{\lambda_n^+ t}}h_n^++\chi_n{\lambda_n^+ e^{\lambda_n^- t}}h^-_n\right)|e_n|_{U}\phi_T(t),\quad  n\in\N,  \\
\langle v_1(t),e_n/|e_n|_{U}\rangle_{U}
& :=-{\mu_n^{\gamma}}\left({ e^{\lambda_n^+ t}}h_n^++\chi_n{ e^{\lambda_n^- t}}h_n^-\right)|e_n|_{U}\phi_T(t), \quad n\in\N, 
\end{align*}
for $t\in(0,T]$, and the control
\begin{align*}
v(t):=v_{0}(t)+v'_1(t), \quad t\in(0,T], \qquad v(0)=0.
\end{align*}
Here, $h_n^+=\langle h^+,\Phi_n^+\rangle_{\widetilde H}$ and $h_n^+=\langle h^-,\Phi_n^-\rangle_{\widetilde H}$ for any $n\in\N$. Arguing as in the proof of Theorem \ref{1thm:stima_energia} it 
is possible to prove that $v$ steers the initial state $h$ at $0$ at time $T$. It remains to estimate the $L^2$-norm of $v$. Arguing as before we get
\begin{align*}
\|v\|_{L^2(0,T;U)}\leq \frac{c|h|^2_{H}}{T^{1/2+2\gamma}},
\end{align*}
and we conclude.
\end{proof}}

\section{Proof of Lemma \ref{abs_lem_interpolazione}}
\label{appendix B}
\begin{proof}
In the proof $C$ is a positive constant which may vary from line to line. 
\newline Let $t\in(0,T]$, let $y\in H$, and let us consider the linear operators
\begin{align*}
\nabla_y \calr_t:C_b^1(H)\rightarrow C_b(H), \quad \nabla_y \calr_t:C_b(H)\rightarrow C_b(H).
\end{align*}
For any $\phi\in C_b^1(H)$ and any $x,y\in H$ we have
\begin{align*}
\nabla_y\calr_t[\phi](x)
= & \int_H\langle \nabla\phi(e^{tA}x+z),e^{tA}y\rangle_H\mathscr N(0,Q_t)(dz),
\end{align*}
from which it follows that
\begin{align}
\label{abs_stima_gradiente_gradiente_scalare}
\sup_{x\in H}|\nabla_y \calr_t[\phi](x)|\leq C_T\|\phi\|_{C_b^1(H)}|y|_H, \quad \phi\in C_b^1(H).
\end{align}
Further, if $\phi\in C_b(H)$ we have
\begin{align*}
\nabla_y\calr_t[\phi](x)
= & \int_H\langle \Gamma(t)y,Q^{-1/2}z \rangle \phi(e^{tA}x+z)\mathscr N(0,Q_t)(dz),
\end{align*}
which combined with \eqref{abs_Gamma_OU} implies
\begin{align}
\label{stima_grad_funz_scalare}
\sup_{x\in H}|\nabla_y \calr_t[\phi](x)|\leq C_T\|\phi\|_{C_b(H)}\Lambda_1(t)|y|_H, \quad \phi\in C_b(H), 
\end{align}
Recalling (\ref{abs_interpolation_result}) and interpolating between \ref{abs_stima_gradiente_gradiente_scalare} and \ref{stima_grad_funz_scalare} we infer that
\begin{align}
\label{abs_stima_inte_1}
\sup_{x\in H}|\nabla_y \calr_t[\phi](x)|_H\leq C_T\|\phi\|_{\beta}\Lambda_1^{1-\beta}(t)|y|_H, \quad \phi\in C^\beta_b(H).
\end{align}
Let us consider $h\in H$ and $\Phi\in C^\beta_b(H;H)$. From \eqref{abs_smgr_sc_smgr_vet} and \eqref{abs_stima_inte_1} we have
\begin{align*}
\sup_{x\in H}|\langle \nabla_yR_t[\Phi](x),h\rangle_H|
\leq C_T\|\Phi_h\|_{\beta}\Lambda_1^{1-\beta}(t)|y|_H
\leq C_T\|\Phi\|_{\beta}\Lambda_1^{1-\beta}(t)|y|_H|h|_H, 
\end{align*}
which gives \eqref{abs_stima_holder_1}. To prove \eqref{abs_stima_holder_3} we fix $y\in H$ and $k\in U$ and we consider the linear operators
\begin{align*}
\nabla_y\nabla^G_k \calr_t:C_b^1(H)\rightarrow C_b(H), \quad \nabla_y\nabla_k^G \calr_t:C_b(H)\rightarrow C_b(H).
\end{align*}
For any $\phi\in C_b^1(H)$ it follows that
\begin{align*}
\nabla_y\nabla^G_k \calr_t[\phi](x)
= \int_H\langle \Gamma(t)Gk,Q^{-1/2}_tz\rangle_H\langle \nabla\phi(e^{tA}x+z),{e^{t A}y}\rangle \mathscr N(0,Q_t)(dz),
\end{align*}
which implies that
\begin{align}
\label{abs_stima_der_seconda_gradiente}
|\nabla_y(\nabla_{k}^G\calr_t[\phi])(x)|_{H}
\leq & {C}|\Phi\|_{C_b^1(H)}\Lambda_2(t)|y|_H|k|_U, \quad t\in(0,T], \ \phi\in C_b^1(H),
\end{align}
for any $x,y\in H$ and any $k\in U$. 
As above, we get 
\begin{align}
\label{abs_stima_der_seconda_funz}
|\nabla_y(\nabla_k^G\calr_t[\phi])(x)  |_{H}
\leq C\|\phi\|_{C_b(H)}\Lambda_1(t)\Lambda_2(t)|y|_H|k|_U,
\end{align}
for any $k\in U$, $x,y\in H$ and $t\in(0,T]$. Interpolating between \eqref{abs_stima_der_seconda_gradiente} and \eqref{abs_stima_der_seconda_funz} we infer that
\begin{align}
\label{stima_sc_der_seconde_G}
|\nabla_y(\nabla_{k}^G\calr_t[\phi])(x)|_{H}
\leq C\|\phi\|_{\beta}\Lambda_1^{1-\beta}(t)\Lambda_2(t)|y|_H|k|_U, \quad t\in(0,T], \quad \phi\in C_b^\beta(H),
\end{align} 
for any $x,y\in H$ and $k\in U$. Therefore, from \eqref{abs_smgr_sc_smgr_vet} and \eqref{stima_sc_der_seconde_G} we infer that
\begin{align*}
\sup_{x\in H}|\langle \nabla_y\nabla^G_kR_t[\Phi](x),h\rangle_H|
\leq & C_T\|\Phi_h\|_{\beta}\Lambda_1^{1-\beta}(t)\Lambda_2(t)|y|_H
\leq C_T\|\Phi\|_{\beta}\Lambda_1^{1-\beta}(t)\Lambda_2(t)|y|_H|h|_H|k|_U.
\end{align*} 
\end{proof}


\begin{thebibliography}{10}
%
%
%

\bibitem{av-las}
\textsc{Avalos G.} and \textsc{Lasiecka I.} (2003).
Optimal blowup rates for the minimal energy null control of the strongly damped abstract wave equation. \textit{Ann. Sc. Norm. Super. Pisa Cl. Sci.} \textbf{5} 601--616.



\bibitem{brez1}
\textsc{Brze\'zniak, Z.} and \textsc{Maslowski, B.} and \textsc{Seidler, J.} (2005). Stochastic nonlinear beam equations. \textit{Probab. Theory Related Fields} \textbf{132} 119--149.

\bibitem{brez2}
\textsc{Brze\'zniak, Z.} and \textsc{Ondreját, M} and \textsc{Seidler, J.} (2016). Invariant measures for stochastic nonlinear beam and wave equations. \textit{J. Differential Equations} \textbf{260} 4157--4179.


\bibitem{CeDaFl13}
\textsc{Cerrai S.} and \textsc{Da Prato G.} and \textsc{Flandoli F.} (2013).
{Pathwise uniqueness for stochastic reaction-diffusion equations in Banach spaces with an H\"{o}lder drift component}.
\textit{Stoch. Partial Differ. Equ. Anal. Comput.} \textbf{1} 507--551. 

\bibitem{chen-russ}
\textsc{Chen G.} and \textsc{Russell D. L.} (1981/1982).
A mathematical model for linear elastic systems with structural damping.
\textit{Quart. Appl. Math.} \textbf{39} 433--454.

\bibitem{chen-trig-1}
\textsc{Chen S. P.} and \textsc{Triggiani R.} (1988).
Proof of two conjectures by {G}. {C}hen and {D}. {L}. {R}ussell on
  structural damping for elastic systems.
\textit{Approximation and optimization ({H}avana, 1987)} 234--256.

\bibitem{chow}
\textsc{Chow, P.L.} and \textsc{ Menaldi, J.L.} (1999). 
Stochastic PDE for nonlinear vibration of elastic panels.
\textit{Differential Integral Equations} \textbf{12} 419--434.

\bibitem{dapr-fl}
\textsc{Da Prato G.} and \textsc{Flandoli F.} (2010).
Pathwise uniqueness for a class of {SDE} in {H}ilbert spaces and
  applications. \textit{J. Funct. Anal.} \textbf{259} 243--267.

\bibitem{DPFPR13}
\textsc{Da Prato G.} and \textsc{Flandoli F.} and \textsc{Priola E.} and \textsc{R\"{o}ckner M.} (2013). Strong uniqueness for stochastic evolution equations in {H}ilbert  spaces perturbed by a bounded measurable drift. \textit{Ann. Probab.} \textbf{41} 3306--3344.

\bibitem{DPFPR15}
\textsc{Da Prato G.} and \textsc{Flandoli F.} and \textsc{Priola E.} and \textsc{R\"{o}ckner M.} (2015). Strong uniqueness for stochastic evolution equations with unbounded measurable drift term.
\textit{J. Theoret. Probab.} \textbf{28} 1571--1600.

\bibitem{DPFRV16}
\textsc{Da Prato G.} and \textsc{Flandoli F.} and \textsc{R\"{o}ckner M.} and \textsc{Veretennikov A. Y.} (2016). 
Strong uniqueness for {SDE}s in {H}ilbert spaces with nonregular  drift.
\textit{Ann. Probab.} \textbf{44} 1985--2023.

\bibitem{DP3}
\textsc{Da Prato G.} and \textsc{Zabczyk J.} (2002).
\textit{Second order partial differential equations in {H}ilbert  spaces}, \textbf{293}. London Mathematical Society Note Series, Cambridge University Press, Cambridge.

 \bibitem {DPsecond}
\textsc{Da Prato G.} and \textsc{Zabczyk J.} (2014). 
\textit{Stochastic equations in infinite dimensions}, \textbf{152}. 2nd ed. 
Encyclopedia of Mathematics and and its Applications, Cambridge University Press, Cambridge.



\bibitem{FeFlandoli2013}
\textsc{Fedrizzi E.} and \textsc{Flandoli F.} (2013).
{Holder Flow and Differentiability for SDEs with Nonregular Drift}.
\textit{Stochastic Analysis and Applications } \textbf{31} 708--736.





\bibitem{Fl} 
\textsc{Flandoli F.} (2015). \textit{Random perturbation of PDEs andfluid dynamic models}, Saint Flour Summer School Lectures,  Lecture Notes in Math., Springer, Berlin.


\bibitem{fute}
\textsc{Fuhrman M.} and \textsc{Tessitore G.} (2002). Nonlinear {K}olmogorov equations in infinite dimensional spaces: the   backward stochastic differential equations.
\textit{Ann. Probab.} \textbf{30} 1397--1465.

\bibitem{Gua1}
\textsc{Guatteri G.} (2007). {On a Class of Forward-Backward Stochastic Differential Systems
in Infinite Dimensions}. \textit{J. of Appl. Math. and Stoc. Anal.} 33pp.

\bibitem{GuaTess}
\textsc{Guatteri G.} and \textsc{ Tessitore G.} (2005). {On the backward stochastic Riccati equation in infinite dimensions}. \textit{SIAM J. Control Optim.} \textbf{44} no. 1, 159-194. 




\bibitem{GP93}
\textsc{Gy\"{o}ngy I.} and \textsc{Pardoux \'{E}.} (1993). 
On the regularization effect of space-time white noise on  quasi-linear parabolic partial differential equations.
\textit{Probab. Theory Related Fields} \textbf{97} 211--229.

\bibitem{hairer} 
\textsc{Hairer M.} (2009).
\textit{An Introduction to Stochastic PDEs}, available at {arXiv:0907.4178v1.}  
 
 \bibitem{HuPeng} 
\textsc{Hu Y.} and \textsc{Peng S.} (1991). 
{Adapted solution of a backward semilinear stochastic evolution equation}. \textit{Stochastic Anal. Appl.} \textbf{9} 445--459.
  


\bibitem{lasi-trig}
\textsc{Lasiecka I.} and \textsc{Triggiani R.} (1998).
Exact null controllability of structurally damped and thermo-elastic
  parabolic models.
\textit{Atti Accad. Naz. Lincei Cl. Sci. Fis. Mat. Natur. Rend. Lincei (9) Mat. Appl.} \textbf{9} 43--69.

\bibitem{LR15}
\textsc{Liu W.} and \textsc{R\"ockner M} (2015).
\textit{Stochastic partial differential equations: an introduction}, Universitext. Springer, Cham. 

\bibitem{Lun}
\textsc{ Lunardi A.} (2009).
\textit{Interpolation Theory}, Second edition. Lecture Notes. Scuola Normale Superiore di Pisa (New Series), Edizioni della Normale, Pisa.

\bibitem{Mas05}
\textsc{Masiero F.} (2005).
Semilinear {K}olmogorov equations and applications to stochastic
  optimal control.
\textit{Appl. Math. Optim.}\textbf{51} 201--250.

\bibitem{Mas-Ban}
\textsc{Masiero F.} (2008).
Stochastic optimal control problems and parabolic equations in Banach spaces.
\textit{SIAM J. Control Optim.} \textbf{47} 251--300. 

\bibitem{MasPri17}
\textsc{Masiero F.} and \textsc{Priola E.} (2017).
Well-posedness of semilinear stochastic wave equations with
  {H}\"{o}lder continuous coefficients.
\textit{J. Differential Equations} \textbf{263} 1773--1812.

{\bibitem{MasPri23}
\textsc{Masiero F.} and \textsc{Priola E.}
Correction to
``Well-posedness of semilinear stochastic  wave equations with H\"{o}lder continuous    coefficients'', preprint arxiv 1607.00029v2, https://arxiv.org/pdf/1607.00029.pdf}

\bibitem{Ondre04}
\textsc{Ondrej\'{a}t M.} (2004).
Uniqueness for stochastic evolution equations in {B}anach spaces.
\textit{Dissertationes Math. (Rozprawy Mat.)} \textbf{426} 63pp.


\bibitem{trig}
\textsc{Triggiani R.} (2003).
Optimal estimates of norms of fast controls in exact null controllability of two non-classical abstract parabolic systems.
\textit{Adv. Differential Equations} \textbf{8} 189--229.


\bibitem{Ver} 
\textsc{Veretennikov A. J.} (1980).
{Strong solutions and explicit formulas for solutions of stochastic integral equations}.
\textit{Mat. Sb.} \textbf{111 (153)} 434--452.


\bibitem{Zab08}
\textsc{Zabczyk J.} (2008)
\textit{Mathematical Control Theory}, Second edition. Birkh\"auser.


\bibitem{Zv74} 
\textsc{Zvonkin A. K.} (1974).
{A transformation of the phase space of a diffusion process that will remove the drift}, 
\textit{Mat. Sb.} \textbf{93(135)} 129--149.

\end{thebibliography}
\end{document}